\documentclass{birkjour}
\usepackage{stackengine}
\stackMath
\newcommand\tsup[2][2]{%
 \def\useanchorwidth{T}%
  \ifnum#1>1%
    \stackon[-.5pt]{\tsup[\numexpr#1-1\relax]{#2}}{\scriptscriptstyle\sim}%
  \else%
    \stackon[.5pt]{#2}{\scriptscriptstyle\sim}%
  \fi%
}
 \theoremstyle{definition}
 
 \theoremstyle{remark}

 \numberwithin{equation}{section}

\usepackage{verbatim}
\usepackage{xcolor}
\usepackage{graphicx}
\usepackage{amssymb, amsmath, amsthm}
\usepackage{amsfonts}
\usepackage{amssymb}
\usepackage{float}
\usepackage{hyperref}
\usepackage{mathrsfs}
\usepackage{enumitem}

\newtheorem{theorem}{Theorem}

\newtheorem{corollary}[theorem]{Corollary}

\newtheorem{definition}[theorem]{Definition}

\newtheorem{lemma}[theorem]{Lemma}

\newtheorem{proposition}[theorem]{Proposition}
\newtheorem{remark}[theorem]{Remark}

\newcommand{\dis}{\displaystyle}
\newcommand{\thetanu}{\theta^{(\nu)}}
\newcommand{\unu}{u^{(\nu)}}

\newcommand{\divv}{\text{\rm div}}

\newcommand{\bK}{\bar{K}_{\infty}}
\newcommand{\Dh}{\delta_{h}}
\newcommand{\dt}{\partial_t}
\newcommand{\1}{\mathbf{1}}
\newcommand{\dimf}{\text{\rm dim}_{f}}
\newcommand{\A}{{\mathcal A}}
\newcommand{\An}{{\mathcal A}^{\nu}}
\newcommand{\Gg}{\mathcal{G}^{\nu}}
\newcommand{\Ggo}{\mathcal{G}^{\nu_0}}

\newcommand{\Lh}{\mathcal{L}_{h}}
\newcommand{\K}{K_{\infty}}

\newcommand{\N}{\mathbb N}
\newcommand{\Tr}{\text{\rm Tr}}
\newcommand{\T}{\mathbb T}
\newcommand{\Z}{\mathbb Z}

\newcommand{\pin}{\pi^{\nu}}
\newcommand{\pino}{\tilde\pi^{\nu}}
\newcommand{\pio}{\tilde\pi^{0}}

\newcommand{\I}{I^*}
\newcommand{\U}{\mathcal U}

\newcommand{\R}{\mathbb R}

\newcommand{\dist}{\text{\rm dist}}

\newcommand{\intox}{\int_{\mathbb{T}^d}}
\renewcommand{\theequation}{\thesection.\arabic{equation}}
\numberwithin{equation}{section}
\numberwithin{theorem}{section}
\numberwithin{figure}{section}

\begin{document}

%-------------------------------------------------------------------------
% editorial commands: to be inserted by the editorial office
%
%\firstpage{1} \volume{228} \Copyrightyear{2004} \DOI{003-0001}
%
%
%\seriesextra{Just an add-on}
%\seriesextraline{This is the Concrete Title of this Book\br H.E. R and S.T.C. W, Eds.}
%
% for journals:
%
%\firstpage{1}
%\issuenumber{1}
%\Volumeandyear{1 (2004)}
%\Copyrightyear{2004}
%\DOI{003-xxxx-y}
%\Signet
%\commby{inhouse}
%\submitted{March 14, 2003}
%\received{March 16, 2000}
%\revised{June 1, 2000}
%\accepted{July 22, 2000}
%
%
%
%---------------------------------------------------------------------------
%Insert here the title, affiliations and abstract:
%

\title[Vanishing diffusion limits]
 {Vanishing diffusion limits and long time behaviour of a class of forced active scalar equations}

 %----------Author 1
\author{Susan Friedlander}

\address{Department of Mathematics\\
University of Southern California}

\email{susanfri@usc.edu}

%\thanks{This work was completed with the support of our
%\TeX-pert.}
%----------Author 2
\author{Anthony Suen}

\address{Department of Mathematics and Information Technology\\
The Education University of Hong Kong}

\email{acksuen@eduhk.hk}

\date{}



\begin{abstract}
We investigate the properties of an abstract family of advection diffusion equations in the context
of the fractional Laplacian. Two independent diffusion parameters enter the system,
one via the constitutive law for the drift velocity and one as the prefactor of the 
fractional Laplacian. We obtain existence and convergence results in certain
parameter regimes and limits. We study the long time behaviour of solutions to the general
problem and prove the existence of a unique global attractor. We apply results to two
particular active scalar equations arising in geophysical fluid dynamics, namely the
surface quasigeostrophic equation and the magnetogeostrophic equation.
\end{abstract}

%%% ----------------------------------------------------------------------
\maketitle
%%% ----------------------------------------------------------------------
\tableofcontents
\section{Introduction}\label{introduction}

Active scalar equations have been a topic of considerable study in recent years, in part because they arise in many physical models and in part because they present challenging nonlinear PDEs. In particular, such equations are prevalent in mathematical fluid dynamics. One such equation is the surface quasi-geostrophic equation (SQG) which was introduced by Constantin, Majda and Tabak as a 2 dimensional analogue for the three dimensional Euler equations \cite{CMT94}, \cite{HPG95}, \cite{OY97}. Another model with related, but distinct, features is the magnetogeostrophic equation (MG) which was proposed by Moffatt and Loper as a model for magnetogeostrophic turbulence \cite{FV11a}, \cite{FRV14}, \cite{M08}, \cite{ML94}. The physics of an active scalar equation is encoded in the constitutive law that relates the transport velocity vector $u$ with a scalar field $\theta$. This law produces a differential operator that when applied to the scalar field determines the velocity. The singular or smoothing properties of this operator are closely connected with the mathematics of the nonlinear advection equation for $\theta$. In this present paper we study an abstract class of active scalar equations in $\mathbb{T}^d\times(0,\infty)=[0,2\pi]^d\times(0,\infty)$ with $d\in\{2,3\}$ \footnote{We point out that most of the results given in our work hold for $d\ge2$.} of the following form
\begin{align}
\label{abstract active scalar eqn} \left\{ \begin{array}{l}
\partial_t\theta+u\cdot\nabla\theta=-\kappa\Lambda^{\gamma}\theta+S, \\
u_j[\theta]=\partial_{x_i} T_{ij}^{\nu}[\theta],\theta(x,0)=\theta_0(x)
\end{array}\right.
\end{align}
where $\nu\ge0$, $\kappa\ge0$, $\gamma\in(0,2]$ and $\Lambda:=\sqrt{-\Delta}$. Here $\theta_0$ is the initial datum and $S=S(x)$ is a given function that represents the forcing of the system. We assume that \footnote{Such mean zero assumption is common in many physical models which include SQG equation and MG equation; see \cite{CTV14} and \cite{FS18} for example.}
\begin{equation}\label{zero mean assumption on data and forcing}
\int_{\mathbb{T}^d}\theta_0(x)dx=\int_{\mathbb{T}^d}S(x)=0,
\end{equation}
and throughout this paper, we consider mean-zero (zero average) solutions. $\{T_{ij}^{\nu}\}_{\nu\ge0}$ is a sequence of operators which satisfy:
\begin{enumerate}
\item[A1]  $\partial_{x_i}\partial_{x_j} T^{\nu}_{ij}f=0$ for any smooth functions $f$ for all $\nu\ge0$.
\item[A2] $T_{ij}^\nu:L^\infty(\mathbb{T}^d)\rightarrow BMO(\mathbb{T}^d)$ are bounded uniformly in $\nu$ for all $\nu\ge0$.
\item[A3] For each $\nu>0$, there exists a constant $C_{\nu}>0$ such that for all $1\le i,j\le d$, $$|\widehat{T^{\nu}_{ij}}(k)|\le C_{\nu}|k|^{-3}, \qquad\forall k\in\mathbb{Z}^d\backslash\{|k|=0\}.$$
\item[A4] For each $1\le i,j\le d$ and $\nu\ge0$, $\widehat{T^{\nu}_{ij}}(k)=0$ for $|k|=0$.
\item[A5] There exists a constant $C_{0}>0$ independent of $\nu$, such that for all $1\le i,j\le d$, 
\begin{equation*}
\sup_{\nu\in[0,1]}\sup_{\{k\in\mathbb{Z}^d_*\}}|\widehat{T^{\nu}_{ij}}(k)|\le C_{0}.
\end{equation*}
\end{enumerate}
\begin{remark}
There are several remarks for the assumptions A1 to A5 as given above:
\begin{itemize}
\item A1 implies that $u=u[\theta]$ is divergence-free for all $\nu\ge0$. Hence together with \eqref{zero mean assumption on data and forcing}, it immediately implies that $\theta$ obeys
\begin{equation}\label{zero mean assumption}
\int_{\mathbb{T}^d}\theta(x,t)dx=0,\qquad\forall t\ge0. 
\end{equation}
\item A2 implies that the drift velocity $u$ lies in the space $L^{\infty}_t BMO^{-1}_x$ for all $\nu\ge0$.
\item A3 implies that $u_j[\cdot]=\partial_{x_i} T_{ij}^{\nu}[\cdot]$ are operators of smoothing order 2 for $\nu>0$, in the sense that for any $s\ge0$ and $f\in L^p$ with $p>1$,
\begin{align}\label{two order smoothing for u when nu>0}
\|\Lambda^s u[f]\|_{L^p}\le C_{\nu}\|\Lambda^{s'}f\|_{L^p},
\end{align}
where $s'=\max\{s-2,0\}$. Here $C_{\nu}$ is a positive constant which depends on $\nu$, $p$ and $d$ only, and $C_\nu$ may blow up as $\nu\to0$.
\item A4 implies that $u$ has zero mean, which is consistent with $\theta$ also having zero mean.
\item A5 implies that $u_j[\cdot]=\partial_{x_i} T_{ij}^{0}[\cdot]$ is a singular operator of order 1, in the sense that for any $f\in L^p$ with $p>1$,
\begin{align}\label{one order singular for u when nu=0}
\|u[f]\|_{L^p}\le C_{0}\|\nabla f\|_{L^p},
\end{align}
for some positive constant $C_{0}$ which depends on $p$ and $d$ only.
\end{itemize}
\end{remark}

The abstract family of active scalar equations \eqref{abstract active scalar eqn} satisfying the properties A1--A5 include as special cases the SQG equation and the MG equation, both of which model phenomena in rotating fluids. The critical SQG equation is an example where the dimension $d=2$, the diffusive parameter $\nu=0$, the thermal diffusion $\kappa>0$, the fractional power $\gamma=1$, and the relation between the velocity $u$ and the scalar field $\theta$ is given by the perpendicular Riesz transform. We note that this is a singular integral operator of degree zero. Global well-posedness for the critical SQG equation was first proved in Kiselev, Nazarov and Volberg \cite{KFV07} and Caffarelli and Vasseur \cite{CV10}. More recently there has been a considerable literature on the long time dynamics of the forced SQG equation including \cite{CD14}, \cite{CTV14}, \cite{CZV16} and references there in.

The MG equation is an example where the dimension $d=3$, the diffusive parameter $\nu\ge0$, the thermal diffusion $\kappa\ge0$, and the fractional power $\gamma=2$. The derivation of the MG equation via the postulates in \cite{ML94} reduces the MHD system to an active scalar equation
\begin{align}\label{active scalar equation general MG intro}
\partial_t \theta + u \cdot \nabla \theta = \kappa \Delta \theta + S
\end{align}
where the constitutive law is obtained from the linear system
\begin{align}
e_3 \times u &= - \nabla P + e_2 \cdot \nabla b + \theta e_3 + \nu \Delta u,\label{MG equation linear system 1 intro}\\
0 &= e_2 \cdot \nabla u + \Delta b, \label{MG equation linear system 2 intro}\\
\nabla \cdot u &= 0, \nabla \cdot b = 0. \label{MG equation linear system 3 intro}
\end{align}
This system encodes the vestiges of the physics in the problem, namely the Coriolis force, the Lorentz force and gravity. Vector manipulations of \eqref{MG equation linear system 1 intro}-\eqref{MG equation linear system 3 intro} give the expression
\begin{align}\label{Vector manipulations intro}
\{[\nu \Delta^2 - (e_2 \cdot \nabla)^2]^2 + (e_3 \cdot \nabla)^2 \Delta \} u &= - [\nu \Delta^2 - (e_2 \cdot \nabla)^2] \nabla \times (e_3 \times \nabla \theta)\notag\\
&\qquad + (e_3 \cdot \nabla)\Delta (e_3 \times \nabla \theta).
\end{align}
Here $(e_1, e_2, e_3)$ denote Cartesian unit vectors. The explicit expression for the components of the Fourier multiplier symbol $\widehat{M}^\nu$ as functions of the Fourier variable $k = (k_1, k_2, k_3) \in \mathbb{Z}^3$ with $k_3\neq0$ are obtained from the constitutive law \eqref{Vector manipulations intro} to give
\begin{align}
\widehat M^{\nu}_1(k)&=[k_2k_3|k|^2-k_1k_3(k_2^2+\nu|k|^4)]D(k)^{-1},\label{MG Fourier symbol_1 intro}\\
\widehat M^{\nu}_2(k)&=[-k_1k_3|k|^2-k_2k_3(k_2^2+\nu|k|^4)]D(k)^{-1},\label{MG Fourier symbol_2 intro}\\
\widehat M^{\nu}_3(k)&=[(k_1^2+k_2^2)(k_2^2+\nu|k|^4)]D(k)^{-1},\label{MG Fourier symbol_3 intro}
\end{align}
where
\begin{align}
D(k)=|k|^2k_3^2+(k_2^2+\nu|k|^4)^2.\label{MG Fourier symbol_4 intro}
\end{align}

In the magnetostrophic turbulence model the parameters $\nu$, the nondimensional viscosity, and $\kappa$, the nondimensional thermal diffusivity, are extremely small. The behaviour of the MG equation is dramatically different when the parameters $\nu$ and $\kappa$ are present (i.e. positive) or absent (i.e. zero). The limit as either or both parameters vanish is highly singular. Since both parameters multiply a Laplacian term, their presence is smoothing. However $\kappa$ enters \eqref{active scalar equation general MG intro} in a parabolic heat equation role whereas $\nu$ enters via the constitutive law \eqref{Vector manipulations intro}.  The mathematical properties of the MG equation have been determined in various settings of the parameters via an analysis of the Fourier multiplier symbol $\widehat{M}^\nu$ given by \eqref{MG Fourier symbol_1 intro}-\eqref{MG Fourier symbol_4 intro}. When $\nu = 0$ the relation between $u$ and $\theta$ is given by a {\it singular} operator of order 1. The implications of this fact for the inviscid MG$^0$ equation are summarized in the survey article by Friedlander, Rusin and Vicol \cite{FRV14}. In particular, when $\kappa > 0$ the inviscid but thermally dissipative MG$^0$ equation is globally well-possed \cite{FV11a} \footnote{$L^\infty$ is the critical Lebesgue space with respect to the natural scaling for both the critically diffusive SQG and MG$^0$ equations.}. In contrast when $\nu = 0$ {\it and} $\kappa = 0$, the singular inviscid MG$^0$ equation is {\it ill-possed} in the sense of Hadamard in any Sobolev space \cite{FV11b}. In a recent paper Friedlander and Suen \cite{FS18} examine the limit of vanishing viscosity in the case when $\kappa > 0$. They prove global existence of classical solutions to the forced MG$^\nu$ equations and obtain strong convergence of solutions as the viscosity $\nu$ vanishes. 

The purpose of our current paper is to investigate properties of the abstract system \eqref{abstract active scalar eqn} under assumptions A1--A5, with emphasis on convergence results in the context of the fractional Laplacian. In Section~\ref{Existence and convergence of Hs-solutions section} the parameter $\nu$ is taken to be positive and the ensuing smoothing properties of $T_{ij}^{\nu}$ permits existence and convergence in Sobolev space $H^s$ as $\kappa$ goes to zero. In contrast, when the parameter $\nu$ is set to zero, A5 implies that $\partial_{x_i} T_{ij}^{\nu}$ is a singular operator. In this case the existence and convergence results proved in Section~\ref{Existence and convergence of Gevrey-class solutions section} are restricted to analytic and Gevrey-class solutions. We note that the results in Section~\ref{Existence and convergence of Gevrey-class solutions section} are valid in the example of the critical SQG equation. In Section~\ref{long time behaviour and attractors section} we study the long time behaviour of solutions to the abstract system when both $\nu>0$ and $\kappa>0$. We prove that the solution map associated with \eqref{abstract active scalar eqn} possesses a unique global attractor $\Gg$ in $H^1$ for all $\nu>0$. With the restriction that $\gamma$ lies in $[1,2]$, we prove that $\Gg$ has finite fractal dimension. In Section~\ref{Applications to magneto-geostrophic equations section} we apply the results proved in Section~\ref{long time behaviour and attractors section} to the specific example of the MG equation. We prove convergence of $\Gg$ as $\nu$ goes to zero to the global attractor $\A$ in $L^2$ for the MG$^0$ equation whose existence was demonstrated in \cite{FS18}. It is of interest to note that although we prove that $\Gg$ has finite fractal dimension for all $\nu > 0$, it is unknown whether or not $\A$ has finite fractal dimension.

\section{Main results}\label{main results}

The main results that we prove for the forced problem \eqref{abstract active scalar eqn} are stated in the following theorems, and they will be proved in Section~\ref{Existence and convergence of Hs-solutions section}, Section~\ref{Existence and convergence of Gevrey-class solutions section} and Section~\ref{long time behaviour and attractors section}. These results will then be applied to the magnetogeostrophic (MG) active scalar equation which will be discussed in Section~\ref{Applications to magneto-geostrophic equations section}.

\begin{theorem}[$H^s$-convergence as $\kappa\rightarrow0$ when $\nu>0$]\label{Hs convergence}
Let $\nu>0$ and $\gamma\in(0,2]$ be given in \eqref{abstract active scalar eqn}, and let $\theta_0,S\in C^\infty$ be the initial datum and forcing term respectively which satisfy \eqref{zero mean assumption on data and forcing}. If $\theta^\kappa$ and $\theta^0$ are smooth solutions to \eqref{abstract active scalar eqn} for $\kappa>0$ and $\kappa=0$ respectively, then 
\begin{align}\label{Hs convergence thm} 
\lim_{\kappa\rightarrow0}\|(\theta^\kappa-\theta^0)(\cdot,t)\|_{H^s}=0,
\end{align}
for all $s\ge0$ and $t\ge0$.
\end{theorem}
\begin{theorem}[Analytic convergence as $\kappa\rightarrow0$ when $\nu=0$]\label{analytic convergence kappa}
Let $\nu=0$ and $\gamma\in(0,2]$ be given in \eqref{abstract active scalar eqn}, and let $\theta_0,S$ be the initial datum and forcing term respectively which satisfy \eqref{zero mean assumption on data and forcing}. Suppose that $\theta_0$ and $S$ are both analytic functions. Then if $\theta^{\kappa}$, $\theta^{0}$ are analytic solutions to \eqref{abstract active scalar eqn} for $\kappa>0$ and $\kappa=0$ respectively with initial datum $\theta_0$ on $\mathbb{T}^d\times[0,\bar{T}]$ with radius of convergence at least $\bar{\tau}$, then there exists $T\le\bar{T}$ and $\tau=\tau(t)<\bar{\tau}$ such that, for $t\in[0,T]$, we have:
\begin{align}\label{analytic convergence}
\lim_{\kappa\to0}\|(\Lambda^re^{\tau\Lambda}\theta^{\kappa}-\Lambda^re^{\tau\Lambda}\theta^{0})(\cdot,t)\|_{L^2}=0,
\end{align}
where $\Lambda:=(-\Delta)^\frac{1}{2}$ and $r>\frac{d}{2}+\frac{5}{2}$ is the Sobolev exponent.
\end{theorem} 
\begin{theorem}[Existence of global attractors]\label{existence of global attractor theorem}
Let $S\in L^\infty\cap H^1$ be the forcing term. For $\nu$, $\kappa>0$ and $\gamma\in(0,2]$, let $\pin(t)$ be the solution operator for the initial value problem \eqref{abstract active scalar eqn} via
\begin{align*}
\pin(t): H^1\to H^1,\qquad \pin(t)\theta_0=\theta(\cdot,t),\qquad t\ge0.
\end{align*}
Then the solution map $\pin(t):H^1\to H^1$ associated to \eqref{abstract active scalar eqn} possesses a unique global attractor $\Gg$ for all $\nu>0$. In particular, if we assume that $\gamma\in[1,2]$, then for all $\nu>0$, the global attractor $\Gg$ of $\pin(t)$ further enjoys the following properties:
\begin{itemize}
\item $\Gg$ is fully invariant, namely
\begin{align*}
\pin(t)\Gg=\Gg,\qquad \forall t\ge0.
\end{align*}
\item $\Gg$ is maximal in the class of $H^1$-bounded invariant sets.
\item $\Gg$ has finite fractal dimension.
\end{itemize}
\end{theorem}

\begin{remark}
There are several remarks for the main results as given above:
\begin{itemize}
\item The existence of global-in-time $H^s$-solutions to \eqref{abstract active scalar eqn} when $\nu>0$ will be given by Theorem~\ref{global-in-time wellposedness in Sobolev}.
\item The local-in-time existence of analytic solutions to \eqref{abstract active scalar eqn} when $\nu=0$ will be given by Theorem~\ref{Local-in-time existence of analytic solutions thm}. In particular, under a smallness assumption on the initial data in terms of {\it Sobolev} norm, the analytic solutions as claimed by Theorem~\ref{Local-in-time existence of analytic solutions thm} can be extended globally in time; refer to Theorem~\ref{Global-in-time existence of analytic solutions thm} for more details. 
\item Under a stronger assumption on the Fourier symbols $\widehat{T^{\nu}_{ij}}(k)$, one can extend Theorem~\ref{analytic convergence kappa} to Gevrey-class solutions when the initial datum $\theta_0$ and forcing term $S$ are both Gevrey-class functions; refer to Theorem~\ref{Existence and convergence of Gevrey-class solutions thm} for more details.
\item The results claimed above can be applied to physical models which include MG equation and SQG equation; refer to Section~\ref{Existence and convergence of Gevrey-class solutions section} and Section~\ref{Applications to magneto-geostrophic equations section} for related discussions.
\end{itemize}
\end{remark}

\section{Preliminaries}\label{preliminaries}

We introduce the following notations and conventions:

\begin{itemize}
\item We write $u:=u[\theta]$, where $u$ is given by $u_j:=\partial_{x_i} T_{ij}^{\nu}[\theta]$. We also refer $u[\cdot]$ to an operator in the sense that $u_j[f]=\partial_{x_i} T_{ij}^{\nu}[f]$ for appropriate functions $f$.
\item To emphasise the dependence of solutions on $\nu$ and $\kappa$, we sometimes write $\theta=\theta^\kappa$ and $u=u^{\kappa}:=u[\theta^\kappa]$ for varying $\kappa$, while we write $\theta=\theta^{(\nu)}$ and $u^{(\nu)}[\cdot]:=\partial_{x_i} T_{ij}^{\nu}[\cdot]$ for varying $\nu$.
\item $W^{s,p}$ is the usual inhomogeneous Sobolev space with norm $\|\cdot\|_{W^{s,p}(\mathbb{T}^d)}$, and we write $H^s=W^{s,2}$. For simplicity, we write $\|\cdot\|_{L^p}=\|\cdot\|_{L^p(\mathbb{T}^d)}$, $\|\cdot\|_{W^{s,p}}=\|\cdot\|_{W^{s,p}(\mathbb{T}^d)}$, etc. unless otherwise specified. 
\item We define $\langle\cdot,\cdot\rangle$ to be the $L^2$-inner product on $\mathbb{T}^d$, that is
\begin{equation*}
\langle f,g\rangle:=\int_{\mathbb{T}^d}fgdx,
\end{equation*}
for any $f,g\in L^2$.
\item Regarding the constants used in this work, we have the following conventions:
\begin{itemize}
\item $C$ shall denote a positive and sufficiently large constant, whose value may change from line to line; 
\item $C$ is allowed to depend on the size of $\T^d$ and other universal constants which are fixed throughout this work; 
\item In order to emphasise the dependence of $C$ on a certain quantity $Q$ we usually write $C_{Q}$ or $C(Q)$.
\end{itemize}
\end{itemize}

We recall the following Sobolev embedding inequalities from the literature (see for example Bahouri-Chemin-Danchin \cite{BCD11} and Ziemer \cite{Z89}): 
\begin{itemize}
\item Let $d\ge2$ be the dimension. There exists $C=C(d)>0$ such that
\begin{equation}
\|f\|_{L^\infty }\le C \|f\|_{W^{2,d} }.\label{L infty bound 1}
\end{equation}
\item For $q>d$, there exists $C=C(q)>0$ such that
\begin{equation}
\|f\|_{L^\infty }\le C \|f\|_{W^{1,q} }.\label{L infty bound 2}
\end{equation}
\item If $k>l$ and $k-\frac{d}{p}>l-\frac{d}{q}$, then there exists $C=C(k,l,d,p,q)>0$ such that
\begin{align}
\|f\|_{W^{l,q}}\le C\|f\|_{W^{k,p}}.\label{Kondrachov embedding theorem}
\end{align}
\item For $q>1$, $q'\in[q,\infty)$ and $\frac{1}{q'}=\frac{1}{q}-\frac{s}{d}$, if $\Lambda^s h\in L^q$, there exists $C=C(q,q',d,s)>0$ such that
\begin{align}\label{Sobolev inequality}
\|h\|_{L^{q'}}\le C\|\Lambda^s h\|_{L^{q}}.
\end{align}
\end{itemize}

We also recall the following product and commutator estimates: If $s>0$ and $p>1$, then for all $f,g\in H^s\cap L^\infty$, we have
\begin{align}\label{product estimate}
\|\Lambda^s(fg)\|_{L^p}\le C\Big(\|f\|_{L^{p_1}}\|\Lambda^s g\|_{L^{p_2}}+\|\Lambda^s f\|_{L^{p_3}}\|g\|_{L^{p_4}}\Big),
\end{align}
\begin{align}\label{commutator estimate}
\|\Lambda^s(fg)-f\Lambda^s(g)\|_{L^p}\le C\Big(\|\nabla f\|_{L^{p_1}}\|\Lambda^{s-1} g\|_{L^{p_2}}+\|\Lambda^s f\|_{L^{p_3}}\|g\|_{L^{p_4}}\Big),
\end{align}
where $\frac{1}{p}=\frac{1}{p_1}+\frac{1}{p_2}=\frac{1}{p_3}+\frac{1}{p_4}$ and $p$, $p_2$, $p_3\in(1,\infty)$. 

As in \cite{FT89}, \cite{LO97} and \cite{MV11}, for $s\ge1$, the Gevrey-class $s$ is defined by
\begin{align*}
\bigcup_{\tau>0}\mathcal{D}(\Lambda^re^{\tau\Lambda^\frac{1}{s}}),
\end{align*}
where for any $r\ge0$,
\begin{align*}
\mathcal{D}(\Lambda^re^{\tau\Lambda^\frac{1}{s}}):=\{f\in H^r:\|\Lambda^re^{\tau\Lambda^\frac{1}{s}}f\|_{L^2}<\infty\},
\end{align*}
and we use the Gevrey-class real-analytic norm given by
\begin{align*}
\|\Lambda^re^{\tau\Lambda^\frac{1}{s}}f\|^2_{L^2}:=\sum_{k\in\mathbb{Z}^d_*}|k|^{2r}e^{2\tau|k|^\frac{1}{s}}|\hat{f}(k)|^2,
\end{align*}
where $\tau=\tau(t)>0$ denotes the radius of convergence. When $s = 1$, we recover the class of real-analytic functions with radius of analyticity $\tau$. When $s > 1$, the Gevrey-classes consist of $C^\infty$-functions which are in general not analytic. In view of the mean-zero assumption \eqref{zero mean assumption on data and forcing}, it makes sense to take $\mathbb{Z}^d_*=\mathbb{Z}^d\backslash\{|k|=0\}$ in the definition of Gevrey-class norm.

\section{Existence and convergence of $H^s$-solutions when $\nu>0$}\label{Existence and convergence of Hs-solutions section}

In this section, we first obtain the global-in-time existence of $H^s$-solutions to \eqref{abstract active scalar eqn} when $\nu>0$ by applying the De Giorgi iteration method. We then prove $H^s$-convergence of the solutions $\theta$ to the active scalar equation \eqref{abstract active scalar eqn} as $\kappa\rightarrow0$. More precisely, given $\nu>0$ and $\gamma\in(0,2]$, we prove that if $\theta^\kappa$ and $\theta^0$ are smooth solutions to \eqref{abstract active scalar eqn} for $\kappa>0$ and $\kappa=0$ respectively, then $\|(\theta^\kappa-\theta^0)(\cdot,t)\|_{H^s}\rightarrow0$ as $\kappa\rightarrow0$ for all $t>0$ and $s\ge0$. 

\subsection{Existence of $H^s$-solutions}

In this subsection, we prove the following theorem which gives the desired global-in-time wellposedness for \eqref{abstract active scalar eqn} in Sobolev space $H^s$ when $\kappa\ge0$ and $\nu>0$. 

\begin{theorem}[Global-in-time wellposedness in Sobolev space]\label{global-in-time wellposedness in Sobolev}
Fix $\nu>0$, $\gamma\in(0,2]$ and $s\ge0$, and let $\theta_0\in H^s$ and $S\in H^s\cap L^\infty$ be given. 
\begin{itemize}
\item For any $\kappa>0$, there exists a global-in-time solution to \eqref{abstract active scalar eqn} such that 
\begin{align}\label{functional spaces for theta}
\theta^\kappa\in C([0,\infty);H^s)\cap L^2([0,\infty);H^{s+\frac{\gamma}{2}}).
\end{align}
\item For $\kappa=0$, if we further assume that $\theta_0\in L^\infty$, then there exists a global-in-time solution to \eqref{abstract active scalar eqn} such that $\theta^0(\cdot,t)\in H^s$ for all $t\ge0$.
\end{itemize}
\end{theorem}
In view of the case when $\kappa>0$, the most subtle part for proving Theorem~\ref{global-in-time wellposedness in Sobolev} is to estimate the $L^\infty$-norm of $\theta^\kappa(\cdot,t)$ when $\theta_0$ is {\it not} necessarily in $L^\infty$. In achieving our goal, we apply De Giorgi iteration method which will be illustrated in Lemma~\ref{De Giorgi iteration lemma}. 

We first recall the following energy inequalities which were established in \cite{CTV14} for the case $\gamma=1$. The general cases for $\gamma\in(0,2]$ follows similarly and we omit the details here.

\begin{proposition}\label{boundedness on theta with kappa>0 prop}
Assume that $S\in L^2\cap L^\infty$, and let $\theta^\kappa$ be a smooth solution to \eqref{abstract active scalar eqn}. For $\kappa>0$, we have
\begin{align}\label{energy inequality kappa>0}
\|\theta^\kappa(\cdot,t)\|^2_{L^2}+\kappa\int_0^t\|\Lambda^\frac{\gamma}{2}\theta^\kappa(\cdot,\tau)\|^2_{L^2}d\tau\le\|\theta_0\|^2_{L^2}+\frac{t}{c_0\kappa}\|S\|^2_{L^2}, \qquad\forall t\ge0.
\end{align}
If we further assume that $\theta_0\in L^\infty$, for all $\kappa>0$, it gives
\begin{align}\label{infty bound on theta with kappa>0}
\|\theta^\kappa(\cdot,t)\|_{L^\infty}\le\|\theta_0\|_{L^\infty}e^{-c_o\kappa t}+\frac{\|S\|_{L^\infty}}{c_0\kappa}, \qquad\forall t\ge0,
\end{align}
where $c_0>0$ is a universal constant which depends only on the dimension $d$.
\end{proposition}

\begin{remark}
Instead of the decay estimate given in \eqref{infty bound on theta with kappa>0}, we also have the following bound on $\|\theta^\kappa(\cdot,t)\|_{L^\infty}$ for all $\kappa\ge0$ provided that $\theta_0\in L^\infty$:
\begin{align}\label{infty bound on theta with kappa=0 or >0}
\|\theta^\kappa(\cdot,t)\|_{L^\infty}\le \|\theta_0\|_{L^\infty}+\|S\|_{L^\infty},\qquad\forall t\ge0.
\end{align} 
\end{remark}

Next we state and prove the following lemma which gives local-in-time existence results for the equation \eqref{abstract active scalar eqn}. 

\begin{lemma}\label{local-in-time existence lemma}
Let $\nu>0$, $\kappa>0$, $\gamma\in(0,2]$ and fix $s \ge 0$. Assume that $\theta_0\in H^s$ and $S\in H^s\cap L^\infty$, then there exists $T_*=T_*(\theta_0,S)>0$ and a unique solution $\theta^\kappa$ of \eqref{abstract active scalar eqn} such that
\begin{align*}
\theta^\kappa\in C([0,T_*);H^s)\cap L^2([0,T_*);H^{s+\frac{\gamma}{2}}).
\end{align*}
When $\kappa=0$ and $s \ge 0$, if $\theta_0$, $S\in H^s\cap L^\infty$, then there exists $T_*=T_*(\theta_0,S)>0$ and a unique solution $\theta^0$ of \eqref{abstract active scalar eqn} such that $\theta^0(\cdot,t)\in H^s$ for all $t\in[0,T_*)$.
\end{lemma}

\begin{proof}
We consider only the case for $\kappa >0$. We note that the case for $\kappa=0$ was addressed in \cite{FS19}. Following the argument given in the proof of \cite[Theorem~3.1]{FRV12}, applying \eqref{two order smoothing for u when nu>0} for $\nu>0$ and using standard energy method, there exists $T_*=T_*(\theta_0,S)>0$ and a unique solution $\theta^\kappa$ of \eqref{abstract active scalar eqn} such that
\begin{align*}
\theta^\kappa\in L^\infty([0,T_*);H^s)\cap L^2([0,T_*);H^{s+\frac{\gamma}{2}}),
\end{align*}
and in particular, the following inequality holds for $t\in[0,T_*)$:
\begin{align}\label{H^s bound on theta with S and higher norm}
\|\theta^\kappa(\cdot,t)\|^2_{H^s}\le \|\theta_0\|^2_{H^s}+Ct\|S\|^2_{H^s}+C\int_0^t\|\theta^\kappa(\cdot,\tau)\|^2_{H^{s+\frac{\gamma}{2}}}d\tau,
\end{align}
where $C$ is a positive constant which is independent of $t$. It then remains to prove the continuity of $\theta^\kappa$ in time. By the standard bootstrap argument, it suffices to show that $\theta^\kappa(\cdot,t)\to\theta_0$ in $H^s$ as $t\to0^+$. Using the fact that $\theta^\kappa\in L^\infty([0,T_*);H^s)$, we apply Lemma~1.4 in \cite[Chapter 3]{T01} to show that $\theta^\kappa\in C_{w}([0,T_*);H^s)$. Therefore, using \eqref{H^s bound on theta with S and higher norm}, we obtain
\begin{align*}
&\|\theta^\kappa(\cdot,t)-\theta_0\|^2_{H^s}\\
&\le2\|\theta_0\|^2_{H^s}-2\langle\theta^\kappa(\cdot,t),\theta_0\rangle_{H^{s}}+Ct\|S\|^2_{H^s}+C\int_0^t\|\theta^\kappa(\cdot,\tau)\|^2_{H^{s+\frac{\gamma}{2}}}d\tau,
\end{align*}
where $\langle\cdot,\cdot\rangle_{H^s}$ is the inner product on $H^s$. As $t\to0^+$, one can show that $\|\theta^\kappa(\cdot,t)-\theta_0\|_{H^s}\to0$ which follows by the argument given in \cite{M06} and we omit the details here.
\end{proof}

The lemma below gives the desired bound on $\|\theta^\kappa(\cdot,t)\|_{L^\infty}$ for the case when the time $t$ is small, which will be used for proving Theorem~\ref{global-in-time wellposedness in Sobolev}. The idea follows from the work given in \cite{CV10} and \cite{CZV16}.
\begin{lemma}[From $L^2$ to $L^\infty$]\label{De Giorgi iteration lemma}
Fix $\nu>0$, $\kappa>0$, $s\ge0$ and $\gamma\in(0,2]$. Let $\theta(t)$ be the solution to \eqref{abstract active scalar eqn} with initial datum $\theta_0\in H^s$. Then for all $t\in(0,1]$, we have
\begin{align}\label{infty bound on theta in terms of L2 norm of initial theta}
\|\theta(t)\|_{L^\infty}\le C\Big[\Big(\frac{2}{t}+1\Big)^{\frac{d+1-\gamma}{2\gamma}}\Big(\|\theta_0\|_{L^2}+\frac{\|S\|_{L^2}}{c_0^\frac{1}{2}\kappa^\frac{1}{2}}\Big)+\|S\|_{L^\infty}\Big],
\end{align}
where $C=C(d)>0$ is a constant which only depends on the dimension $d$.
\end{lemma}

\begin{remark}
By exploiting the bounds \eqref{infty bound on theta with kappa>0} and \eqref{infty bound on theta in terms of L2 norm of initial theta} on $\|\theta\|_{L^\infty}$, for $\theta_0\in H^s$ with $s\ge0$, we have the following bound for $t\ge1$ (see also \cite[Theorem~3.2]{CZV16}):
\begin{align}\label{infty bound on theta with kappa>0 and t>1}
\|\theta(t)\|_{L^\infty}\le\frac{C}{\kappa}\left[\|\theta_0\|_{L^2}+\frac{\|S\|_{L^2}}{\kappa^\frac{1}{2}}\right]e^{-c_0\kappa t}+\frac{1}{c_0\kappa}\|S\|_{L^\infty},
\end{align}
for some constant $C>0$.
\end{remark}

\begin{proof}[Proof of Lemma~\ref{De Giorgi iteration lemma}]
Fix $\kappa>0$, we drop $\kappa$ from $\theta^{\kappa}$ and write $\theta=\theta^{\kappa}$ for simplicity. Suppose that $M\ge2\|S\|_{L^\infty}$, where $M>0$ to be fixed later, we define
\begin{align*}
\lambda_{n}:=M(1-2^{-n}),\qquad n\in\N\cup\{0\},
\end{align*}
and $\theta_{n}$ to be the truncated function $\theta_n=\max\{\theta(t)-\lambda_n,0\}$. Fix $t\in(0,1]$, we define the time cutoffs by
\begin{align*}
T_{n}:=t(1-2^{-n}),
\end{align*}
and we denote the level set of energy by
\begin{align*}
Q_{n}:=\sup_{T_{n}\le\tau\le1}\|\theta_n(\cdot,\tau)\|^2_{L^2}+2\kappa\int_{T_{n}}^1\|\Lambda^{\frac{\gamma}{2}}\theta_{n}(\cdot,\tau)\|^2_{L^2}d\tau.
\end{align*}
Using the point-wise inequality given in \cite[Proposition 2.3]{CC04}, for all $s\in(T_{n-1},T_{n})$, we have the following level set inequality:
\begin{align*}
&\sup_{T_{n}\le\tau\le1}\|\theta_{n}(\cdot,\tau)\|^2_{L^2}+2\kappa\int_{T_{n}}^1\|\Lambda^{\frac{\gamma}{2}}\theta_{n}(\cdot,\tau)\|^2_{L^2}d\tau\\
&\le\|\theta_{n}(\cdot,s)\|^2_{L^2}+2\|S\|_{L^\infty}\int_{T_{n-1}}^1\|\theta_{n}(\cdot,\tau)\|_{L^1}d\tau.
\end{align*}
We take the mean value in $s$ on $(T_{n-1},T_{n})$ and multiply by $\frac{2^n}{t}$ to obtain
\begin{align}\label{bound on Qn}
Q_{n}\le \frac{2^{n}}{t}\int_{T_{n-1}}^1\|\theta_{n}(\cdot,\tau)\|^2_{L^2}d\tau+2\|S\|_{L^\infty}\int_{T_{n-1}}^1\|\theta_{n}(\cdot,\tau)\|_{L^1}d\tau
\end{align}
and using \eqref{energy inequality kappa>0}, we also have
\begin{align}\label{bound on Q0}
Q_{0}\le\|\theta_0\|^2_{L^2}+\frac{1}{c_0\kappa}\|S\|^2_{L^2}.
\end{align}
We aim at bounding the right side of \eqref{bound on Qn} by a power of $Q_{n-1}$. Using H\"{o}lder inequality and Sobolev embedding, there exists $C_d>0$ which depends on $d$ such that
\begin{align}\label{bound on Qn-1 using theta n-1}
Q_{n-1}\ge C_{d}\|\theta_{n-1}\|^2_{L^\frac{2(d+1)}{d+1-\gamma}(\T^d\times[T_{n-1},1]}\qquad\forall n\in\N.
\end{align}
Since $\theta_{n-1}\ge 2^{-n}M$ on the set $\{(x,\tau):\theta_{n}(x,\tau)>0\}$, together with \eqref{bound on Qn-1 using theta n-1}, we have
\begin{align*}
\frac{2^{n}}{t}\int_{T_{n}}^1\|\theta_{n}(\cdot,\tau)\|^2_{L^2}d\tau&\le\frac{2^{n+1}}{t}\int_{T_{n}}^1\int_{\T^d}\theta^2_{n-1}\cdot\1_{\{\theta_{n}>0\}}\\
&\le\frac{2^{n+1}}{t}\int_{T_{n}}^1\int_{\T^d}\theta^2_{n-1}\Big(\frac{2^n\theta_{n-1}}{M}\Big)^\frac{2\gamma}{d+1-\gamma}\\
&\le 2\Big(\frac{2^{n(\frac{d+1+\gamma}{d+1-\gamma})}}{tM^{\frac{2\gamma}{d+1-\gamma}}}\Big)\int_{T_{n}}^1\theta_{n-1}^\frac{2(d+1)}{d+1-\gamma}\\
&\le 2C_{d}\Big(\frac{2^{n(\frac{d+1+\gamma}{d+1-\gamma})}}{tM^{\frac{2\gamma}{d+1-\gamma}}}\Big)Q_{n-1}^\frac{d+1}{d+1-\gamma}.
\end{align*}
Similarly, we have
\begin{align*}
2\|S\|_{L^\infty}\int_{T_{n-1}}^1\|\theta_{n}(\cdot,\tau)\|_{L^1}d\tau\le 2C_{d}\|S\|_{L^\infty}\Big(\frac{2^{n(\frac{d+1+\gamma}{d+1-\gamma})}}{M^{\frac{2\gamma}{d+1-\gamma}+1}}\Big)Q_{n-1}^\frac{d+1}{d+1-\gamma}.
\end{align*}
Since $M\ge2\|S\|_{L^\infty}$, we further obtain
\begin{align*}
Q_{n}\le \Big(\frac{2C_{d}}{t}+C_{d}\Big)\Big(\frac{2^{n(\frac{d+1+\gamma}{d+1-\gamma})}}{M^{\frac{2\gamma}{d+1-\gamma}}}\Big)Q_{n-1}^\frac{d+1}{d+1-\gamma}.
\end{align*}
If we assume that
\begin{align}\label{bounding condition on M}
M\ge\Big(\frac{2C_d}{t}+C_d\Big)^\frac{d+1-\gamma}{2\gamma}Q_{0}^\frac{1}{2},
\end{align} 
then $Q_{n}\to0$ as $n\to\infty$. Hence using \eqref{bound on Q0}, if we choose $M>0$ such that
\begin{align*}
M\ge\Big(\frac{2C_d}{t}+C_d\Big)^\frac{d+1-\gamma}{2\gamma}\Big(\|\theta_0\|_{L^2}+\frac{\|S\|_{L^2}}{c_0^\frac{1}{2}\kappa^\frac{1}{2}}\Big)+2\|S\|_{L^\infty},
\end{align*}
then it implies that $\theta$ is bounded above by $M$. Applying the same argument to $-\theta$, we conclude that the bound \eqref{infty bound on theta in terms of L2 norm of initial theta} holds for all $t\in(0,1]$.
\end{proof}
With the help of the bound \eqref{infty bound on theta in terms of L2 norm of initial theta} on $\|\theta(\cdot,t)\|_{L^\infty}$, we are ready to give the proof of Theorem~\ref{global-in-time wellposedness in Sobolev}.

\begin{proof}[Proof of Theorem~\ref{global-in-time wellposedness in Sobolev}]
It is enough to establish an {\it a priori} estimate on $\theta^\kappa$ with initial data $\theta_0\in H^s$ for $s\ge0$. We multiply \eqref{abstract active scalar eqn} by $\Lambda^{2s}\theta^{\kappa}$ and integrate over $\T^d$ to obtain
\begin{align}\label{differential inequality on theta}
\frac{1}{2}\frac{d}{dt}\|\Lambda^s\theta^\kappa\|^2_{L^2}+\kappa\|\Lambda^{s+\frac{\gamma}{2}}\theta^\kappa\|^2_{L^2}\le\Big|\int_{\T^d}S\Lambda^{2s}\theta^\kappa\Big|+\Big|\int_{\T^d}u^\kappa\cdot\nabla\theta^\kappa\Lambda^{2s}\theta^\kappa\Big|.
\end{align}
The term $\dis\Big|\int_{\T^d}S\Lambda^{2s}\theta^\kappa\Big|$ is readily bounded by $\dis\|\Lambda^sS\|_{L^2}\|\Lambda^s\theta^\kappa\|_{L^2}$, and for $\dis\Big|\int_{\T^d}u^\kappa\cdot\nabla\theta^\kappa\Lambda^{2s}\theta^\kappa\Big|$, it can be estimated as follows.
\begin{align}\label{estimate on convective term}
&\Big|\int_{\T^d}u^\kappa\cdot\nabla\theta^\kappa\Lambda^{2s}\theta^\kappa\Big|=\Big|\int_{\T^d}(u^\kappa\theta^\kappa)\cdot\nabla\Lambda^{2s}\theta^\kappa\Big|\notag\\
&\le\Big|\int_{\T^d}[\Lambda^{s+1}(u^\kappa\theta^\kappa)-u^\kappa\Lambda^{s+1}\theta^\kappa]\cdot\nabla\Lambda^{s-1}\theta^\kappa\Big|+\Big|\int_{\T^d}u^\kappa\Lambda^{s+1}\theta^\kappa\cdot\nabla\Lambda^{s-1}\theta^\kappa\Big|.
\end{align}
Since $\divv (u^\kappa)=0$, the second term on the right side of \eqref{estimate on convective term} can be rewritten as follows
\begin{align*}
\Big|\int_{\T^d}u^\kappa\Lambda^{s+1}\theta^\kappa\cdot\nabla\Lambda^{s-1}\theta^\kappa\Big|=\Big|\int_{\T^d}\Lambda^{s}\theta^\kappa(\Lambda(u^\kappa\cdot\Lambda^{s-1}\theta^\kappa)-u^\kappa\cdot\Lambda^{s}\theta^\kappa)\Big|.
\end{align*}
Hence using \eqref{commutator estimate}, we have
\begin{align*}
&\Big|\int_{\T^d}u^\kappa\Lambda^{s+1}\theta^\kappa\cdot\nabla\Lambda^{s-1}\theta^\kappa\Big|=\Big|\int_{\T^d}\Lambda^{s}\theta^\kappa(\Lambda(u^\kappa\cdot\Lambda^{s-1}\theta^\kappa)-u^\kappa\cdot\Lambda^{s}\theta^\kappa)\Big|\\
&\le C\|\Lambda^{s}\theta^\kappa\|_{L^2}\|\nabla u^\kappa\|_{L^\infty}\|\nabla\Lambda^{s-1}\theta^\kappa\|_{L^2}\le C\|\Lambda u^\kappa\|_{L^\infty}\|\Lambda^{s}\theta^\kappa\|^2_{L^2}.
\end{align*}
For the leading term on the right side of \eqref{estimate on convective term}, we apply \eqref{commutator estimate} again to obtain
\begin{align*}
&\Big|\int_{\T^d}[\Lambda^{s+1}(u^\kappa\theta^\kappa)-u^\kappa\Lambda^{s+1}\theta^\kappa]\cdot\nabla\Lambda^{s-1}\theta^\kappa\Big|\\
&\le \|\Lambda^{s+1}(u^\kappa\theta^\kappa)-u^\kappa\Lambda^{s+1}\theta^\kappa\|_{L^2}\|\nabla\Lambda^{s-1}\theta^\kappa\|_{L^2}\\
&\le\Big(\|\Lambda u^\kappa\|_{L^\infty}\|\Lambda^{s}\theta^\kappa\|_{L^2}+\|\Lambda^{s+1}u^\kappa\|_{L^2}\|\theta^\kappa\|_{L^\infty}\Big)\|\nabla\Lambda^{s-1}\theta^\kappa\|_{L^2}.
\end{align*}
Hence we have from \eqref{estimate on convective term} that 
\begin{align}\label{estimate on convective term step 2}
\Big|\int_{\T^d}u^\kappa\cdot\nabla\theta^\kappa\Lambda^{2s}\theta^\kappa\Big|\le C\Big(\|\Lambda u^\kappa\|_{L^\infty}\|\Lambda^{s}\theta^\kappa\|_{L^2}+\|\Lambda^{s+1}u^\kappa\|_{L^2}\|\theta^\kappa\|_{L^\infty}\Big)\|\Lambda^{s}\theta^\kappa\|_{L^2}.
\end{align}
By \eqref{two order smoothing for u when nu>0}, the term $\|\Lambda^{s+1}u^\kappa\|_{L^2}$ can be bounded by $C_\nu\|\Lambda^{s}\theta^\kappa\|_{L^2}$, and for the term $\|\Lambda u^\kappa\|_{L^\infty}$, using \eqref{L infty bound 2} for $q=d+1$ and together with \eqref{two order smoothing for u when nu>0}, we have
\begin{align}\label{L infty bound on nabla u}
\|\Lambda u^\kappa\|_{L^\infty}&\le C\|\Lambda u^\kappa\|_{W^{1,d+1}}\notag\\
&\le C\|\theta^\kappa\|_{L^{d+1}}\le C\|\theta^{\kappa}\|_{L^\infty}.
\end{align}
Therefore, we obtain from \eqref{estimate on convective term step 2} that
\begin{align}\label{estimate on convective term step 3}
\Big|\int_{\T^d}u^\kappa\cdot\nabla\theta^\kappa\Lambda^{2s}\theta^\kappa\Big|\le C\|\theta^\kappa\|_{L^\infty}\|\Lambda^{s}\theta^\kappa\|^2_{L^2}.
\end{align}
We apply \eqref{estimate on convective term step 3} on \eqref{differential inequality on theta} and deduce that
\begin{align}\label{differential inequality on theta final}
\frac{d}{dt}\|\Lambda^s\theta^\kappa\|^2_{L^2}+2\kappa\|\Lambda^{s+\frac{\gamma}{2}}\theta^\kappa\|^2_{L^2}\le C\|\theta^\kappa\|_{L^\infty}\|\Lambda^{s}\theta^\kappa\|^2_{L^2}+2\|\Lambda^sS\|_{L^2}\|\Lambda^s\theta^\kappa\|_{L^2}.
\end{align}
For the case when $\kappa>0$, using the bounds \eqref{infty bound on theta in terms of L2 norm of initial theta} and \eqref{infty bound on theta with kappa>0 and t>1} on $\theta^\kappa$, for each $\tau>0$, there exists $C_\tau >0$ which depends on $\kappa$, $\tau$, $c_0$, $\|\theta_0\|_{L^2}$, $\|S\|_{L^\infty}$ but independent of $t$ such that
\begin{equation}\label{L infty bound on theta for t>tau}
\|\theta^\kappa(\cdot,t)\|_{L^\infty}\le C_\tau ,\qquad \forall t\ge\tau.
\end{equation}
Furthermore, by the continuity in time as proved in Lemma~\ref{local-in-time existence lemma}, we choose $\tau>0$ sufficiently small such that
\[\sup_{0\le t\le\tau}\|\Lambda^{s}\theta^\kappa(\cdot,t)\|_{L^2}\le 2\|\Lambda^{s}\theta_0\|_{L^2},\]
Hence with the help of Gr\"{o}nwall's inequality, we conclude from \eqref{differential inequality on theta final} that for $\kappa>0$,
\begin{align}\label{H^s bound for theta kappa>0 and nu>0}
\|\Lambda^{s}\theta^\kappa(\cdot,t)\|_{L^2}\le (2\|\Lambda^{s}\theta_0\|_{L^2}+C\|\Lambda^{s}S\|_{L^2})(e^{CC_\tau t}+1),\qquad \forall t\ge0.
\end{align}
Therefore, the above inequality implies that $\|\Lambda^{s}\theta^\kappa(\cdot,t)\|_{L^2}$ remains finite for all positive time when $\kappa>0$.

For the case when $\kappa=0$, under the assumption that $\theta_0\in L^\infty$, we use the bound \eqref{infty bound on theta with kappa=0 or >0} to deduce from \eqref{differential inequality on theta final} that
\begin{align*}
&\|\Lambda^s\theta^0(\cdot,t)\|_{L^2}\\
&\le \exp\Big(Ct(\|\theta_0\|_{L^\infty}+\|S\|_{L^\infty})\Big)(\|\Lambda^{s}\theta_0\|_{L^2}+C\|\Lambda^{s}S\|_{L^2}),\qquad\forall t>0,
\end{align*}
and hence $\|\Lambda^s\theta^0(\cdot,t)\|_{L^2}$ remains finite for all $t>0$ as well.
\end{proof}

\subsection{Convergence of $H^s$-solutions as $\kappa\to0$}

We are now ready to give the proof of Theorem~\ref{Hs convergence}. 

\begin{proof}[Proof of Theorem~\ref{Hs convergence}]
We let $\theta_0$, $S\in C^\infty$. The proof is divided into three parts.

\medskip

\noindent{\bf Uniform $H^s$-bound:}
For fixed $\nu>0$, we first obtain a uniform (independent of $\kappa$) $H^{s}$-bound for all $s\ge0$ on $\theta^\kappa$. Apply the estimate \eqref{differential inequality on theta final} on $\theta^\kappa$, we have
\begin{align}\label{differential inequality on theta s+1}
\frac{d}{dt}\|\Lambda^{s}\theta^\kappa\|^2_{L^2}+2\kappa\|\Lambda^{s+\frac{\gamma}{2}}\theta^\kappa\|^2_{L^2}\le C\|\theta^\kappa\|_{L^\infty}\|\Lambda^{s+1}\theta^\kappa\|^2_{L^2}+2\|\Lambda^{s}S\|_{L^2}\|\Lambda^{s}\theta^\kappa\|_{L^2}.
\end{align}
Since $\theta_0\in C^\infty\subset L^\infty$, we can apply the uniform bound \eqref{infty bound on theta with kappa=0 or >0} to deduce from \eqref{differential inequality on theta s+1} that 
\begin{align*}
&\frac{d}{dt}\|\Lambda^{s}\theta^\kappa\|^2_{L^2}+2\kappa\|\Lambda^{s+\frac{\gamma}{2}}\theta^\kappa\|^2_{L^2}\\
&\le C(\|\theta_0\|_{L^\infty}+\|S\|_{L^\infty})\|\Lambda^{s}\theta^\kappa\|^2_{L^2}+2\|\Lambda^{s}S\|_{L^2}\|\Lambda^{s}\theta^\kappa\|_{L^2}.
\end{align*}
Upon integrating the above inequality from $0$ to $t$, we obtain, for all $\kappa\ge0$ that
\begin{align}\label{uniform Hs+1 bound on theta for all kappa}
&\|\Lambda^{s}\theta^\kappa(\cdot,t)\|_{L^2}\\
&\le C(\|\Lambda^{s}\theta_0\|_{L^2}+\|\Lambda^{s}S\|_{L^2})\exp\Big(Ct(\|\theta_0\|_{L^\infty}+\|S\|_{L^\infty})\Big),\qquad\forall t>0.
\end{align}

\noindent{\bf $L^2$-convergence:}
Fix $\nu>0$ and $\gamma\in(0,2]$. We let $\theta^\kappa$, $\theta^0$ be the smooth solution to \eqref{abstract active scalar eqn} for $\kappa>0$ and $\kappa=0$ respectively. Define $\varphi=\theta^\kappa-\theta^0$, then $\varphi$ satisfies
\begin{align}\label{equation for varphi}
\partial_t\varphi+(u^\kappa-u^0)\cdot\nabla\theta^0+u^\kappa\cdot\nabla\varphi=-\kappa\Lambda^{\gamma}\theta^\kappa,
\end{align}
where $u^\kappa$ and $u^0$ are given by
\begin{align*}
u_j^\kappa:=\partial_{x_i} T_{ij}^{\nu}[\theta^\kappa],\qquad u_j^0:=\partial_{x_i} T_{ij}^{\nu}[\theta^0]
\end{align*}
for $1\le i,j\le d$. Multiply \eqref{equation for varphi} by $\varphi$ and integrate,
\begin{align}\label{integral of varphi}
\frac{1}{2}\frac{d}{dt}\|\varphi(\cdot,t)\|^2_{L^2}=-\int_{\mathbb{T}^d}(u^\kappa-u^0)\cdot\nabla\theta^0\cdot\varphi-\int_{\mathbb{T}^d}u^\kappa\cdot\nabla\varphi\cdot\varphi-\kappa\int_{\mathbb{T}^d}\Lambda^{\gamma}\theta^\kappa\cdot\varphi.
\end{align}
Using the divergence-free assumption on $u^\kappa$, the term $\dis-\int_{\mathbb{T}^d}u^\kappa\cdot\nabla\varphi\cdot\varphi$ is zero. On the other hand, using \eqref{two order smoothing for u when nu>0}, we have
\begin{equation*}
\|(u^\kappa-u^0)(\cdot,t)\|_{L^2}\le C_\nu\|(\theta^\kappa-\theta^0)(\cdot,t)\|_{L^2}.
\end{equation*}
Hence the first term on the right side of \eqref{integral of varphi} can be bounded as follows.
\begin{align*}
-\intox(u^\kappa-u^0)\cdot\nabla\theta^0\cdot\varphi&\le\|\nabla\theta^0(\cdot,t)\|_{L^\infty}\|(u^\kappa-u^0)(\cdot,t)\|_{L^2}\|\varphi(\cdot,t)\|_{L^2}\notag\\
&\le C_{\nu}\|\nabla\theta^0(\cdot,t)\|_{L^\infty}\|\varphi(\cdot,t)\|_{L^2}^2.
\end{align*}
For the third term on the right side of \eqref{integral of varphi}, we can write 
\begin{align*}
-\kappa\int_{\mathbb{T}^d}\Lambda^{\gamma}\theta^\kappa\cdot\varphi\le4\kappa\|\Lambda^{\gamma}\theta^\kappa(\cdot,t)\|^2_{L^2}+\frac{\kappa}{4}\|\varphi(\cdot,t)\|^2_{L^2}.
\end{align*}
Applying the above estimates on \eqref{integral of varphi}, we obtain
\begin{align}\label{differential of varphi}
\frac{1}{2}\frac{d}{dt}\|\varphi(\cdot,t)\|_{L^2}^2&\le4\kappa\|\Lambda^\gamma\theta^\kappa\|_{L^2}^2+\frac{\kappa}{4}\|\varphi(\cdot,t)\|^2_{L^2}+C_{\nu}\|\nabla\theta^0(\cdot,t)\|_{L^\infty}\|\varphi(\cdot,t)\|_{L^2}^2\notag\\
&\le\Big[C_{\nu}\|\nabla\theta^0(\cdot,t)\|_{L^\infty}+\frac{\kappa}{4}\Big]\|\varphi(\cdot,t)\|_{L^2}^2+4\kappa\|\Lambda^\gamma\theta^\kappa\|_{L^2}^2.
\end{align}
Using the bound \eqref{uniform Hs+1 bound on theta for all kappa}, for sufficiently large $s>1$, we have
\begin{align*}
&\|\nabla\theta^0(\cdot,t)\|_{L^\infty}\\
&\le C\|\Lambda^{s}\theta^\kappa(\cdot,t)\|_{L^2}\le C(\|\Lambda^{s}\theta_0\|_{L^2}+\|\Lambda^{s}S\|_{L^2})\exp\Big(Ct(\|\theta_0\|_{L^\infty}+\|S\|_{L^\infty})\Big),
\end{align*}
as well as
\begin{align*}
\|\Lambda^\gamma\theta^\kappa\|_{L^2}&\le C\|\Lambda^{s}\theta^\kappa(\cdot,t)\|_{L^2}\\
&\le C(\|\Lambda^{s}\theta_0\|_{L^2}+\|\Lambda^{s}S\|_{L^2})\exp\Big(Ct(\|\theta_0\|_{L^\infty}+\|S\|_{L^\infty})\Big).
\end{align*}
Hence by integrating \eqref{differential of varphi} over $t$ and using Gr\"{o}nwall's inequality, for $t>0$ there exist positive functions $C_1(t)$, $C_2(t)$ depending on $t$, $\theta_0$ and $S$ such that
\begin{align*}
\|\varphi(\cdot,t)\|_{L^2}^2\le \kappa C_1(t)e^{C_2(t)},
\end{align*}
where we recall that $\varphi(\cdot,0)=0$. Therefore, as $\kappa\rightarrow0$, we have 
\begin{align}\label{L2 convergence}
\lim_{\kappa\rightarrow0}\int_{\mathbb{T}^d}|\theta^\kappa-\theta^0|^2(x,t)dx=\lim_{\kappa\rightarrow0}\|\varphi(\cdot,t)\|_{L^2}^2=0.
\end{align}

\noindent{\bf $H^s$-convergence:} By Sobolev inequality, for $s>0$, there exists $\sigma\in(0,1)$ such that for $t>0$,
\begin{align}\label{Hs estimate on varphi}
\|(\theta^\kappa-\theta^0)(\cdot,t)\|_{H^s}\le\|(\theta^\kappa-\theta^0)(\cdot,t)\|_{L^2}^\sigma\|(\theta^\kappa-\theta^0)(\cdot,t)\|_{H^{s+1}}^{1-\sigma}.
\end{align}
Using the uniform bound \eqref{uniform Hs+1 bound on theta for all kappa} on $\|(\theta^\kappa-\theta^0)(\cdot,t)\|_{H^{s+1}}^{1-\sigma}$ and the $L^2$-convergence result \eqref{L2 convergence}, by taking $\kappa\to0$ on \eqref{Hs estimate on varphi}, the result \eqref{Hs convergence thm} holds for all $s>0$ as well. 
\end{proof}

\section{Existence and convergence of analytic and Gevrey-class solutions when $\nu=0$}\label{Existence and convergence of Gevrey-class solutions section}

In this section, we address the existence and convergence of the analytic and Gevrey-class solutions to \eqref{abstract active scalar eqn} for $\nu=0$ as $\kappa\to0$. More precisely, we consider the following system:
\begin{align}
\label{abstract active scalar eqn nu=0} \left\{ \begin{array}{l}
\partial_t\theta^\kappa+u^\kappa\cdot\nabla\theta^\kappa=-\kappa\Lambda^{\gamma}\theta^\kappa+S, \\
u^\kappa_j=\partial_{x_i} T_{ij}^{0}[\theta^\kappa],\theta^\kappa(x,0)=\theta_0(x),
\end{array}\right.
\end{align}
where $T_{ij}^{0}$ satisfies the bounds
\begin{equation*}
\sup_{\{k\in\mathbb{Z}^d_*\}}|\widehat{T^0_{ij}}(k)|\le C_{0},
\end{equation*}
and $C_0$ is given in assumption A5. The results will be proved in the following subsections. In Subsection~\ref{Local-in-time existence and convergence of analytic solutions}, we prove that the equation \eqref{abstract active scalar eqn nu=0} possesses local-in-time analytic solutions $\theta^\kappa$ for all $\kappa\ge0$ and $\gamma\in(0,2]$, and we prove that $\theta^\kappa$ converges to $\theta^0$ in terms of analytic norm as $\kappa\to0$. In Subsection~\ref{Global-in-time existence of analytic solutions with small Sobolev initial data}, we proceed to prove that the local-in-time analytic solutions obtained in Subsection~\ref{Local-in-time existence and convergence of analytic solutions} can be extended globally in time under a smallness assumption on the initial data in terms of a certain {\it Sobolev} norm. Finally, in Subsection~\ref{Existence and convergence of Gevrey-class solutions with bounded operators}, we prove the existence and convergence of the Gevrey-class $s$ solutions to \eqref{abstract active scalar eqn nu=0} for $s\ge1$ when the Fourier symbols $\widehat{\partial_{x_i} T_{ij}^{0}}(k)$ are {\it bounded} functions in $k$.

\subsection{Local-in-time existence and convergence of analytic solutions}\label{Local-in-time existence and convergence of analytic solutions} 

In this subsection, we prove the convergence of analytic solutions as stated in Theorem~\ref{analytic convergence kappa}. We first obtain the local-in-time existence of analytic solutions to the equation \eqref{abstract active scalar eqn nu=0} when $\kappa\ge0$ and $\gamma\in(0,2]$, which is illustrated by Theorem~\ref{Local-in-time existence of analytic solutions thm} below.

\begin{theorem}[Local-in-time existence of analytic solutions]\label{Local-in-time existence of analytic solutions thm}
Let $\kappa\ge0$ and $\gamma\in(0,2]$ be fixed, and let $\theta_0$ and $S$ be the initial datum and forcing term respectively. Fix $K_0>0$. Suppose $\theta_0$ and $S$ are analytic functions with radius of convergence $\tau_0>0$ and
\begin{align}\label{bounds on analytic norm of S and theta0}
\|\Lambda^re^{\tau_0\Lambda}\theta_0\|_{L^2}\le K_0,\qquad\|\Lambda^re^{\tau_0\Lambda}S\|_{L^2}\le K_0,
\end{align}
where $\Lambda:=(-\Delta)^\frac{1}{2}$ and $r>\frac{d}{2}+\frac{5}{2}$. Then there exists $T_*=T_*(\tau_0, K_0) > 0$ and a unique analytic solution on $[0, T_*)$ to the initial value problem associated to \eqref{abstract active scalar eqn nu=0}.
\end{theorem}

\begin{proof}[Proof of Theorem~\ref{Local-in-time existence of analytic solutions thm}]
A proof can be found in \cite{FV11b} and we include here for the sake of completeness. We fix $r$ such that $r>\frac{d}{2}+\frac{5}{2}$. We denote $\tau=\tau(t)$ and then multiply \eqref{abstract active scalar eqn nu=0}$_1$ by $\Lambda^{2r} e^{2\tau\Lambda}\theta^{\kappa}$ and integrate to obtain
\begin{align}\label{evolution identity in analytic norm}
&\frac{1}{2}\frac{d}{dt}\|\Lambda^r e^{\tau\Lambda}\theta^{\kappa}\|^2_{L^2}-\dot{\tau}\|\Lambda^{r+\frac{1}{2}}e^{\tau\Lambda}\theta^{\kappa}\|^2_{L^2}+\kappa\|\Lambda^{r+\frac{\gamma}{2}}e^{\tau\Lambda}\theta^{\kappa}\|^2_{L^2} \notag \\
&= \mathcal{R}+\langle S,\Lambda^{2r}e^{2\tau\Lambda}\theta^{\kappa}\rangle,
\end{align}
where $\mathcal{R}$ is given by $\mathcal{R}=\langle u^{\kappa}\cdot\nabla\theta^{\kappa},\Lambda^{2r} e^{2\tau\Lambda}\theta^{\kappa}\rangle$. The term $\langle S,\Lambda^{2r}e^{2\tau\Lambda}\theta^{\kappa}\rangle$ can be readily bounded by $\|\Lambda^r e^{\tau\Lambda}S\|_{L^2}\|\Lambda^r e^{\tau\Lambda}\theta^{\kappa}\|_{L^2}$ and thus we focus on $\mathcal{R}$. Using the bound \eqref{one order singular for u when nu=0} that $|\widehat{u}(j)|\le C|j||\widehat{\theta}(j)|$ and the fact $|j|^\frac{1}{2}\le 2 |l|^\frac{1}{2}|k|^\frac{1}{2}$ and $|k|^\frac{1}{2}\le 2 |l|^\frac{1}{2}|j|^\frac{1}{2}$ for $s\ge1$ and $|j|,|k|,|l|\ge1$, we can bound $\mathcal{R}$ by
\begin{align*}
\mathcal{R}&\le C\sum_{j+k=l;j,k,l\in\mathbb{Z}^d_*}|j||\hat{\theta^{\kappa}}(j)||k||\hat{\theta^{\kappa}}(k)||l|^{2r}e^{2\tau|l|}|\hat{\theta^{\kappa}}(l)|\notag\\
&\le C\sum_{j+k=l;j,k,l\in\mathbb{Z}^d_*}|\hat{\theta^{\kappa}}(j)|e^{\tau|j|}|\hat{\theta^{\kappa}}(k)|e^{\tau|k|}|\hat{\theta^{\kappa}}(l)|e^{\tau|l|}|l|^{r}(|j|^{r+1}|k|+|j||k|^{r+1})\notag\\
&\le C\sum_{j+k=l;j,k,l\in\mathbb{Z}^d_*}\Big[|\hat{\theta^{\kappa}}(j)||j|^{r+\frac{1}{2}}e^{\tau|j|}\Big]\Big[|\hat{\theta^{\kappa}}(k)||k|^{1+\frac{1}{2}}e^{\tau|k|}\Big]\Big[|\hat{\theta^{\kappa}}(l)||l|^{r+\frac{1}{2}}e^{\tau|l|}\Big]\notag\\
&+C\sum_{j+k=l;j,k,l\in\mathbb{Z}^d_*}\Big[|\hat{\theta^{\kappa}}(j)||j|^{1+\frac{1}{2}}e^{\tau|j|}\Big]\Big[|\hat{\theta^{\kappa}}(k)||k|^{r+\frac{1}{2}}e^{\tau|k|}\Big]\Big[|\hat{\theta^{\kappa}}(l)||l|^{r+\frac{1}{2}}e^{\tau|l|}\Big]\notag\\
&\le C\|\Lambda^{r+\frac{1}{2}}e^{\tau\Lambda}\theta^{\kappa}\|^2_{L^2}\|\Lambda^r e^{\tau\Lambda}\theta^{\kappa}\|_{L^2},
\end{align*}
where the last inequality follows since $r>\frac{d}{2}+\frac{5}{2}$. Hence we obtain from \eqref{evolution identity in analytic norm} that
\begin{align}\label{a priori bound on theta singular analytic}
&\frac{1}{2}\frac{d}{dt}\|\Lambda^r e^{\tau\Lambda}\theta^{\kappa}\|^2_{L^2}\notag\\
&\le (\dot{\tau}(t)+C\|\Lambda^r e^{\tau\Lambda}\theta^{\kappa}\|_{L^2})\|\Lambda^{r+\frac{1}{2}}e^{\tau\Lambda}\theta^{\kappa}\|^2_{L^2}+\|\Lambda^r e^{\tau\Lambda}S\|_{L^2}\|\Lambda^r e^{\tau\Lambda}\theta^{\kappa}\|_{L^2}.
\end{align}
In view of \eqref{a priori bound on theta singular analytic}, we can apply the same argument given in \cite{FS19}. Specifically, we let $\tau(t)$ be decreasing and satisfies
\begin{align*}
\dot{\tau}+4CK_0=0,
\end{align*}
with initial condition $\tau(0)=\tau_0$, then we have $\dot{\tau}(t)+C\|\Lambda^r e^{\tau\Lambda}\theta^{\kappa}\|_{L^2}<0$, and from \eqref{a priori bound on theta singular analytic} that 
\begin{align}\label{bound on solution nu=0}
\|\Lambda^r e^{\tau\Lambda}\theta^{\kappa}\|_{L^2}\le\|\Lambda^r e^{\tau\Lambda}\theta_0\|_{L^2}+2t\|\Lambda^r e^{\tau\Lambda}S\|_{L^2}=3K_0
\end{align}
as long as $\tau(t)>0$ and $t\le1$. Hence it implies the existence of analytic solution $\theta^\kappa$ on $[0,T_*)$, where the maximal time of existence of the analytic solution is given by $T_*=\min\{\frac{\tau_0}{4CK_0},1\}$.
\end{proof}

Once we obtain the existence of analytic solutions to \eqref{abstract active scalar eqn nu=0}, we are ready to prove the convergence of analytic solutions as $\kappa\to0$, thereby proving Theorem~\ref{analytic convergence kappa}.

\begin{proof}[Proof of Theorem~\ref{analytic convergence kappa}] Let $\varphi=\theta^{\kappa}-\theta^{0}$, then we have
\begin{equation}\label{energy identity for the difference convergence}
\dt\varphi+(u^{\kappa}-u^{0})\cdot\nabla\theta^{0}+u^{\kappa}\nabla\varphi=-\kappa\Lambda^{\gamma}\theta^{0}.
\end{equation}
For $r>\frac{d}{2}+\frac{5}{2}$, we multiply \eqref{energy identity for the difference convergence} by $\Lambda^{2r}e^{2\tau\Lambda}\theta^0$ and obtain
\begin{align}\label{differential equation for analytic norm of varphi} 
\frac{1}{2}\frac{d}{dt}\|\Lambda^r e^{\tau\Lambda}\varphi(\cdot,t)\|^2_{L^2}&=\dot{\tau}\|\Lambda^{r+\frac{1}{2}} e^{\tau\Lambda}\varphi(\cdot,t)\|^2_{L^2}+\langle(u^{\kappa}-u^{0})\cdot\nabla\theta^{0},\Lambda^{2r}e^{2\tau\Lambda^\frac{1}{2}}\varphi\rangle\notag\\
&\qquad+\langle u^{\kappa}\cdot\nabla\varphi,\Lambda^{2r}e^{2\tau\Lambda^\frac{1}{2}}\varphi\rangle-\langle\kappa\Lambda^{\gamma}\theta^{0},\Lambda^{2r}e^{2\tau\Lambda^\frac{1}{2}}\varphi\rangle\notag\\
&=\dot{\tau}\|\Lambda^{r+\frac{1}{2}} e^{\tau\Lambda}\varphi(\cdot,t)\|^2_{L^2}+\mathcal{R}_1+\mathcal{R}_2+\mathcal{R}_3,
\end{align}
where
\begin{align*}
\mathcal{R}_1&=\langle(u^{\kappa}-u^{0})\cdot\nabla\theta^{0},\Lambda^{2r}e^{2\tau\Lambda^\frac{1}{2}}\varphi\rangle,\\
\mathcal{R}_2&=\langle u^{\kappa}\cdot\nabla\varphi,\Lambda^{2r}e^{2\tau\Lambda^\frac{1}{2}}\varphi\rangle,\\
\mathcal{R}_3&=-\langle\kappa\Lambda^{\gamma}\theta^{0},\Lambda^{2r}e^{2\tau\Lambda^\frac{1}{2}}\varphi\rangle.
\end{align*}
To estimate $\mathcal{R}_1$, following the method of bounding $\mathcal{R}$ as in the proof of Theorem~\ref{Local-in-time existence of analytic solutions thm}, we have
\begin{align*}
|\mathcal{R}_1|&\le C\Big|\sum_{j+k=l;j,k,l\in\mathbb{Z}^d_*}\widehat{(u^{\kappa}-u^{0})}(j)\cdot k\widehat{\theta^{0}}(k)|l|^{2r}e^{2\tau|l|}\hat{\varphi}(-l)\Big|\\
&\le C\sum_{j+k=l;j,k,l\in\mathbb{Z}^d_*}|j||k|\Big[|j|^r+|k|^r\big]|\hat{\varphi}(j)||\hat{\varphi}(j)|e^{\tau|j|}|\widehat{\theta^{0}}(k)|e^{\tau|k|}|l|^re^{\tau|l|}|\hat{\varphi}(l)|\\
&\le C\|\Lambda^{r}e^{\tau\Lambda}\theta^{0}\|_{L^2}\|\Lambda^{r+\frac{1}{2}}e^{\tau\Lambda}\varphi\|^2_{L^2}+C\|\Lambda^{r+\frac{1}{2}}e^{\tau\Lambda}\theta^{0}\|_{L^2}\|\Lambda^{r+\frac{1}{2}}e^{\tau\Lambda}\varphi\|_{L^2}\|\Lambda^r e^{\tau\Lambda}\varphi\|_{L^2}\\
&\le C\|\Lambda^{r+\frac{1}{2}}e^{\tau\Lambda}\theta^{0}\|_{L^2}\|\Lambda^{r+\frac{1}{2}}e^{\tau\Lambda}\varphi\|^2_{L^2}.
\end{align*}
Similarly, we can bound $\mathcal{R}_2$ and $\mathcal{R}_3$ respectively by 
\begin{equation*}
|\mathcal{R}_2|\le C\|\Lambda^{r+\frac{1}{2}}e^{\tau\Lambda}\theta^{\kappa}\|_{L^2}\|\Lambda^{r+\frac{1}{2}}e^{\tau\Lambda}\varphi\|^2_{L^2},
\end{equation*}
and
\begin{equation*}
|\mathcal{R}_3|\le \kappa\|\Lambda^{r+\frac{\gamma}{2}}e^{\tau\Lambda}\theta^{0}\|_{L^2}\|\Lambda^r e^{\tau\Lambda}\varphi\|_{L^2}.
\end{equation*}
Hence we conclude from \eqref{differential equation for analytic norm of varphi} that
\begin{align*}
\frac{1}{2}\frac{d}{dt}\|\Lambda^r e^{\tau\Lambda}\varphi\|^2_{L^2}&\le \|\Lambda^{r+\frac{1}{2}}e^{\tau\Lambda}\varphi\|^2_{L^2}\Big[\dot{\tau}+C\|\Lambda^{r+\frac{1}{2}}e^{\tau\Lambda}\theta^{0}\|_{L^2}+C\|\Lambda^{r+\frac{1}{2}}e^{\tau\Lambda}\theta^{\kappa}\|_{L^2}\Big]\\
&\qquad+\kappa\|\Lambda^{r+\frac{\gamma}{2}}e^{\tau\Lambda}\theta^{0}\|_{L^2}\|\Lambda^r e^{\tau\Lambda}\varphi\|_{L^2}.
\end{align*}
Choose $\tau>0$ such that 
$$\dot{\tau}+C\|\Lambda^{r+\frac{1}{2}}e^{\tau\Lambda}\theta^{0}\|_{L^2}+C\|\Lambda^{r+\frac{1}{2}}e^{\tau\Lambda}\theta^{\kappa}\|_{L^2}<0,$$ 
then we get
\begin{equation*}
\frac{d}{dt}\|\Lambda^r e^{\tau\Lambda}\varphi\|_{L^2}\le \kappa\|\Lambda^{r+\frac{\gamma}{2}}e^{\tau\Lambda}\theta^{0}\|_{L^2},
\end{equation*}
and the result follows by taking $\kappa\to0$.
\end{proof}

\subsection{Global-in-time existence of analytic solutions with small Sobolev initial data}\label{Global-in-time existence of analytic solutions with small Sobolev initial data}

In this subsection, we prove that for the case when $\kappa>0$ and $\gamma\in[1,2]$, under a smallness assumption on the initial data, the analytic solutions to \eqref{abstract active scalar eqn nu=0} obtained in Theorem~\ref{Local-in-time existence of analytic solutions thm} exist for all time:

\begin{theorem}[Global-in-time existence of analytic solutions]\label{Global-in-time existence of analytic solutions thm}
Let $\kappa>0$ and $\gamma\in[1,2]$, and suppose that both $\theta_0$ and $S$ are both analytic functions. There exists $\varepsilon>0$ depending on $\kappa$ such that, if $\theta_0$ and $S$ satisfy
\begin{equation}\label{smallness assumption on theta0 and S 1}
\|\theta_0\|_{L^2}^\beta\|\theta_0\|_{H^\alpha}^{1-\beta}+\|\theta_0\|_{L^2}^\beta\|S\|^{1-\beta}_{H^{\alpha-\frac{\gamma}{2}}}\le\varepsilon,
\end{equation}
and
\begin{equation}\label{smallness assumption on theta0 and S 2}
\|\Lambda^\alpha\theta_0\|^2_{L^2}+\frac{2}{\kappa^2}\|S\|^2_{H^{\alpha-\frac{\gamma}{2}}}\le\varepsilon,
\end{equation} 
where $\alpha>\frac{d+2}{2}+(1-\gamma)$ and $\beta=1-\frac{1}{\alpha}\Big[\frac{d+2}{2}+(1-\gamma)\Big]$, then the local-in-time analytic solution $\theta^\kappa$ as claimed by Theorem~\ref{Local-in-time existence of analytic solutions thm} can be extended to all time.
\end{theorem}

Before we give the proof of Theorem~\ref{Global-in-time existence of analytic solutions thm}, we state and prove the following global-in-time existence theorem of $H^\alpha$-solution to \eqref{abstract active scalar eqn nu=0} under the smallness assumption \eqref{smallness assumption on theta0 and S 1}.

\begin{theorem}[Global-in-time existence of Sobolev solutions]\label{Global-in-time existence of Sobolev solutions thm}
Let $\kappa>0$, $\gamma\in[1,2]$ and $S\in H^{\alpha-\frac{\gamma}{2}}(\mathbb{T}^d)$, where $\alpha>\frac{d+2}{2} + (1-\gamma)$.
There exists a small enough constant $\varepsilon>0$ depending on $\kappa$, such that if $\theta_0$ satisfies \eqref{smallness assumption on theta0 and S 1}, then there exists a unique global-in-time $H^\alpha$-solution to \eqref{abstract active scalar eqn nu=0}. In particular, for all $t>0$, we have the following bound on $\theta^\kappa$:
\begin{equation}\label{bound on H alpha norm}
\|\Lambda^\alpha\theta^\kappa(\cdot,t)\|^2_{L^2}\le \|\Lambda^\alpha\theta_0\|^2_{L^2}+\frac{2}{\kappa^2}\|S\|^2_{H^{\alpha-\frac{\gamma}{2}}}.
\end{equation}
\end{theorem}

\begin{proof}[Proof of Theorem~\ref{Global-in-time existence of Sobolev solutions thm}]
The proof is similar to the one given in \cite[Theorem~3.2]{FRV12}. We multiply \eqref{abstract active scalar eqn nu=0}$_1$ by $\Lambda^{2\alpha}\theta^\kappa$, integrate by parts to obtain
\begin{align}\label{L2 inequality with alpha}
\frac{1}{2}\frac{d}{dt}\|\Lambda^\alpha\theta^\kappa\|^2_{L^2}+\kappa\|\Lambda^{\frac{\gamma}{2}+\alpha}\theta^\kappa\|^2_{L^2}\le\Big|\intox S\Lambda^{2\alpha}\theta^\kappa\Big|+\Big|\intox u^\kappa\cdot\nabla\theta^\kappa\Lambda^{2\alpha}\theta^\kappa\Big|.
\end{align}
The first term on the right side of \eqref{L2 inequality with alpha} is bounded by
\begin{align}\label{L2 estimates on the S term}
\Big|\intox S\Lambda^{2\alpha}\theta^\kappa\Big|\le \frac{1}{2\kappa}\|\Lambda^{\alpha-\frac{\gamma}{2}}S\|^2_{L^2}+\frac{\kappa}{2}\|\Lambda^{\alpha+\frac{\gamma}{2}}\theta^\kappa\|^2_{L^2}.
\end{align}
The second term on the right side of \eqref{L2 inequality with alpha} can be estimated as follows.
\begin{align}\label{estimate on the trillinear term}
\Big|\intox u^\kappa\cdot\nabla\theta^\kappa\Lambda^{2\alpha}\theta^\kappa\Big|\le \|\Lambda^{\alpha-\frac{\gamma}{2}}(u^\kappa\cdot\nabla\theta^\kappa)\|_{L^2}\|\Lambda^{\alpha+\frac{\gamma}{2}}\theta^\kappa\|_{L^2}.
\end{align}
To estimate the product term $\|\Lambda^{\alpha-\frac{\gamma}{2}}(u^\kappa\cdot\nabla\theta^\kappa)\|_{L^2}$, using \eqref{product estimate}, we readily have
\begin{align}\label{estimate on product term 1}
&\|\Lambda^{\alpha-\frac{\gamma}{2}}(u^\kappa\cdot\nabla\theta^\kappa)\|_{L^2}\notag\\
&\le C \Big(\|\Lambda^{\alpha-\frac{\gamma}{2}}u^\kappa\|_{L^p}\|\nabla\theta^\kappa\|_{L^\frac{2p}{p-2}}+\|\Lambda^{\alpha-\frac{\gamma}{2}}\nabla\theta^\kappa\|_{L^p}\|u^\kappa\|_{L^\frac{2p}{p-2}}\Big),
\end{align}
where $p=\frac{2d}{2(1-\gamma)+d}$, and notice that $p>2$ as $\gamma\in[1,2]$. Since by \eqref{one order singular for u when nu=0}, we have
\begin{align*}
\|u^\kappa\|_{L^p}\le C|\Lambda\theta^\kappa\|_{L^p}.
\end{align*}
and together with \eqref{Sobolev inequality}, we further get 
\begin{align*}
\|\Lambda^{\alpha-\frac{\gamma}{2}}u^\kappa\|_{L^p}&\le C|\Lambda^{\alpha-\frac{\gamma}{2}+1}\theta^\kappa\|_{L^p}\\
&\le C \|\Lambda^{\alpha-\frac{\gamma}{2}+1+s}\theta^\kappa\|_{L^2},
\end{align*}
where $s$ satisfies $\frac{1}{p}=\frac{1}{2}-\frac{s}{d}$ and gives $s=\gamma-1$. Therefore we obtain from \eqref{estimate on product term 1} that 
\begin{align}\label{estimate on product term 2}
&\|\Lambda^{\alpha-\frac{\gamma}{2}}(u^\kappa\cdot\nabla\theta^\kappa)\|_{L^2}\notag\\
&\le C \Big(\|\Lambda^{\alpha-\frac{\gamma}{2}+1+\gamma-1}\theta^\kappa\|_{L^2}\|\nabla\theta^\kappa\|_{L^\frac{2p}{p-2}}+\|\Lambda^{\alpha-\frac{\gamma}{2}+1+\gamma-1}\theta^\kappa\|_{L^p}\|\Lambda\theta^\kappa\|_{L^\frac{2p}{p-2}}\Big)\notag\\
&\le C\|\Lambda^{\alpha+\frac{\gamma}{2}}\theta^\kappa\|_{L^2}\|\Lambda\theta^\kappa\|_{L^\frac{2p}{p-2}}.
\end{align}
We apply \eqref{estimate on product term 2} on \eqref{estimate on the trillinear term} to get
\begin{align}\label{estimate on the trillinear term final}
\Big|\intox u^\kappa\cdot\nabla\theta^\kappa\Lambda^{2\alpha}\theta^\kappa\Big|\le C\|\Lambda^{\alpha+\frac{\gamma}{2}}\theta^\kappa\|_{L^2}^2\|\Lambda\theta^\kappa\|_{L^\frac{2p}{p-2}}.
\end{align}
Using the estimates \eqref{L2 estimates on the S term} and \eqref{estimate on the trillinear term final} on \eqref{L2 inequality with alpha} and applying Cauchy inequality,
\begin{align}\label{L2 inequality with alpha 2}
\frac{1}{2}\frac{d}{dt}\|\Lambda^\alpha\theta^\kappa\|^2_{L^2}+\frac{\kappa}{2}\|\Lambda^{\frac{\gamma}{2}+\alpha}\theta^\kappa\|^2_{L^2}\le\frac{1}{2\kappa}\|\Lambda^{\alpha-\frac{\gamma}{2}}S\|^2_{L^2}+C\|\Lambda^{\alpha+\frac{\gamma}{2}}\theta^\kappa\|_{L^2}^2\|\Lambda\theta^\kappa\|_{L^\frac{2p}{p-2}}.
\end{align}
For the term $\|\Lambda\theta^\kappa\|_{L^\frac{2p}{p-2}}$, we can bound it as follows. Take $\beta=1-\frac{1}{\alpha}[\frac{d+2}{2}+(1-\gamma)]$, then $\beta\in(0,1)$ and we have
\begin{align*}
\|\Lambda\theta^\kappa\|_{L^\frac{2p}{p-2}}\le C\|\theta^\kappa\|^\beta_{L^2}\|\Lambda^\alpha\theta^\kappa\|^{1-\beta}_{L^2}.
\end{align*}
If $\theta^\kappa$ satisfies
\begin{equation}\label{condition on theta kappa}
\|\theta^\kappa\|^\beta_{L^2}\|\Lambda^\alpha\theta^\kappa\|^{1-\beta}_{L^2}\le\frac{\kappa}{4C},
\end{equation}
then we obtain from equation \eqref{L2 inequality with alpha 2} that
\begin{align*}
\frac{1}{2}\frac{d}{dt}\|\Lambda^\alpha\theta^\kappa\|^2_{L^2}+\frac{\kappa}{2}\|\Lambda^{\frac{\gamma}{2}+\alpha}\theta^\kappa\|^2_{L^2}\le\frac{1}{2\kappa}\|\Lambda^{\alpha-\frac{\gamma}{2}}S\|^2_{L^2}+\frac{\kappa}{4}\|\Lambda^{\alpha+\frac{\gamma}{2}}\theta^\kappa\|_{L^2}^2,
\end{align*}
which implies
\begin{align}\label{L2 inequality with alpha 3}
\frac{d}{dt}\|\Lambda^\alpha\theta^\kappa\|^2_{L^2}+\frac{\kappa}{2}\|\Lambda^{\frac{\gamma}{2}+\alpha}\theta^\kappa\|^2_{L^2}\le\frac{1}{\kappa}\|S\|^2_{H^{\alpha-\frac{\gamma}{2}}}.
\end{align}
We conclude from \eqref{L2 inequality with alpha 3} that for all $t>0$,
\begin{equation}\label{energy bound on theta kappa final}
\left\{ \begin{array}{l}
\|\Lambda^\alpha\theta^\kappa(\cdot,t)\|^2_{L^2}\le\|\Lambda^\alpha\theta_0\|^2_{L^2}+\frac{2}{\kappa^2}\|S\|^2_{H^{\alpha-\frac{\gamma}{2}}}, \\
\dis\|\theta^\kappa(\cdot,t)\|^2_{L^2}+2\kappa\int_0^t\|\Lambda^\frac{\gamma}{2}\theta^\kappa(\cdot,\tilde t)\|^2_{L^2}d\tilde t\le\|\theta_0\|^2_{L^2}.
\end{array}\right.
\end{equation}
In view of \eqref{energy bound on theta kappa final}, we can see that condition \eqref{condition on theta kappa} is satisfied for all $t>0$ if \eqref{smallness assumption on theta0 and S 1} holds for $\varepsilon$ being sufficiently small. Hence we complete the proof of Theorem~\ref{Global-in-time existence of Sobolev solutions thm}.
\end{proof}

We are now ready to give the proof of Theorem~\ref{Global-in-time existence of analytic solutions thm}.

\begin{proof}[Proof of Theorem~\ref{Global-in-time existence of analytic solutions thm}]
We fix $r$ such that $r>\frac{d}{2}+\frac{5}{2}$. We multiply \eqref{abstract active scalar eqn nu=0}$_1$ by $\Lambda^{2r}e^{2\tau\Lambda}\theta^{\kappa}$ and integrate to obtain
\begin{align}\label{evolution eqn in gevrey norm}
&\frac{1}{2}\frac{d}{dt}\|\Lambda^r e^{\tau\Lambda}\theta^{\kappa}\|^2_{L^2}-\dot{\tau}\|\Lambda^{r+\frac{1}{2}}e^{\tau\Lambda}\theta^{\kappa}\|^2_{L^2}+\kappa\|\Lambda^{r+\frac{\gamma}{2}}e^{\tau\Lambda}\theta^{\kappa}\|^2_{L^2}\notag\\
&\le \mathcal{R}+\|\Lambda^r e^{\tau\Lambda}S\|_{L^2}\|\Lambda^r e^{\tau\Lambda}\theta^{\kappa}\|_{L^2},
\end{align}
where $\mathcal{R}$ is defined in the proof of Theorem~\ref{Local-in-time existence of analytic solutions thm}. We modify the way for estimating the term $\mathcal{R}$ as follows. Using the bound 
\begin{align*}
\mbox{$e^x\le1+xe^x$, \qquad for all $x\ge0$,}
\end{align*}
we can bound $\mathcal{R}$ by
\begin{align}\label{bound on R gevrey estimates}
\mathcal{R}&\le C\sum_{j+k=l;j,k,l\in\mathbb{Z}^d_*}\Big[|\hat{\theta^{\kappa}}(j)||j|^{r+\frac{1}{2}}e^{\tau|j|}\Big]\Big[|\hat{\theta^{\kappa}}(k)||k|^{1+\frac{1}{2}}e^{\tau|k|}\Big]\Big[|\hat{\theta^{\kappa}}(l)||l|^{r+\frac{1}{2}}e^{\tau|l|}\Big]\notag\\
&+C\sum_{j+k=l;j,k,l\in\mathbb{Z}^d_*}\Big[|\hat{\theta^{\kappa}}(j)||j|^{1+\frac{1}{2}}e^{\tau|j|}\Big]\Big[|\hat{\theta^{\kappa}}(k)||k|^{r+\frac{1}{2}}e^{\tau|k|}\Big]\Big[|\hat{\theta^{\kappa}}(l)||l|^{r+\frac{1}{2}}e^{\tau|l|}\Big]\notag\\
&\le C\Big[\sum_{m\in\mathbb{Z}^d_*}|\hat{\theta^{\kappa}}(m)||m|^{\frac{3}{2}}(1+\tau|m|e^{\tau|m|})\Big]\|\Lambda^{r+\frac{1}{2}}e^{\tau\Lambda}\theta^{\kappa}\|^2_{L^2}\notag\\
&\le C_r\|\Lambda^{r+\frac{1}{2}}e^{\tau\Lambda}\theta^{\kappa}\|^2_{L^2}\Big[a(t)+\tau\|\Lambda^r e^{\tau\Lambda}\theta^{\kappa}\|_{L^2}\Big],
\end{align}
where $\dis a(t):=\sum_{m\in\mathbb{Z}^d_*}|m|^{\frac{3}{2}}|\hat{\theta^{\kappa}}(m)|(t)$ and we used the fact that
\begin{align*}
&\sum_{m\in\mathbb{Z}^d_*}|m|^{1+\frac{3}{2}}|\hat{\theta^{\kappa}}(m)|e^{\tau|m|}\\
&\le \Big(\sum_{m\in\mathbb{Z}^d_*}|m|^{5-2r}\Big)^\frac{1}{2}\Big(\sum_{m\in\mathbb{Z}^d_*}|m|^{2r}|\hat{\theta^{\kappa}}(m)|^2e^{2\tau|m|}\Big)^\frac{1}{2}\le C_r\|\Lambda^r e^{\tau\Lambda}\theta^{\kappa}\|_{L^2},
\end{align*}
for $C_r$ being a large enough positive constant depending only on $r$ and $r>\frac{d}{2}+\frac{5}{2}$. We apply \eqref{bound on R gevrey estimates} to \eqref{evolution eqn in gevrey norm} and obtain
\begin{align}\label{evolution eqn in gevrey norm v2}
&\frac{1}{2}\frac{d}{dt}\|\Lambda^r e^{\tau\Lambda}\theta^{\kappa}\|^2_{L^2}+\kappa\|\Lambda^{r+\frac{\gamma}{2}}e^{\tau\Lambda}\theta^{\kappa}\|^2_{L^2}\notag\\
&\le \Big[\dot{\tau}+C_ra(t)+\tau C_r\|\Lambda^r e^{\tau\Lambda}\theta^{\kappa}\|_{L^2}\Big]\|\Lambda^{r+\frac{1}{2}}e^{\tau\Lambda}\theta^{\kappa}\|^2_{L^2}+\|\Lambda^r e^{\tau\Lambda}S\|_{L^2}\|\Lambda^r e^{\tau\Lambda}\theta^{\kappa}\|_{L^2}.
\end{align}
For $\gamma\ge1$, we have $\frac{\gamma}{2}\ge\frac{1}{2}$, hence
\begin{align*}
\|\Lambda^{r+\frac{\gamma}{2}}e^{\tau\Lambda^{r+\frac{1}{s}}}\theta^{\kappa}\|_{L^2}\ge \|\Lambda^{r+\frac{1}{2}}e^{\tau\Lambda}\theta^{\kappa}\|_{L^2}.
\end{align*}
On the other hand, for $\alpha>\frac{d+2}{2}+\frac{1}{2}$, we have 
\begin{align}\label{bound on a(t) in terms of theta}
a(t)\le\Big(\sum_{m\in\mathbb{Z}^d_*}|m|^{3-2\alpha}\Big)^\frac{1}{2}\Big(\sum_{m\in\mathbb{Z}^d_*}|m|^{2\alpha}|\hat{\theta^{\kappa}}(m)|^2\Big)^\frac{1}{2}\le C_{\alpha}\|\Lambda^\alpha\theta\|_{L^2}
\end{align}
for some large enough positive constant depending only on $\alpha$. If we assume that
\begin{align*}
\|\Lambda^\alpha\theta_0\|^2_{L^2}+\frac{2}{\kappa^2}\|S\|^2_{H^{\alpha-\frac{\gamma}{2}}}\le\frac{\kappa^2}{4C_rC_{\alpha}},
\end{align*}
then using the bound \eqref{bound on H alpha norm}, \eqref{bound on a(t) in terms of theta} implies that
\begin{equation}\label{bound on a(t)}
a(t)\le\frac{\kappa^2}{4C_r}.
\end{equation}
Applying the bound \eqref{bound on a(t)} to \eqref{evolution eqn in gevrey norm v2}, we obtain
\begin{align}\label{evolution eqn in analytic norm v3}
&\frac{1}{2}\frac{d}{dt}\|\Lambda^r e^{\tau\Lambda}\theta^{\kappa}\|^2_{L^2}+\frac{\kappa}{2}\|\Lambda^{r+\frac{\gamma}{2}}e^{\tau\Lambda}\theta^{\kappa}\|^2_{L^2}\notag\\
&\le \Big[\dot{\tau}+\tau C_r\|\Lambda^r e^{\tau\Lambda}\theta^{\kappa}\|_{L^2}\Big]\|\Lambda^{r+\frac{1}{2}}e^{\tau\Lambda}\theta^{\kappa}\|^2_{L^2}+\|\Lambda^r e^{\tau\Lambda}S\|_{L^2}\|\Lambda^r e^{\tau\Lambda}\theta^{\kappa}\|_{L^2}.
\end{align}
Choose $\dot{\tau}+\tau C_r\|\Lambda^r e^{\tau\Lambda}\theta^{\kappa}\|_{L^2}=0$ in \eqref{evolution eqn in analytic norm v3}, then it gives
\begin{align*}
\frac{1}{2}\frac{d}{dt}\|\Lambda^r e^{\tau\Lambda}\theta^{\kappa}\|^2_{L^2}+\frac{\kappa}{2}\|\Lambda^{r+\frac{\gamma}{2}}e^{\tau\Lambda}\theta^{\kappa}\|^2_{L^2}\le \|\Lambda^r e^{\tau\Lambda}S\|_{L^2}\|\Lambda^r e^{\tau\Lambda}\theta^{\kappa}\|_{L^2},
\end{align*}
hence 
\begin{align*}
\|\Lambda^r e^{\tau\Lambda}\theta^{\kappa}\|_{L^2}\le \|\Lambda^r e^{\tau\Lambda}\theta_0\|_{L^2}+t\|\Lambda^r e^{\tau\Lambda}S\|_{L^2}
\end{align*}
and for all $t>0$, $\tau=\tau(t)$ satisfies the bound
\begin{align*}
\tau(t)&=\tau_0\exp\Big(-C_r\int_0^t\|\Lambda^r e^{\tau\Lambda}\theta^{\kappa}(\cdot,\tilde t)\|_{L^2}d\tilde t\Big)\\
&\ge\tau_0\exp\Big(-C_r\int_0^t(\|\Lambda^r e^{\tau\Lambda}\theta_0\|_{L^2}+\tilde t\|\Lambda^r e^{\tau\Lambda}S\|_{L^2})d\tilde t\Big)>0,
\end{align*}
Therefore the local-in-time analytic solution as claimed by Theorem~\ref{Local-in-time existence of analytic solutions thm} can be extended to all time, thereby proving Theorem~\ref{Global-in-time existence of analytic solutions thm}.
\end{proof}

\begin{remark}
For the case when $\kappa>0$ and $\gamma=2$, we point out that the smallness assumption \eqref{smallness assumption on theta0 and S 1} on $\theta_0$ can be removed. The reason is that, we can obtain the bound \eqref{bound on a(t)} on $a(t)$ {\it without} the smallness assumption \eqref{smallness assumption on theta0 and S 1} (also refer to \cite{FV12} for more detailed discussions). The critical MG$^0$ equation is an example of equation \eqref{abstract active scalar eqn nu=0} with $\gamma = 2$.
\end{remark}

\begin{remark}
It is also worth mentioning that all the abstract results obtained in Theorem~\ref{Local-in-time existence of analytic solutions thm}, Theorem~\ref{Global-in-time existence of analytic solutions thm} and Theorem~\ref{Global-in-time existence of Sobolev solutions thm} can be applied to the critical SQG equation, which is a special example of \eqref{abstract active scalar eqn nu=0} with $\gamma = 1$.
\end{remark}

\subsection{Existence and convergence of Gevrey-class solutions with bounded symbols $\widehat{\partial_{x_i} T_{ij}^{0}}$}\label{Existence and convergence of Gevrey-class solutions with bounded operators}

In this subsection, we address the existence and convergence of Gevrey-class solutions to the equation \eqref{abstract active scalar eqn nu=0} under a stronger assumption on the operators $T_{ij}^{0}$. More precisely, for all $1\le i,j\le d$, we further assume that
\begin{align}\label{stronger assumption on the singular operators}
\sup_{\{k\in\mathbb{Z}^d_*\}}|\widehat{\partial_{x_i} T_{ij}^{0}}(k)|\le C_{*}
\end{align}
for some positive constant $C_{*}$. The condition \eqref{stronger assumption on the singular operators} implies that 
\begin{align}\label{stronger assumption on the singular operators u} 
\mbox{$|\hat{u^{\kappa}}(j)|\le C(C_{*})|\hat{\theta^{\kappa}}(j)|$,\qquad for all $j\in\mathbb{Z}^d_*$.}
\end{align}
The results are summarised in the following theorem:
\begin{theorem}[Existence and convergence of Gevrey-class solutions]\label{Existence and convergence of Gevrey-class solutions thm}
Let $\kappa\ge0$ and $\gamma\in(0,2]$ be fixed, and let $\theta_0$ and $S$ be the initial datum and forcing term respectively. Fix $s\ge1$ and $K_0>0$. Suppose $\theta_0$ and $S$ both belong to Gevrey-class $s$ with radius of convergence $\tau_0>0$ and
\begin{align}\label{bounds on Gevrey norm of S and theta0}
\|\Lambda^re^{\tau_0\Lambda^\frac{1}{s}}\theta^\kappa(\cdot,0)\|_{L^2}\le K_0,\qquad\|\Lambda^re^{\tau_0\Lambda^\frac{1}{s}}S\|_{L^2}\le K_0,
\end{align}
where $r>\frac{d}{2}+\frac{5}{2}$. Assume further that the condition \eqref{stronger assumption on the singular operators} holds for $1\le i,j\le d$. Then there exists $T_*=T_*(\tau_0, K_0) > 0$ and a unique Gevrey-class $s$ solution on $[0, T_*)$ to the initial value problem associated to \eqref{abstract active scalar eqn nu=0}. Moreover, there exists $T\le T_*$ and $\tau=\tau(t)<\tau_0$ such that, for $t\in[0,T]$, we have:
\begin{align}\label{Gevrey convergence}
\lim_{\kappa\to0}\|(\Lambda^re^{\tau\Lambda^\frac{1}{s}}\theta^{\kappa}-\Lambda^re^{\tau\Lambda^\frac{1}{s}}\theta^{0})(\cdot,t)\|_{L^2}=0.
\end{align}
\end{theorem}

\begin{proof}[Proof of Theorem~\ref{Existence and convergence of Gevrey-class solutions thm}]
We fix $r$ such that $r>\frac{d}{2}+\frac{5}{2}$ and multiply \eqref{abstract active scalar eqn nu=0}$_1$ by $\Lambda^r e^{2\tau\Lambda^{\frac{1}{s}}}\theta^{\kappa}$ and integrate to obtain
\begin{align}\label{evolution identity in gevrey norm}
&\frac{1}{2}\frac{d}{dt}\|\Lambda^r e^{\tau\Lambda^{\frac{1}{s}}}\theta^{\kappa}\|^2_{L^2}-\dot{\tau}\|\Lambda^{r+\frac{1}{2s}}e^{\tau\Lambda^{\frac{1}{s}}}\theta^{\kappa}\|^2_{L^2}+\kappa\|\Lambda^{r+\frac{\gamma}{2}}e^{\tau\Lambda^{\frac{1}{s}}}\theta^{\kappa}\|^2_{L^2}\notag\\
&= \mathcal{T}+\langle S,\Lambda^{2r}e^{2\tau\Lambda^{\frac{1}{s}}}\theta^{\kappa}\rangle,
\end{align}
where $\mathcal{T}$ is given by $\mathcal{T}:=\langle u^{\kappa}\cdot\nabla\theta^{\kappa},\Lambda^{2r} e^{2\tau\Lambda^{\frac{1}{s}}}\theta^{\kappa}\rangle$. The term $\langle S,\Lambda^{2r}e^{2\tau\Lambda^{\frac{1}{s}}}\theta^{\kappa}\rangle$ can be readily bounded by $\|\Lambda^r e^{\tau\Lambda^{\frac{1}{s}}}S\|_{L^2}\|\Lambda^r e^{\tau\Lambda^{\frac{1}{s}}}\theta^{\kappa}\|_{L^2}$. For $\mathcal{T}$, since $\divv (u^\kappa)=0$, we have $\langle u^{\kappa}\cdot\nabla \Lambda^r e^{\tau\Lambda^{\frac{1}{s}}}\theta^\kappa,\Lambda^r e^{\tau\Lambda^{\frac{1}{s}}}\theta^\kappa\rangle=0$ and hence
\begin{align*}
|\mathcal{T}|=|\langle u^{\kappa}\cdot\nabla\theta^{\kappa},\Lambda^{2r} e^{2\tau\Lambda^{\frac{1}{s}}}\theta^{\kappa}\rangle-\langle u^{\kappa}\cdot\nabla \Lambda^r e^{\tau\Lambda^{\frac{1}{s}}}\theta^\kappa,\Lambda^r e^{\tau\Lambda^{\frac{1}{s}}}\theta^\kappa\rangle|\le \mathcal{T}_1+\mathcal{T}_2
\end{align*}
where $\mathcal{T}_1$ and $\mathcal{T}_2$ are given by
\begin{align*}
\mathcal{T}_1&:=C\sum_{j+k=l;j,k,l\in\mathbb{Z}^d_*}||l|^r-|k|^r||\hat{u^{\kappa}}(j)||\hat{\theta^{\kappa}}(k)||k||l|^r|\hat{\theta^{\kappa}}(l)|e^{\tau|l|^\frac{1}{s}}e^{\tau|k|^\frac{1}{s}},\\
\mathcal{T}_2&:=C\sum_{j+k=l;j,k,l\in\mathbb{Z}^d_*}|l|^r|e^{\tau|l|^\frac{1}{s}-\tau|k|^\frac{1}{s}}-1||\hat{u^{\kappa}}(j)||\hat{\theta^{\kappa}}(k)||k||l|^re^{\tau|l|^\frac{1}{s}}e^{\tau|k|^\frac{1}{s}}.
\end{align*}
To estimate $\mathcal{T}_1$, we apply the similar method given in \cite{KV09} and \cite{MV11}. Using mean value theorem, there exists $\zeta_{k,l}\in(0,1)$ such that
\begin{align*}
|l|^r-|k|^r=((\zeta_{k,l}|l|+(1-\zeta_{k,l})|k|)^{r-1}-|k|^{r-1})+(|l|-|k|)|k|^{r-1}.
\end{align*}
Since $j+k=l$, we have $|(|l|-|k|)|k|^{r-1}|\le |j||k|^{r-1}$ as well as
\begin{align*}
|(\zeta_{k,l}|l|+(1-\zeta_{k,l})|k|)^{r-1}-|k|^{r-1}|\le C|j|^2(|j|^{r-2}+|k|^{r-2}).
\end{align*}
Together with \eqref{stronger assumption on the singular operators u}, we have
\begin{align*}
\mathcal{T}_1&\le C\sum_{j+k=l;j,k,l\in\mathbb{Z}^d_*}|j|^2(|j|^{r-2}+|k|^{r-2})|\hat{\theta^{\kappa}}(j)||\hat{\theta^{\kappa}}(k)||k||l|^r|\hat{\theta^{\kappa}}(l)|e^{\tau|l|^\frac{1}{s}}e^{\tau|k|^\frac{1}{s}}\\
&\qquad+C\sum_{j+k=l;j,k,l\in\mathbb{Z}^d_*}|j||k|^{r-1}|\hat{\theta^{\kappa}}(j)||\hat{\theta^{\kappa}}(k)||k||l|^r|\hat{\theta^{\kappa}}(l)|e^{\tau|l|^\frac{1}{s}}e^{\tau|k|^\frac{1}{s}}\\
&\le C\|\Lambda^r\theta^{\kappa}\|_{L^2}\|\Lambda^r e^{\tau\Lambda^{\frac{1}{s}}}\theta^{\kappa}\|^2_{L^2},
\end{align*}
where the last inequality holds for $r>\frac{d}{2}+\frac{5}{2}$. To estimate $\mathcal{T}_2$, we use \eqref{stronger assumption on the singular operators u} and the inequality $|e^{\tau|l|^\frac{1}{s}-\tau|k|^\frac{1}{s}}-1|\le |\tau|l|^\frac{1}{s}-\tau|k|^\frac{1}{s}|e^{|\tau|l|^\frac{1}{s}-\tau|k|^\frac{1}{s}|}$ to obtain
\begin{align*}
\mathcal{T}_2&\le C\tau\sum_{j+k=l;j,k,l\in\mathbb{Z}^d_*}|l|^r||l|^\frac{1}{s}-|k|^\frac{1}{s}|e^{|\tau|l|^\frac{1}{s}-\tau|k|^\frac{1}{s}|}|\hat{\theta^{\kappa}}(j)||\hat{\theta^{\kappa}}(k)||k||l|^re^{\tau|l|^\frac{1}{s}}e^{\tau|k|^\frac{1}{s}}\\
&\le C\tau\sum_{j+k=l;j,k,l\in\mathbb{Z}^d_*}|l|^r||l|^\frac{1}{s}-|k|^\frac{1}{s}|e^{\tau|j|^\frac{1}{s}}|\hat{\theta^{\kappa}}(j)||\hat{\theta^{\kappa}}(k)||k||l|^re^{\tau|l|^\frac{1}{s}}e^{\tau|k|^\frac{1}{s}},
\end{align*}
where the last inequality follows since $|\tau|l|^\frac{1}{s}-\tau|k|^\frac{1}{s}|\le\tau|j|^\frac{1}{s}$. Depending on the values of $|j|$ and $|k|$, we have the following estimates:
\begin{itemize}
\item If $|j|\le|k|$, using the inequality $||l|^\frac{1}{s}-|k|^\frac{1}{s}|\le\frac{|j|}{|l|^{1-\frac{1}{s}}+|k|^{1-\frac{1}{s}}}$ for $s\ge1$, we have
\begin{align*}
&|l|^r||l|^\frac{1}{s}-|k|^\frac{1}{s}|e^{\tau|j|^\frac{1}{s}}|\hat{\theta^{\kappa}}(j)||\hat{\theta^{\kappa}}(k)||k||l|^re^{\tau|l|^\frac{1}{s}}e^{\tau|k|^\frac{1}{s}}\\
&\le |k|^{r-\frac{1}{2s}} e^{\tau|k|^\frac{1}{s}} \frac{|j|}{|l|^{1-\frac{1}{s}}+|k|^{1-\frac{1}{s}}}|\hat{\theta^{\kappa}}(j)||\hat{\theta^{\kappa}}(k)||k||l|^{r+\frac{1}{2s}}e^{\tau|l|^\frac{1}{s}}e^{\tau|k|^\frac{1}{s}}\\
&\le\Big(|k|^{r+\frac{1}{2s}}e^{\tau|k|^\frac{1}{s}}|\hat{\theta^{\kappa}}(k)|\Big)\Big(|l|^{r+\frac{1}{2s}}e^{\tau|l|^\frac{1}{s}}|\hat{\theta^{\kappa}}(l)|\Big)\Big(|j|e^{\tau|j|^\frac{1}{s}}|\hat{\theta^{\kappa}}(j)|\Big).
\end{align*}
\item If $|j|\ge|k|$, using the inequality $|l|^\frac{1}{2s}\le C|j|^\frac{1}{2s}$ for $s\ge1$, we have
\begin{align*}
&|l|^r||l|^\frac{1}{s}-|k|^\frac{1}{s}|e^{\tau|j|^\frac{1}{s}}|\hat{\theta^{\kappa}}(j)||\hat{\theta^{\kappa}}(k)||k||l|^re^{\tau|l|^\frac{1}{s}}e^{\tau|k|^\frac{1}{s}}\\
&\le C|j|^{r+\frac{1}{2s}}e^{\tau|j|^\frac{1}{s}}(|l|^\frac{1}{s}+|k|^\frac{1}{s})|\hat{\theta^{\kappa}}(j)||\hat{\theta^{\kappa}}(k)||k||l|^{r-\frac{1}{2s}}e^{\tau|l|^\frac{1}{s}}e^{\tau|k|^\frac{1}{s}}\\
&\le C\Big(|k|e^{\tau|k|^\frac{1}{s}}|\hat{\theta^{\kappa}}(k)|\Big)\Big(|l|^{r+\frac{1}{2s}}e^{\tau|l|^\frac{1}{s}}|\hat{\theta^{\kappa}}(l)|\Big)\Big(|j|^{r+\frac{1}{2s}}e^{\tau|j|^\frac{1}{s}}|\hat{\theta^{\kappa}}(j)|\Big)\\
&\qquad+C\Big(|k|^{1+\frac{1}{s}}e^{\tau|k|^\frac{1}{s}}|\hat{\theta^{\kappa}}(k)|\Big)\Big(|l|^{r-\frac{1}{2s}}e^{\tau|l|^\frac{1}{s}}|\hat{\theta^{\kappa}}(l)|\Big)\Big(|j|^{r+\frac{1}{2s}}e^{\tau|j|^\frac{1}{s}}|\hat{\theta^{\kappa}}(j)|\Big).
\end{align*}
\end{itemize}
Hence for $r>\frac{d}{2}+\frac{5}{2}$, we have
\begin{align*}
\mathcal{T}_2\le C\tau\|\Lambda^r e^{\tau\Lambda^{\frac{1}{s}}}\theta^{\kappa}\|_{L^2}\|\Lambda^{r+\frac{1}{2s}} e^{\tau\Lambda^{\frac{1}{s}}}\theta^{\kappa}\|^2_{L^2}.
\end{align*}
Combining the estimates on $\mathcal{T}_1$ and $\mathcal{T}_2$, we conclude that
\begin{align}\label{bound on nonlinear term T}
|\mathcal{T}|\le C(\|\Lambda^r\theta^{\kappa}\|_{L^2}+\tau\|\Lambda^r e^{\tau\Lambda^{\frac{1}{s}}}\theta^{\kappa}\|_{L^2})\|\Lambda^{r+\frac{1}{2s}} e^{\tau\Lambda^{\frac{1}{s}}}\theta^{\kappa}\|^2_{L^2}.
\end{align}
We apply \eqref{bound on nonlinear term T} on \eqref{evolution identity in gevrey norm} to obtain that
\begin{align*}
&\frac{1}{2}\frac{d}{dt}\|\Lambda^r e^{\tau\Lambda^{\frac{1}{s}}}\theta^{\kappa}\|^2_{L^2}\notag\\
&\le (\dot{\tau}+C(\|\Lambda^r\theta^{\kappa}\|_{L^2}+\tau\|\Lambda^r e^{\tau\Lambda^{\frac{1}{s}}}\theta^{\kappa}\|_{L^2}))\|\Lambda^{r+\frac{1}{2s}}e^{\tau\Lambda^{\frac{1}{s}}}\theta^{\kappa}\|^2_{L^2}\\
&+\|\Lambda^r e^{\tau\Lambda^{\frac{1}{s}}}S\|_{L^2}\|\Lambda^r e^{\tau\Lambda^{\frac{1}{s}}}\theta^{\kappa}\|_{L^2}\notag\\
&\le (\dot{\tau}+C\|\Lambda^r e^{\tau\Lambda^{\frac{1}{s}}}\theta^{\kappa}\|_{L^2})\|\Lambda^{r+\frac{1}{2s}}e^{\tau\Lambda^{\frac{1}{s}}}\theta^{\kappa}\|^2_{L^2}+\|\Lambda^r e^{\tau\Lambda^{\frac{1}{s}}}S\|_{L^2}\|\Lambda^r e^{\tau\Lambda^{\frac{1}{s}}}\theta^{\kappa}\|_{L^2},
\end{align*}
where the last inequality follows provided that $\tau\le1$. We can then apply the same argument given in the proof of Theorem~\ref{Local-in-time existence of analytic solutions thm}, which implies the existence of a Gevrey-class $s$ solution $\theta^\kappa$ on $[0,T_*)$, where the maximal time of existence of the Gevrey-class $s$ solution is given by $T_*=\frac{\tau_0}{4CK_0}$.

To prove the convergence result \eqref{Gevrey convergence}, we let $\varphi=\theta^{\kappa}-\theta^{0}$. Then we multiply \eqref{energy identity for the difference convergence} by $\Lambda^{2r}e^{2\tau\Lambda^\frac{1}{s}}\theta^0$ to give
\begin{align}\label{differential equation for Gevrey norm of varphi} 
\frac{1}{2}\frac{d}{dt}\|\Lambda^r e^{\tau\Lambda^\frac{1}{s}}\varphi(\cdot,t)\|^2_{L^2}&=\dot{\tau}\|\Lambda^{r+\frac{1}{2s}} e^{\tau\Lambda^\frac{1}{s}}\varphi(\cdot,t)\|^2_{L^2}+\mathcal{T}_1+\mathcal{T}_2+\mathcal{T}_3,
\end{align}
where $\mathcal{T}_1$, $\mathcal{T}_2$ and $\mathcal{T}_3$ are given by
\begin{align*}
\mathcal{T}_1&=\langle(u^{\kappa}-u^{0})\cdot\nabla\theta^{0},\Lambda^{2r}e^{2\tau\Lambda^\frac{1}{2s}}\varphi\rangle,\\
\mathcal{T}_2&=\langle u^{\kappa}\cdot\nabla\varphi,\Lambda^{2r}e^{2\tau\Lambda^\frac{1}{2s}}\varphi\rangle,\\
\mathcal{T}_3&=-\langle\kappa\Lambda^{\gamma}\theta^{0},\Lambda^{2r}e^{2\tau\Lambda^\frac{1}{2s}}\varphi\rangle.
\end{align*}
Using the method for bounding the term $\mathcal{T}$, we have 
\begin{align*}
&|\mathcal{T}_1|+|\mathcal{T}_2|\\
&\le C\|\Lambda^{r+\frac{1}{2s}}e^{\tau\Lambda}\theta^{0}\|_{L^2}\|\Lambda^{r+\frac{1}{2s}}e^{\tau\Lambda}\varphi\|^2_{L^2}+C\|\Lambda^{r+\frac{1}{2s}}e^{\tau\Lambda}\theta^{\kappa}\|_{L^2}\|\Lambda^{r+\frac{1}{2s}}e^{\tau\Lambda}\varphi\|^2_{L^2},
\end{align*}
and $\mathcal{T}_3$ can be readily bounded by $\kappa\|\Lambda^{r+\frac{\gamma}{2}}e^{\tau\Lambda^\frac{1}{s}}\theta^{0}\|_{L^2}\|\Lambda^r e^{\tau\Lambda^\frac{1}{s}}\varphi\|_{L^2}$. We conclude from \eqref{differential equation for Gevrey norm of varphi} that
\begin{align*}
&\frac{1}{2}\frac{d}{dt}\|\Lambda^r e^{\tau\Lambda^\frac{1}{s}}\varphi(\cdot,t)\|^2_{L^2}\\
&\le \|\Lambda^{r+\frac{1}{2s}}e^{\tau\Lambda^\frac{1}{s}}\varphi\|^2_{L^2}\Big[\dot{\tau}+C\|\Lambda^{r+\frac{1}{2s}}e^{\tau\Lambda^\frac{1}{s}}\theta^{0}\|_{L^2}+C\|\Lambda^{r+\frac{1}{2s}}e^{\tau\Lambda^\frac{1}{s}}\theta^{\kappa}\|_{L^2}\Big]\\
&\qquad+\kappa\|\Lambda^{r+\frac{\gamma}{2}}e^{\tau\Lambda^\frac{1}{s}}\theta^{0}\|_{L^2}\|\Lambda^r e^{\tau\Lambda^\frac{1}{s}}\varphi\|_{L^2}.
\end{align*}
Choose $\tau>0$ such that 
$$\dot{\tau}+C\|\Lambda^{r+\frac{1}{2s}}e^{\tau\Lambda^\frac{1}{s}}\theta^{0}\|_{L^2}+C\|\Lambda^{r+\frac{1}{2s}}e^{\tau\Lambda^\frac{1}{s}}\theta^{\kappa}\|_{L^2}<0,$$ 
then we get
\begin{equation*}
\frac{d}{dt}\|\Lambda^r e^{\tau\Lambda^\frac{1}{s}}\varphi(\cdot,t)\|_{L^2}\le \kappa\|\Lambda^{r+\frac{\gamma}{2}}e^{\tau\Lambda^\frac{1}{s}}\theta^{0}\|_{L^2},
\end{equation*}
and \eqref{Gevrey convergence} follows by taking $\kappa\to0$.
\end{proof}

\begin{remark}
Theorem~\ref{Existence and convergence of Gevrey-class solutions thm} can be applied on critical SQG equation as well as incompressible porous media (IPM) equation; see \cite{CFG11}, \cite{CGO07}, \cite{FS19} for more discussions on IPM equation.
\end{remark}

\begin{remark}
In the same spirit as Theorem~\ref{Global-in-time existence of analytic solutions thm}, one can show that the Gevrey-class $s$ solutions obtained in Theorem~\ref{Existence and convergence of Gevrey-class solutions thm} can be extended globally in time provided that 
\begin{equation*}
\|\theta_0\|_{L^2}^\beta\|\theta_0\|_{H^r}^{1-\beta}+\|\theta_0\|_{L^2}^\beta\|S\|^{1-\beta}_{H^{r-\frac{\gamma}{2}}}\le\varepsilon
\end{equation*}
and
\begin{equation*}
\|\Lambda^\alpha\theta_0\|^2_{L^2}+\frac{2}{\kappa^2}\|S\|^2_{H^{r-\frac{\gamma}{2}}}\le\varepsilon
\end{equation*}
for some sufficiently small number $\varepsilon$. Here $r>\frac{d}{2}+\frac{5}{2}$ is the one given in Theorem~\ref{Existence and convergence of Gevrey-class solutions thm} and $\beta$ satisfies $\beta=1-\frac{1}{r}\Big[\frac{d+2}{2}+(1-\gamma)\Big]$. The key of the proof is to apply the estimate \eqref{bound on nonlinear term T} on $\mathcal{T}$, which gives
\begin{align}\label{evolution eqn in gevrey norm v3}
&\frac{1}{2}\frac{d}{dt}\|\Lambda^r e^{\tau\Lambda^{\frac{1}{s}}}\theta^{\kappa}\|^2_{L^2}+\frac{\kappa}{2}\|\Lambda^{r+\frac{\gamma}{2}}e^{\tau\Lambda^{\frac{1}{s}}}\theta^{\kappa}\|^2_{L^2}\notag\\
&\le \Big[\dot{\tau}+C(\|\Lambda^r\theta^{\kappa}\|_{L^2}+\tau\|\Lambda^r e^{\tau\Lambda^{\frac{1}{s}}}\theta^{\kappa}\|_{L^2})\Big]\|\Lambda^{r+\frac{1}{2s}}e^{\tau\Lambda^{\frac{1}{s}}}\theta^{\kappa}\|^2_{L^2}\notag\\
&\qquad+\|\Lambda^r e^{\tau\Lambda^{\frac{1}{s}}}S\|_{L^2}\|\Lambda^r e^{\tau\Lambda^{\frac{1}{s}}}\theta^{\kappa}\|_{L^2}.
\end{align}
Under the smallness assumption on $\theta_0$, we can apply the bound \eqref{bound on H alpha norm} on $C\|\Lambda^r\theta^{\kappa}\|_{L^2}$ so that the term $C\|\Lambda^r\theta^{\kappa}\|_{L^2}\|\Lambda^{r+\frac{1}{2s}}e^{\tau\Lambda^{\frac{1}{s}}}\theta^{\kappa}\|^2_{L^2}$ can be absorbed by the left side of \eqref{evolution eqn in gevrey norm v3}. The rest of the argument is almost identical to the proof of Theorem~\ref{Global-in-time existence of analytic solutions thm} and we omit the details here.
\end{remark}

\section{Long time behaviour for solutions when $\nu>0$ and $\kappa>0$}\label{long time behaviour and attractors section}

In this section, we study the long time behaviour for solutions to the active scalar equations \eqref{abstract active scalar eqn} when $\nu>0$ and $\kappa>0$. Based on the global-in-time existence results established in Theorem~\ref{global-in-time wellposedness in Sobolev}, for fixed $\nu>0$ and $\kappa>0$, we can define a solution operator $\pin(t)$ for the initial value problem \eqref{abstract active scalar eqn} via
\begin{align}\label{def of solution map nu>0 and kappa>0}
\pin(t): H^1\to H^1,\qquad \pin(t)\theta_0=\theta(\cdot,t),\qquad t\ge0.
\end{align}
We study the long-time dynamics of $\pin(t)$ on the phase space $H^1$. Specifically, we establish the existence of global attractors for $\pin(t)$ in $H^1$, which will be given in Subsection~\ref{existence of global attractors subsection}. Once we obtain the existence of global attractors, we further address some properties for the attractors which will be explained in Subsection~\ref{Further properties for the attractors subsection}. The results obtained in Subsection~\ref{existence of global attractors subsection} and Subsection~\ref{Further properties for the attractors subsection} will be sufficient for proving Theorem~\ref{existence of global attractor theorem}.

%For the sake of simplicity, throughout this section, we write $\theta=\theta^\kappa$ and $u=u[\theta]=u[\theta^\kappa]$ unless otherwise specified.

\subsection{Existence of global attractors in $H^1$-space}\label{existence of global attractors subsection}

The following theorem gives the main results for this subsection:

\begin{theorem}[Existence of $H^1$-global attractor]\label{Existence of H1 global attractor}
Let $S\in L^\infty\cap H^1$. For $\nu$, $\kappa>0$ and $\gamma\in(0,2]$, the solution map $\pin(t):H^1\to H^1$ associated to \eqref{abstract active scalar eqn} possesses a unique global attractor $\Gg$. Moreover, there exists $M_{\Gg}$ which depends only on $\nu$, $\kappa$, $\gamma$, $\|S\|_{L^\infty\cap H^1}$ and universal constants, such that if $\theta_0\in\Gg$, we have that
\begin{align}\label{uniform bound on theta from the attractor} 
\|\theta(\cdot,t)\|_{H^{1+\frac{\gamma}{2}}}\le M_{\Gg},\qquad\forall t\ge0,
\end{align}
and
\begin{align}\label{uniform bound on time integral of theta from the attractor}
\frac{1}{T}\int_t^{t+T}\|\theta(\cdot,\tau)\|_{H^{1+\gamma}}d\tau\le M_{\Gg},\qquad \mbox{$\forall t\ge0$ and $T>0$,}
\end{align}
where $\theta(\cdot,t)=\pin(t)\theta_0$.
\end{theorem}

Theorem~\ref{Existence of H1 global attractor} will be proved in a sequence of lemmas. In the following lemma, we first show the existence of an $L^\infty$-absorbing set by using the bound \eqref{infty bound on theta with kappa>0 and t>1} on $\|\theta(t)\|_{L^\infty}$ for $t\ge1$.

\begin{lemma}[Existence of an $L^\infty$-absorbing set]\label{Existence of an L infty absorbing set lemma}
Let $c_0>0$ be defined in Proposition~\ref{boundedness on theta with kappa>0 prop}. Then the set 
\begin{align*}
B_{\infty}=\left\{\phi\in L^\infty\cap H^1:\|\phi\|_{L^\infty}\le\frac{2}{c_0\kappa}\|S\|_{L^\infty}\right\}
\end{align*}
is an absorbing set for $\pin(t)$. Moreover, using \eqref{infty bound on theta with kappa>0}, we have
\begin{align}\label{bound on solution map in L infty}
\sup_{t\ge0}\sup_{\theta_0\in B_{\infty}}\|\pin(t)\theta_0\|_{L^\infty}\le\frac{3\|S\|_{L^\infty}}{c_0\kappa}.
\end{align}
\end{lemma}

\begin{proof}
The proof follows by the same argument given in \cite[Theorem~3.1]{CZV16}. For a fixed bounded set $B\subset H^1$, we let 
\[R=\sup_{\phi\in B}\|\phi\|_{H^1}.\]
Using the bound \eqref{infty bound on theta with kappa>0 and t>1} and the Poincar\'{e} inequality, we conclude that if $\theta_0\in B$, then
\begin{align}
\|\pin(t)\theta_0\|_{L^\infty}\le\frac{C}{\kappa}\left[R+\frac{\|S\|_{L^\infty}}{\kappa^\frac{1}{2}}\right]e^{-c_0\kappa t}+\frac{\|S\|_{L^\infty}}{c_0\kappa},\qquad\forall t\ge1.
\end{align}
Choose $t_B=t_B(R,\|S\|_{L^2\cap L^\infty})\ge1$ such that
\begin{align*}
\frac{C}{\kappa}\left[R+\frac{\|S\|_{L^\infty}}{\kappa^\frac{1}{2}}\right]e^{-c_0\kappa t_B}\le\frac{\|S\|_{L^\infty}}{c_0\kappa},
\end{align*}
then we have $\pin(t)\theta_0\in B_{\infty}$ and hence $B_\infty$ is absorbing.
\end{proof}

Next we prove the following lemma which gives the necessary {\it a priori} estimate in $C^{\alpha}$-space with some appropriate exponent $\alpha\in(0,1)$. As pointed out in \cite{CZV16}, in view of Lemma~\ref{Existence of an L infty absorbing set lemma}, we can see that the solutions to \eqref{abstract active scalar eqn} emerging from data in a bounded subset of $H^1$ are absorbed in finite time by $B_{\infty}$. Hence in proving Lemma~\ref{Estimates in C alpha-space lemma}, we can assume that $\theta_0\in L^\infty$ and derive {\it a priori} bounds in terms of $\|\theta_0\|_{L^\infty}$.

\begin{lemma}[Estimates in $C^\alpha$-space]\label{Estimates in C alpha-space lemma}
Assume that $\theta_0\in H^1\cap L^\infty$ and fix $\nu$, $\kappa>0$. There exists $\alpha\in(0,\frac{\gamma}{3+\gamma}]$ which depends on $\|\theta_0\|_{L^\infty}$, $\|S\|_{L^\infty}$, $\nu$, $\kappa$, $\gamma$ such that
\begin{align}
\|\theta(\cdot,t)\|_{C\alpha}\le C(\K+\bK),\qquad\forall t\ge t_{\alpha}:=\frac{3}{2\gamma(1-\alpha)},
\end{align}
where $C>0$ is a positive constant, $\K$ and $\bK$ are given respectively by
\begin{align}\label{def of K and bar K}
 \left\{ \begin{array}{l}
\K:=\|\theta_0\|_{L^\infty}+\frac{1}{c_0\kappa}\|S\|_{L^\infty},\\
\bK:=\kappa^{-\frac{1}{\gamma}}\K+\|S\|^\frac{2+\gamma}{2(1+\gamma)}_{L^\infty}\kappa^{-\frac{1}{2(1+\gamma)}}\K^{\frac{\gamma}{2(1+\gamma)}}+\kappa^{-\frac{1}{\gamma}}\K^{\frac{6+\gamma}{4}}
\end{array}\right.
\end{align}
\end{lemma}

\begin{remark}
Using the bound \eqref{infty bound on theta with kappa>0}, under the assumption that $\theta_0\in L^\infty$, for $\kappa>0$, we have
\begin{align}\label{infty bound on theta using K for kappa>0}
\|\theta(\cdot,t)\|_{L^\infty}\le \K,\qquad \forall t\ge0.
\end{align}
\end{remark}

\begin{proof}[Proof of Lemma~\ref{Estimates in C alpha-space lemma}]
The idea of the proof follows from \cite[Theorem~4.1]{CZV16}. As suggested in \cite{CTV14}, \cite{CZV16}, we introduce the following finite difference
\begin{align*}
\Dh\theta(x,t)=\theta(x+h)-\theta(x,t),
\end{align*}
where $x$, $h\in\T^d$. Then $\Dh\theta$ satisfies
\begin{align}\label{eqn for theta with operator Lh}
\Lh(\Dh\theta)^2+D[\Dh\theta]=2(\delta_h S)(\delta_h \theta),
\end{align}
where $\Lh$ is the operator given by $\Lh:=\dt+u\cdot\nabla+(\Dh u)\cdot\nabla_h+\Lambda^\gamma$, and $D[\Dh\theta](x)$ is the function
\begin{align*}
D[\Dh\theta](x)=\int_{\R^d}\frac{[\Dh\theta(x)-\Dh\theta(x+y)]^2}{|y|^{d+\gamma}}dy.
\end{align*}
Let $\xi(t)=\xi:[0,\infty)\to[0,\infty)$ be a bounded decreasing differentiable function which will be defined later, and for $\alpha\in(0,\frac{\gamma}{3+\gamma}]$, we define $v$ by
\begin{align*}
v(x,t;h)=\frac{|\Dh\theta(x,t)|}{(\xi(t)^2+|h|^2)^\frac{\alpha}{2}}.
\end{align*}
Using \eqref{eqn for theta with operator Lh}, we can see that $v$ satisfies the following inequality
\begin{align}\label{ineq for v with operator Lh}
&\Lh v^2+\kappa\cdot\frac{D[\Dh\theta(x,t)]}{(\xi^2+|h|^2)^\alpha}\notag\\
&\le 2\alpha|\dot{\xi}|\frac{\xi}{\xi^2+|h|^2}v^2+2\alpha\frac{|h|}{\xi^2+|h|^2}|\Dh u|v^2+\frac{2\|S\|_{L^\infty}}{(\xi^2+|h|^2)^\frac{\alpha}{2}}v.
\end{align}
We estimate the term in \eqref{ineq for v with operator Lh} as follows. First, using the similar argument given in \cite{CTV14}, it can be shown that for $r\ge4|h|$, there exists some constant $C\ge1$ such that
\begin{align*}
D[\Dh\theta](x)\ge\frac{1}{2r^\gamma}|\Dh\theta(x)|^2-C|\Dh\theta(x)|\|\theta\|_{L^\infty}\frac{|h|}{r^{1+\gamma}}.
\end{align*}
Choose $r\ge4(\xi^2+|h|^2)^\frac{1}{2}\ge 4|h|$ such that
\begin{align*}
r=\frac{4C\|\theta\|_{L^\infty}}{|\Dh\theta(x)|}(\xi^2+|h|^2)^{\frac{(1-\alpha)(2+\gamma)}{6}+\frac{\alpha}{2}},
\end{align*}
then we have
\begin{align}\label{lower bound for Dh alpha}
\frac{D[\Dh\theta](x,t)}{(\xi^2+|h|^2)^\alpha}\ge\frac{|v(x,t;h)|^{2+\gamma}}{c_2\|\theta\|^\gamma_{L^\infty}(\xi^2+|h|^2)^\frac{\gamma(1-\alpha)(2+\gamma)}{6}},
\end{align}
for some positive constant $c_2$. For the term $2\alpha|\dot{\xi}|\frac{\xi}{\xi^2+|h|^2}v^2$, we take $\xi$ such that $\xi$ solves the following ordinary differential equation
\begin{align}\label{def of xi 1}
\dot{\xi}=-\xi^a,\qquad \xi(0)=1,
\end{align}
where $a=1-\frac{2}{3}\gamma(1-\alpha)$. More explictly, $\xi$ can be given by
\begin{equation} \label{def of xi 2}
			\xi(t)=\begin{cases}
				[1-t(1-a)]^\frac{1}{1-a},\qquad &t\in[0,t_{\alpha}] \\ 
				0,\qquad &t\in(t_{\alpha},\infty),
			\end{cases}
\end{equation}
where $t_{\alpha}$ is given by
\begin{align}\label{def of t alpha}
t_{\alpha}=\frac{3}{2\gamma(1-\alpha)}.
\end{align}
Hence for $\alpha\le\frac{\gamma}{3+\gamma}$, we have
\begin{align*}
2\alpha|\dot{\xi}|\frac{\xi}{\xi^2+|h|^2}v^2\le \frac{2\gamma}{3+\gamma}\frac{v^2}{(\xi^2+|h|^2)^\frac{\gamma(1-\alpha)}{3}}
\end{align*}
Using Young's inequality, we can further obtain
\begin{align}\label{bound on the term with derivative on xi}
2\alpha|\dot{\xi}|\frac{\xi}{\xi^2+|h|^2}v^2\le \frac{\kappa|v|^{2+\gamma}}{8c_2\|\theta\|^\gamma_{L^\infty}(\xi^2+|h|^2)^\frac{\gamma(1-\alpha)(2+\gamma)}{6}}+C\kappa^{-\frac{2}{\gamma}}\|\theta\|^2_{L^\infty},
\end{align}
for some positive constant $C$. For the term $\frac{2\|S\|_{L^\infty}}{(\xi^2+|h|^2)^\frac{\alpha}{2}}v$, using Young's inequality, we have
\begin{align*}
&\frac{2\|S\|_{L^\infty}}{(\xi^2+|h|^2)^\frac{\alpha}{2}}v\\
&\le\frac{\kappa|v|^{2+\gamma}}{8c_2\|\theta\|^\gamma_{L^\infty}(\xi^2+|h|^2)^\frac{\gamma(1-\alpha)(2+\gamma)}{6}}+C(\xi^2+|h|^2)^b\|S\|^\frac{2+\gamma}{1+\gamma}_{L^\infty}\kappa^{-\frac{1}{1+\gamma}}\|\theta\|^{\frac{\gamma}{1+\gamma}}_{L^\infty},
\end{align*}
where $b=\Big(\frac{\gamma(1-\alpha)}{6}-\frac{\alpha}{2}\Big)\Big(\frac{2+\gamma}{1+\gamma}\Big)$ and note that $b\ge0$ since $\alpha\le\frac{\gamma}{3+\gamma}$ and $\gamma\in(0,2]$. Since $\xi$, $|h|\le1$, we infer that
\begin{align}\label{bound on the term with S}
\frac{2\|S\|_{L^\infty}}{(\xi^2+|h|^2)^\frac{\alpha}{2}}v\le\frac{\kappa|v|^{2+\gamma}}{8c_2\|\theta\|^\gamma_{L^\infty}(\xi^2+|h|^2)^\frac{\gamma(1-\alpha)(2+\gamma)}{6}}+C\|S\|^\frac{2+\gamma}{1+\gamma}_{L^\infty}\kappa^{-\frac{1}{1+\gamma}}\|\theta\|^{\frac{\gamma}{1+\gamma}}_{L^\infty}.
\end{align}
Similarly, for the term $2\alpha\frac{|h|}{\xi^2+|h|^2}|\Dh u|v^2$, we apply Young's inequality again to obtain
\begin{align}\label{bound on the term with nabla u}
2\alpha\frac{|h|}{\xi^2+|h|^2}|\Dh u|v^2&\le2\alpha\frac{|h|^2}{\xi^2+|h|^2}\|\nabla u\|_{L^\infty}v^2\notag\\
&\le \frac{\kappa|v|^{2+\gamma}}{8c_2\|\theta\|^\gamma_{L^\infty}(\xi^2+|h|^2)^\frac{\gamma(1-\alpha)(2+\gamma)}{6}}+C\|\nabla v\|^\frac{2+\gamma}{2}_{L^\infty}\kappa^{-\frac{2}{\gamma}}\|\theta\|^2_{L^\infty}.
\end{align}
Apply the bounds \eqref{lower bound for Dh alpha}, \eqref{bound on the term with derivative on xi}, \eqref{bound on the term with S} and \eqref{bound on the term with nabla u} on \eqref{ineq for v with operator Lh}, we obtain
\begin{align*}
&\Lh v^2+\kappa\cdot\frac{D[\Dh\theta(x,t)]}{(\xi^2+|h|^2)^\alpha}+\frac{\kappa|v|^{2+\gamma}}{8c_2\|\theta\|^\gamma_{L^\infty}(\xi^2+|h|^2)^\frac{\gamma(1-\alpha)(2+\gamma)}{6}}\\
&\le C\Big[\kappa^{-\frac{2}{\gamma}}\|\theta\|^2_{L^\infty}+\|S\|^\frac{2+\gamma}{1+\gamma}_{L^\infty}\kappa^{-\frac{1}{1+\gamma}}\|\theta\|^{\frac{\gamma}{1+\gamma}}_{L^\infty}+\|\nabla v\|^\frac{2+\gamma}{2}_{L^\infty}\kappa^{-\frac{2}{\gamma}}\|\theta\|^2_{L^\infty}\Big].
\end{align*}
Applying the bound \eqref{infty bound on theta using K for kappa>0} on $\|\theta\|_{L^\infty}$, we have
\begin{align*}
\|\theta\|_{L^\infty}\le\K,
\end{align*}
and following the estimate \eqref{L infty bound on nabla u} on $\Lambda u$, we further obtain
\begin{align}\label{L infty bound on nabla u general}
\|\nabla u\|_{L^\infty}\le C\|\theta\|_{L^\infty}\le C\K.
\end{align}
Together with the fact that $\|S\|_{L^\infty}\le c_0\kappa\K$, it implies that
\begin{align*}
\Lh v^2+\frac{\kappa|v|^{2+\gamma}}{8c_2\|\theta\|^\gamma_{L^\infty}(\xi^2+|h|^2)^\frac{\gamma(1-\alpha)(2+\gamma)}{6}}\le C\bK^2,
\end{align*}
where $\bK$ is defined in \eqref{def of K and bar K}. Since $(\xi^2+|h|^2)^\frac{\gamma(1-\alpha)(2+\gamma)}{6}\le 2^\frac{\gamma(1-\alpha)(2+\gamma)}{6}\le 4$, we conclude that
\begin{align}\label{ineq for v with operator Lh step 2}
\Lh v^2+\frac{\kappa|v|^{2+\gamma}}{32c_2\K^\gamma}\le C\bK^2.
\end{align}
Take $\psi(t)=\|v(t)\|^2_{L^\infty_{x,h}}$, then because of \eqref{ineq for v with operator Lh step 2}, $\psi$ satisfies the following differential inequality
\begin{align}\label{diff ineq for psi}
\frac{d}{dt}\psi+\frac{\kappa\psi^{1+\frac{\gamma}{2}}}{32c_2\K^\gamma}\le C\bK^2.
\end{align}
Moreover, $\psi(0)$ satisfies the bound given by
\begin{align*}
\psi(0)\le\frac{4\|\theta_0\|^2_{L^\infty}}{\xi(0)^{2\alpha}}\le 4\K^2,
\end{align*}
and hence it follows from \eqref{diff ineq for psi} that 
\begin{align}\label{bound on psi all time}
\psi(t)\le C(\K+\bK),\qquad\forall t\ge0,
\end{align}
for some sufficiently large constant $C>0$. in view of \eqref{bound on psi all time}, we prove that 
\begin{align}
[\theta(t)]^2_{C^\alpha}\le C(\K+\bK),\qquad \forall t\ge t_{\alpha},
\end{align}
where $t_{\alpha}$ is given by \eqref{def of t alpha}. Finally, together with the bound \eqref{infty bound on theta using K for kappa>0}, we obtain
\begin{align}
\|\theta(t)\|_{C^\alpha}=\|\theta(t)\|_{L^\infty}+[\theta(t)]^2_{C^\alpha}\le C(\K+\bK),\qquad \forall t\ge t_{\alpha},
\end{align}
which finishes the proof of Lemma~\ref{Estimates in C alpha-space lemma}.
\end{proof}

\begin{remark}
Given $\alpha\in(0,\frac{\gamma}{3+\gamma}]$, if we further assume that $\theta_0\in C^{\alpha}$, then following the similar argument given in the proof of Lemma~\ref{Estimates in C alpha-space lemma} as shown above (also see \cite{CTV14} for reference), we have
\begin{align}\label{bound C alpha norm of theta for all time}
\|\theta(t)\|_{C^\alpha}\le [\theta_0]_{C^\alpha}+C(\K+\bK),\qquad\forall t\ge0.
\end{align}
\end{remark}

With the help of Lemma~\ref{Estimates in C alpha-space lemma}, we obtain the following result which can be regarded as an improvement of the regularity of the absorbing set $B_{\infty}$ defined in Lemma~\ref{Existence of an L infty absorbing set lemma}. 

\begin{lemma}[Existence of an $C^\alpha$-absorbing set]\label{Existence of an C alpha absorbing set lemma}
There exists $\alpha\in(0,\frac{\gamma}{3+\gamma}]$ and a constant $C_{\alpha}=C_{\alpha}(\|S\|_{L^\infty},\alpha,\nu,\kappa,\gamma,\K,\bK)\ge1$ such that the set
\begin{align*}
B_{\alpha}=\left\{\phi\in C^{\alpha}\cap H^1:\|\phi\|_{C^\alpha}\le C_{\alpha}\right\}
\end{align*}
is an absorbing set for $\pin(t)$. Moreover, we have
\begin{align}\label{bound on solution map in C alpha}
\sup_{t\ge0}\sup_{\theta_0\in B_{\alpha}}\|\pin(t)\theta_0\|_{C^\alpha}\le 2C_{\alpha}.
\end{align}
\end{lemma}

\begin{proof}
It is enough to prove that the $L^\infty$-absorbing set $B_\infty$ given in Lemma~\ref{Existence of an L infty absorbing set lemma} is itself absorbed by $B_\alpha$. Take $\theta_0\in B_{\infty}$, then by using \eqref{bound on solution map in L infty}, we have
\begin{align*}
\|\pin(t)\theta_0\|_{L^\infty}\le\frac{3\|S\|_{L^\infty}}{c_0\kappa}\qquad\forall t\ge0.
\end{align*}
Therefore, if we define $C_\alpha$ by
\begin{align*}
&C_\alpha:=\max\left\{1,\frac{3\|S\|_{L^\infty}}{c_0\kappa}+\kappa^{-\frac{1}{\gamma}}\Big(\frac{3\|S\|_{L^\infty}}{c_0\kappa}\Big)\right.\\
&\qquad\qquad\qquad\left.+\|S\|^\frac{2+\gamma}{2(1+\gamma)}_{L^\infty}\kappa^{-\frac{1}{2(1+\gamma)}}\Big(\frac{3\|S\|_{L^\infty}}{c_0\kappa}\Big)^{\frac{\gamma}{2(1+\gamma)}}+\kappa^{-\frac{1}{\gamma}}\Big(\frac{3\|S\|_{L^\infty}}{c_0\kappa}\Big)^{\frac{6+\gamma}{4}}\right\}
\end{align*}
then $\pin(t)\theta_0\in B_{\alpha}$ for all $t\ge t_\alpha$. Since $t_\alpha$ only depends on $\kappa$, $\gamma$ and $\|S\|_{L^\infty}$, we conclude that $\pin(t)B_{\infty}\subset B_{\alpha}$ for all $t\ge t_{\alpha}$. Finally, the bound \eqref{bound on solution map in C alpha} follows immediately from \eqref{bound C alpha norm of theta for all time} and our choice of $C_\alpha$, and we finish the proof.
\end{proof}

We now proceed to prove the existence of a bounded absorbing set in $H^1$:

\begin{lemma}[Existence of an $H^1$-absorbing set]\label{Existence of an H1 absorbing set lemma}
There exists $\alpha\in(0,\frac{\gamma}{3+\gamma}]$ and a constant $R_1=R_1(\|S\|_{L^\infty\cap H^1},\alpha,\nu,\kappa,\gamma)\ge1$ such that the set
\begin{align*}
B_1=\{\phi\in C^{\alpha}\cap H^1:\|\phi\|^2_{H^1}+\|\phi\|^2_{C^\alpha}\le R^2_{1}\}
\end{align*}
is an absorbing set for $\pin(t)$. Moreover, we have
\begin{align}\label{bound on time integral on theta in H 1+gamma/2}
\sup_{t\ge0}\sup_{\theta_0\in B_1}\left[\|\pin(t)\theta_0\|^2_{H^1}+\|\pin(t)\theta_0\|^2_{C^\alpha}+\int_t^{t+1}\|\pin(\tau)\theta_0\|^2_{H^{1+\frac{\gamma}{2}}}d\tau\right]\le 2R_1^2.
\end{align}
\end{lemma}

\begin{proof}
As pointed out in \cite{CZV16}, it is enough to establish an {\it a priori} estimate for initial data in $H^1\cap C^\alpha$. Suppose that $\theta_0\in H^1\cap C^{\alpha}$. We apply $\nabla$ to \eqref{abstract active scalar eqn}$_1$ and take the inner product with $\nabla\theta$ to obtain
\begin{align}\label{eqn for nabla u with D[nabla theta]}
(\dt+u\cdot\nabla+\kappa\Lambda^\gamma)|\nabla\theta|^2+\kappa D[\nabla\theta]=-2\partial_{x_l} u_j\partial_{x_j}\theta\partial_{x_j}\theta+2\nabla S\cdot\nabla\theta,
\end{align}
where $D[\nabla\theta]$ is given by
\begin{align*}
D[\nabla\theta](x)=\int_{\R^d}\frac{[\nabla\theta(x)-\nabla\theta(x+y)]^2}{|y|^{d+\gamma}}dy.
\end{align*}
Since by \eqref{bound on solution map in C alpha}, we have $\|\theta(t)\|_{C^\alpha}\le 2C_\alpha$ where $C_\alpha$ is defined in Lemma~\ref{Existence of an C alpha absorbing set lemma}, hence we can apply \cite[Theorem~2.2]{CV12} to obtain
\begin{align}\label{lower bound for D[nabla theta]}
D[\nabla\theta]\ge\frac{|\nabla\theta|^{2+\frac{\gamma}{1-\alpha}}}{C[\theta]^\frac{\gamma}{1-\alpha}_{C^\alpha}}\ge\frac{|\nabla\theta|^{2+\frac{\gamma}{1-\alpha}}}{C(2C_\alpha)^\frac{\gamma}{1-\alpha}}.
\end{align}
Also, using Young's inequality, we have
\begin{align}\label{bound on terms involving nabla u}
|-2\partial_{x_l} u_j\partial_{x_j}\theta\partial_{x_j}\theta|\le \frac{\kappa}{4}\frac{|\nabla\theta|^{2+\frac{\gamma}{1-\alpha}}}{C(2C_\alpha)^\frac{\gamma}{1-\alpha}}+\Big(\frac{4C}{\kappa}\Big)^\frac{2-2\alpha}{\gamma}(2C_\alpha)^2|\nabla u|^{1+\frac{2(1-\alpha)}{\gamma}}.
\end{align}
We apply \eqref{lower bound for D[nabla theta]} and \eqref{bound on terms involving nabla u} on \eqref{eqn for nabla u with D[nabla theta]} to get
\begin{align}\label{ineq for nabla u with D[nabla theta]}
&(\dt+u\cdot\nabla+\kappa\Lambda^\gamma)|\nabla\theta|^2+\frac{\kappa}{2}D[\nabla\theta]\notag\\&\le\Big(\frac{4C}{\kappa}\Big)^\frac{2-2\alpha}{\gamma}(2C_\alpha)^2|\nabla u|^{1+\frac{2(1-\alpha)}{\gamma}}+2|\nabla S||\nabla \theta|.
\end{align}
Integrate \eqref{ineq for nabla u with D[nabla theta]} over $\T^d$ and using the identity
$\dis\frac{1}{2}\intox D[\nabla \theta](x)dx=\|\theta\|^2_{H^{1+\frac{\gamma}{2}}},$
we obtain
\begin{align}\label{ineq for nabla u with D[nabla theta] final}
&\frac{d}{dt}\|\theta\|^2_{H^1}+\frac{\kappa}{2}\|\theta\|^2_{H^{1+\frac{\gamma}{2}}}\notag\\
&\le\Big(\frac{4C}{\kappa}\Big)^\frac{2-2\alpha}{\gamma}(2C_\alpha)^2\intox|\nabla u|^{1+\frac{2(1-\alpha)}{\gamma}}+2\|S\|_{H^1}\|\theta\|_{H^1}\notag\\
&\le\Big(\frac{4C}{\kappa}\Big)^\frac{2-2\alpha}{\gamma}(2C_\alpha)^2\|\nabla u\|^{1+\frac{2(1-\alpha)}{\gamma}}_{L^\infty}+2\|S\|_{H^1}\|\theta\|_{H^1}\notag\\
&\le\Big(\frac{4C}{\kappa}\Big)^\frac{2-2\alpha}{\gamma}(2C_\alpha)^2\K^{1+\frac{2(1-\alpha)}{\gamma}}+\frac{8}{c_{\gamma,d}\kappa}\|S\|^2_{H^1}+\frac{\kappa}{4}\|\theta\|^2_{H^{1+\frac{\gamma}{2}}},
\end{align}
where the last inequality follows by the bound \eqref{L infty bound on nabla u general} and Young's inequality, and $c_{\gamma,d}>0$ is the dimensional constant for which it satisfies the bound
\begin{align}\label{def of c gamma d}
\|\theta\|^2_{H^1}\le(c_{d,\gamma})^{-1}\|\theta\|^2_{H^{1+\frac{\gamma}{2}}}.
\end{align}
If we choose $K_1\ge1$ such that
\begin{align*}
K_1=\frac{8}{c_{\gamma,d}\kappa}\Big[\Big(\frac{4C}{\kappa}\Big)^\frac{2-2\alpha}{\gamma}(2C_\alpha)^2\K^{1+\frac{2(1-\alpha)}{\gamma}}+\frac{8}{c_{\gamma,d}\kappa}\|S\|^2_{H^1}\Big],
\end{align*}
then by applying Gr\"{o}nwall's inequality on \eqref{ineq for nabla u with D[nabla theta] final}, we conclude that
\begin{align}\label{bound on H^1 in terms of initial H^1 data}
\|\theta(t)\|^2_{H^1}\le\|\theta_0\|^2_{H^1}e^{-\frac{c_{\gamma,d}}{8}\kappa t}+K_1.
\end{align}
Now we define $R_1=2K_1$, then it is straight forward to see that the set $B_1$ is an absorbing set for $\pin(t)$. Moreover, upon integrating \eqref{ineq for nabla u with D[nabla theta] final} on the time interval $(t,t+1)$ and applying \eqref{bound on H^1 in terms of initial H^1 data}, we further obtain \eqref{bound on time integral on theta in H 1+gamma/2} and the proof is complete.
\end{proof}

Once Lemma~\ref{Existence of an H1 absorbing set lemma} is established, we can further improve the regularity of the absorbing set $B_1$ to $H^{1+\frac{\gamma}{2}}$, which is illustrated in the next lemma. 

\begin{lemma}[Existence of an $H^{1+\frac{\gamma}{2}}$-absorbing set]\label{Existence of an H 1+gamma/2 absorbing set lemma}
There exists a constant $R_{1+\frac{\gamma}{2}}\ge1$ which depends on $\|S\|_{L\infty\cap H^1}$, $\nu$, $\kappa$, $\gamma$ such that the set
\begin{align*}
B_{1+\frac{\gamma}{2}}=\left\{\phi\in H^{1+\frac{\gamma}{2}}:\|\phi\|_{H^{1+\frac{\gamma}{2}}}\le R_{1+\frac{\gamma}{2}}\right\}
\end{align*}
is an absorbing set for $\pin(t)$. Moreover
\begin{align}\label{bound on theta in H 1+gamma/2}
\sup_{t\ge0}\sup_{\theta_0\in B_{1+\frac{\gamma}{2}}}\|\pin(t)\theta_0\|_{H^{1+\frac{\gamma}{2}}}\le 2R_{1+\frac{\gamma}{2}}.
\end{align}
\end{lemma}

\begin{proof}
Similar to the previous cases, it is enough to show that $B_{1+\frac{\gamma}{2}}$ absorbs the $H^1$-absorbing set $B_1$ obtained in Lemma~\ref{Existence of an H1 absorbing set lemma}. Suppose $\theta_0\in B_{1}$, then from \eqref{bound on time integral on theta in H 1+gamma/2}, we have
\begin{align}\label{bound on time integral for theta in H 1+gamma/2 from t to t+1}
\sup_{t\ge0}\int_{t}^{t+1}\|\pin(\tau)\theta\|^2_{H^{1+\frac{\gamma}{2}}}d\tau\le 2R_1^2.
\end{align}
Following the argument given in the proof of Theorem~\ref{global-in-time wellposedness in Sobolev} and using the bound \eqref{infty bound on theta using K for kappa>0} on $\|\theta^\kappa\|_{L^\infty}$, for $t\ge0$, we arrive at
\begin{align}\label{differential inequality on theta 1+gamma}
&\frac{d}{dt}\|\Lambda^{1+\frac{\gamma}{2}}\theta^\kappa\|^2_{L^2}+\kappa\|\Lambda^{1+\gamma}\theta^\kappa\|^2_{L^2}\notag\\
&\le C\|\theta^\kappa\|_{L^\infty}\|\Lambda^{1+\frac{\gamma}{2}}\theta^\kappa\|^2_{L^2}+2\|\Lambda^{1+\frac{\gamma}{2}}S\|_{L^2}\|\Lambda^{1+\frac{\gamma}{2}}\theta^\kappa\|_{L^2}\notag\\
&\le C\K\|\Lambda^{1+\frac{\gamma}{2}}\theta^\kappa\|^2_{L^2}+\frac{4}{c_{\gamma,d}\kappa}\|S\|^2_{H^{1+\frac{\gamma}{2}}}+\frac{\kappa}{2}\|\theta\|^2_{H^{1+\gamma}},
\end{align}
where $c_{\gamma,d}$ is defined in \eqref{def of c gamma d}. Hence using the local integrability \eqref{bound on time integral for theta in H 1+gamma/2 from t to t+1} and the uniform Gr\"{o}nwall's lemma (see \cite[Lemma~C.1]{CTV14} for example), we obtain
\begin{align}\label{bound on theta in H 1+gamma/2 for t>1}
\|\pin(t)\theta_0\|^2_{H^{1+\frac{\gamma}{2}}}\le\Big[2C\K R_1^2+\frac{4}{c_{\gamma,d}\kappa}\|S\|^2_{H^{1+\frac{\gamma}{2}}}\Big]e^{C\K R_1^2},\qquad \forall t\ge1.
\end{align}
By setting 
\begin{align*}
R_{1+\frac{\gamma}{2}}:=\Big[2C\K R_1^2+\frac{4}{c_{\gamma,d}\kappa}\|S\|^2_{H^{1+\frac{\gamma}{2}}}\Big]^\frac{1}{2}e^{\frac{C\K R_1^2}{2}},
\end{align*}
we conclude that $\pin(t) B_1\subset B_{1+\frac{\gamma}{2}}$ for all $t\ge1$. Since $B_{1+\frac{\gamma}{2}}$ absorbs $B_1$, the bound \eqref{bound on theta in H 1+gamma/2} then follows  by \eqref{bound on time integral on theta in H 1+gamma/2} and \eqref{bound on theta in H 1+gamma/2 for t>1}, and we finish the proof.
\end{proof}

\begin{remark}
By combining \eqref{differential inequality on theta 1+gamma} with \eqref{bound on theta in H 1+gamma/2 for t>1}, if $\theta_0\in B_1$, then for $t\ge1$ and $T>0$, we further obtain
\begin{align}\label{bound on time integral on theta in H 1+gamma}
\frac{1}{T}\int_{t}^{t+T}\|\pin(\tau)\theta_0\|^2_{H^{1+\gamma}}d\tau\le R_{1+\gamma}^2,
\end{align}
where
\begin{align*}
R_{1+\gamma}:=\Big[C\K R_{1+\frac{\gamma}{2}}^2+\frac{4}{c_{\gamma,d}\kappa}\|S\|^2_{H^{1+\frac{\gamma}{2}}}\Big]^\frac{1}{2}.
\end{align*}
\end{remark}

The existence and regularity of the global attractor claimed by Theorem~\ref{Existence of H1 global attractor} now follows from Lemma~\ref{Existence of an H 1+gamma/2 absorbing set lemma} by applying the argument given in \cite[Proposition~8]{CCP12}, and the bounds \eqref{uniform bound on theta from the attractor}-\eqref{uniform bound on time integral of theta from the attractor} are immediate consequences of \eqref{bound on theta in H 1+gamma/2} and \eqref{bound on time integral on theta in H 1+gamma} by taking $M_{\Gg}:=\max\{2R_{1+\frac{\gamma}{2}},R_{1+\gamma}\}$. We summarise the properties of the global attractor $\Gg$ as claimed by Theorem~\ref{Existence of H1 global attractor}:
\begin{corollary}\label{Unique attractor general corollary}
The solution map $\pin(t):H^1\to H^1$ associated to \eqref{abstract active scalar eqn} possesses a unique global attractor $\Gg$ with the following properties:
\begin{itemize}
\item $\Gg\subset H^{1+\frac{\gamma}{2}}$ and is the $\omega$-limit set of $B_{1+\frac{\gamma}{2}}$, namely
\begin{align*}
\Gg=\omega(B_{1+\frac{\gamma}{2}})=\bigcap_{t\ge0}\overline{\bigcup_{\tau\ge t}\pin(\tau) B_{1+\frac{\gamma}{2}}}.
\end{align*}
\item For every bounded set $B\subset H^1$, we have
\begin{align*}
\lim_{t\to\infty}\dist(\pin(t)B,\Gg)=0,
\end{align*}
where $\dist$ stands for the usual Hausdorff semi-distance between sets given by the $H^1$-norm.
\item $\Gg$ is minimal in the class of $H^1$-closed attracting set.
\end{itemize}
\end{corollary}

\subsection{Further properties for the global attractors}\label{Further properties for the attractors subsection}

In this subsection, we prove some additional properties for the global attractors obtained in Theorem~\ref{Existence of H1 global attractor}. Recall that we have $d=2$ or $d=3$, and throughout this subsection, we impose an extra condition on the exponent $\gamma$, namely
\begin{align}\label{conditions on gamma}
\gamma\in[1,2].
\end{align}
%Notice that assumption \eqref{conditions on gamma} the cases when $\gamma\in[1,2]$ and $d\in\{1,2,3\}$, which are applicable to many physical models. 
Under the assumption \eqref{conditions on gamma}, our goal is to prove the following theorem: 
\begin{theorem}
Let $S\in L^\infty\cap H^1$. For $\nu$, $\kappa>0$, assume that the exponent $\gamma$ satisfies \eqref{conditions on gamma}. Then the global attractor $\Gg$ of $\pin(t)$ further enjoys the following properties:
\begin{itemize}
\item $\Gg$ is fully invariant, namely
\begin{align*}
\pin(t)\Gg=\Gg,\qquad \forall t\ge0.
\end{align*}
\item $\Gg$ is maximal in the class of $H^1$-bounded invariant sets.
\item $\Gg$ has finite fractal dimension.
\end{itemize}
\end{theorem}

We first give the following auxiliary estimates on $u$ and $\theta$ under the assumption \eqref{conditions on gamma}. These estimates will be useful for the later analysis.

\begin{lemma}
Assume that $\gamma$ satisfies the assumption \eqref{conditions on gamma}, then if $f\in H^{1+\frac{\gamma}{2}}$ and $g$, $\theta\in H^1$, we have 
\begin{align}\label{bound on Lambda theta for gamma>1}
\|\Lambda^{2-\frac{\gamma}{2}}f\|_{L^2}\le C\|\Lambda^{1+\frac{\gamma}{2}}f\|_{L^2},
\end{align}
\begin{align}\label{bound on L4 using H1}
\|g\|_{L^4}\le C\|\Lambda g\|_{L^2},
\end{align}
\begin{align}\label{bound on u in terms of nabla theta}
\|u\|_{L^\infty}\le C\|\Lambda^\frac{\gamma}{2}\theta\|_{L^2}\le C_\nu\|\Lambda\theta\|_{L^2},
\end{align}
\begin{align}\label{bound on nabla u in terms of nabla theta}
\|\nabla u\|_{L^\infty}\le C_\nu\|\Lambda^{1+\frac{\gamma}{2}}\theta\|_{L^2},
\end{align}
where $u=u[\theta]$, $C$ is a positive constant which depends on $d$ only, and $C_\nu$ is a positive constant which depends on $d$ and $\nu$.
\end{lemma}

\begin{proof}
The bound \eqref{bound on Lambda theta for gamma>1} follows immediately by observing that $2-\frac{\gamma}{2}\le1+\frac{\gamma}{2}$ when $\gamma\ge1$. To prove \eqref{bound on L4 using H1}, using \eqref{Sobolev inequality} and the assumption $d\in\{2,3\}$, we have
\begin{align*}
\|g\|_{L^4}\le C\|\Lambda^\frac{d}{4}g\|_{L^2}\le C\|\Lambda g\|_{L^2}.
\end{align*}
To prove \eqref{bound on u in terms of nabla theta}, by using the bound \eqref{L infty bound 2} for $q=d+1$ and together with \eqref{two order smoothing for u when nu>0}, we have
\begin{align*}
\|u\|_{L^\infty}=\|u[\theta]\|_{L^\infty}&\le C\|u\|_{W^{1,d+1}}\\
&\le C\|\Lambda^{\frac{d}{2}-\frac{d}{d+1}+1}u\|_{L^{d+1}}\\
&\le C_\nu\|\Lambda^{\frac{d}{2}-\frac{d}{d+1}-1}\theta\|_{L^{d+1}}\\
&\le C_\nu\|\Lambda^\frac{\gamma}{2}\theta\|_{L^2}\le C_\nu\|\Lambda\theta\|_{L^2},
\end{align*}
where the last inequality follows since $\frac{d}{2}-\frac{d}{d+1}\le1+\frac{\gamma}{2}\le2$ for $\gamma\le2$ and $d\in\{2,3\}$. To prove \eqref{bound on nabla u in terms of nabla theta} we have
\begin{align*}
\|\nabla u\|_{L^\infty}&\le C\|\nabla u\|_{W^{1,d+1}}\notag\\
&\le C\|\Lambda^{\frac{d}{2}-\frac{d}{d+1}+2}u\|_{L^{d+1}}\notag\\
&\le C_\nu\|\Lambda^{\frac{d}{2}-\frac{d}{d+1}}\theta\|_{L^{d+1}}\le C_\nu\|\Lambda^{1+\frac{\gamma}{2}}\theta\|_{L^2},
\end{align*}
provided that $\frac{d}{2}-\frac{d}{d+1}\le1+\frac{\gamma}{2}$ holds.
\end{proof}

We now show that the solution map $\pin(t)$ is indeed continuous in the $H^1$-topology. More precisely, we have the following lemma:

\begin{lemma}[Continuity of $\pin(t)$]\label{continuity of solution map lemma}
Let $B_{1+\frac{\gamma}{2}}$ be the absorbing set for $\pin$ defined in Lemma~\ref{Existence of an H 1+gamma/2 absorbing set lemma}. For every $t>0$, the solution map $\pin(t):B_{1+\frac{\gamma}{2}}\to H^1$ is continuous in the topology of $H^1$.
\end{lemma}

\begin{proof}
We fix $t>0$. Let $\theta_0$, $\tilde\theta_0\in B_{1+\frac{\gamma}{2}}$ be arbitrary such that $\theta=\pin(t)\theta_0$ and $\tilde\theta=\pin(t)\tilde\theta_0$ with
\begin{align}\label{bound on theta and tilde theta}
\|\theta(t)\|_{H^{1+\frac{\gamma}{2}}}\le 2R_{1+\frac{\gamma}{2}}, \qquad\|\tilde\theta(t)\|_{H^{1+\frac{\gamma}{2}}}\le 2R_{1+\frac{\gamma}{2}},\qquad\forall t\ge0.
\end{align}
Denote the difference by $\bar{\theta}=\theta-\tilde\theta$ with $\bar{\theta}_0=\theta_0-\tilde\theta_0$, then $\bar{\theta}$ satisfies the following equation
\begin{align}\label{differential eqn for bar theta}
\dt\bar{\theta}+\kappa\Lambda^\gamma\bar{\theta}-\bar{u}\cdot\nabla\bar{\theta}+\tilde u\cdot\nabla\bar{\theta}+\bar{u}\cdot\nabla\tilde\theta=0,
\end{align}
where $\tilde u=u[\tilde\theta]$ and $\bar{u}=u[\bar{\theta}]$. Multiply \eqref{differential eqn for bar theta} by $-\Delta\bar{\theta}$ and integrate,
\begin{align}\label{differential ineq for bar theta step 1}
&\frac{1}{2}\frac{d}{dt}\|\bar\theta\|^2_{H^1}+\kappa\|\bar\theta\|^2_{H^{1+\frac{\gamma}{2}}}\notag\\
&\le\Big|\intox(\nabla\bar{u}\cdot\nabla\bar{\theta})\cdot\nabla\bar{\theta}\Big|+\Big|\intox(\nabla\tilde u\cdot\nabla\bar{\theta})\cdot\nabla\bar{\theta}\Big|+\Big|\intox(\bar{u}\cdot\nabla\tilde\theta)(-\Delta\bar{\theta})\Big|\notag\\
&\le(\|\nabla\bar{u}\|_{L^\infty}+\|\nabla\tilde{u}\|_{L^\infty})\|\nabla\bar{\theta}\|^2_{L^2}+\Big|\intox(\bar{u}\cdot\nabla\tilde\theta)(-\Delta\bar{\theta})\Big|.
\end{align}
Using \eqref{bound on nabla u in terms of nabla theta} and \eqref{bound on theta and tilde theta}, we can bound $\|\nabla\bar{u}\|_{L^\infty}$ and $\|\nabla\tilde{u}\|_{L^\infty}$ as follows:
\begin{align*}
(\|\nabla\bar{u}\|_{L^\infty}+\|\nabla\tilde{u}\|_{L^\infty})\le C\Big(\|\Lambda^{1+\frac{\gamma}{2}}\theta\|_{L^2}+\|\Lambda^{1+\frac{\gamma}{2}}\tilde\theta\|_{L^2}\Big)\le CR_{1+\frac{\gamma}{2}}.
\end{align*}
Next we estimate the last term on the right side of \eqref{differential ineq for bar theta step 1}. Upon integrating by parts and using the product estimate \eqref{product estimate}, it can be bounded by
\begin{align}\label{bound on the cross term bar theta}
\Big|\intox(\bar{u}\cdot\nabla\tilde\theta)(-\Delta\bar{\theta})\Big|&=\Big|\intox\bar{u}\tilde\theta\cdot\nabla\Delta\bar{\theta}\Big|\notag\\
&\le\|\Lambda^{2-\frac{\gamma}{2}}(\bar{u}\tilde\theta)\|_{L^2}\|\Lambda^{1+\frac{\gamma}{2}}\bar{\theta}\|_{L^2}\notag\\
&\le C\Big(\|\bar{u}\|_{L^\infty}\|\Lambda^{2-\frac{\gamma}{2}}\tilde\theta\|_{L^2}+\|\tilde\theta\|_{L^4}\|\Lambda^{2-\frac{\gamma}{2}}\bar{u}\|_{L^4}\Big)\|\Lambda^{1+\frac{\gamma}{2}}\bar{\theta}\|_{L^2}.
\end{align}
Using \eqref{two order smoothing for u when nu>0} and \eqref{bound on L4 using H1}, together with the bounds given in \eqref{bound on theta and tilde theta}, we have
\begin{align*}
&\|\tilde\theta\|_{L^4}\|\Lambda^{2-\frac{\gamma}{2}}\bar{u}\|_{L^4}\|\Lambda^{1+\frac{\gamma}{2}}\bar{\theta}\|_{L^2}\\
&\le C\|\Lambda\tilde\theta\|_{L^2}\|\bar\theta\|_{L^4}\|\Lambda^{1+\frac{\gamma}{2}}\bar{\theta}\|_{L^2}\\
&\le C\|\Lambda\tilde\theta\|_{L^2}\|\Lambda\bar\theta\|_{L^2}\|\Lambda^{1+\frac{\gamma}{2}}\bar{\theta}\|_{L^2}\le CR_{1+\frac{\gamma}{2}}\|\Lambda\bar\theta\|_{L^2}\|\Lambda^{1+\frac{\gamma}{2}}\bar{\theta}\|_{L^2}.
\end{align*}
On the other hand, using \eqref{bound on Lambda theta for gamma>1}, \eqref{bound on u in terms of nabla theta} and \eqref{bound on theta and tilde theta}, we readily have 
\begin{align*}
\|\Lambda^{2-\frac{\gamma}{2}}\tilde\theta\|_{L^2}\le C\|\Lambda^{1+\frac{\gamma}{2}}\tilde\theta\|_{L^2}\le CR_{1+\frac{\gamma}{2}},
\end{align*}
and
\begin{align*}
\|\bar u\|_{L^\infty}\le C\|\Lambda\bar\theta\|_{L^2},
\end{align*}
Hence we obtain from \eqref{bound on the cross term bar theta} that
\begin{align}\label{bound on the cross term bar theta final}
\Big|\intox(\bar{u}\cdot\nabla\tilde\theta)(-\Delta\bar{\theta})\Big|&\le CR_{1+\frac{\gamma}{2}}\|\Lambda\bar\theta\|_{L^2}\|\Lambda^{1+\frac{\gamma}{2}}\bar{\theta}\|_{L^2}
\end{align}
Apply \eqref{bound on the cross term bar theta final} on \eqref{differential ineq for bar theta step 1}, we deduce that
\begin{align}\label{differential ineq for bar theta final}
\frac{d}{dt}\|\bar\theta\|^2_{H^1}+\kappa\|\bar\theta\|^2_{H^{1+\frac{\gamma}{2}}}\le C\Big[R_{1+\frac{\gamma}{2}}+\frac{2}{\kappa}R_{1+\frac{\gamma}{2}}^2\Big]\|\nabla\bar{\theta}\|^2_{L^2}.
\end{align}
Hence we conclude from \eqref{differential ineq for bar theta final} that
\begin{align*}
\|\bar\theta(t)\|^2_{H^1}\le K_t\|\bar\theta_0\|^2_{H^1},
\end{align*}
where $K_t=\exp\Big(Ct\Big(R_{1+\frac{\gamma}{2}}+\frac{2}{\kappa}R_{1+\frac{\gamma}{2}}^2\Big)\Big)$, which implies that $\pin(t)$ is continuous in the $H^1$-topology.
\end{proof}

Following the argument given in \cite[Proposition 5.5]{CTV14} and using the log-convexity method introduced by \cite{AN67}, we can also prove that the solution map $\pin$ is injective on the absorbing set $B_{1+\frac{\gamma}{2}}$, which is illustrated in the following lemma:

\begin{lemma}[Backwards uniqueness]\label{Backwards uniqueness lemma}
Fix $\nu$, $\kappa>0$ and assume that $\gamma$ satisfies \eqref{conditions on gamma}. Let $\tsup[1]{\varphi}_0$, $\tsup[2]{\varphi}_0\in H^1$ be two initial data, and let
\begin{align*}
\tsup[1]{\varphi},\tsup[2]{\varphi}\in C([0, \infty); H^1)\cap L^2([0, \infty); H^{1+\frac{\gamma}{2}})
\end{align*}
be the corresponding solutions of the initial value problem \eqref{abstract active scalar eqn} for $\tsup[1]{\varphi}_0$ and $\tsup[2]{\varphi}_0$ respectively. If there exists $T > 0$ such that $\tsup[1]{\varphi}(\cdot,T) = \tsup[2]{\varphi}(\cdot,T)$, then $\tsup[1]{\varphi}_0= \tsup[2]{\varphi}_0$ holds.
\end{lemma}

\begin{proof}
Define $\theta=\tsup[1]{\varphi}-\tsup[2]{\varphi}$ and $\bar{\theta}=\frac{1}{2}(\tsup[1]{\varphi}+\tsup[2]{\varphi})$. Then $\theta$ and $\bar\theta$ satisfy
\begin{align*}
\dt \theta+\kappa\Lambda^{\gamma}\theta+\bar{u}\cdot\nabla \theta+u\cdot\nabla\bar{\theta}=0,\qquad \theta(\cdot,0)=\theta_0:=\tsup[1]{\varphi}_0-\tsup[2]{\varphi}_0,
\end{align*}
where $\bar u=u[\bar \theta]$ and $u=u[\theta]$. We argue by contradiction: suppose that $\theta_0\neq0$, then by continuity in time, we have $\|\theta(\cdot,t)\|_{L^2}>0$ for a sufficiently small $t>0$, and we define $\tau>0$ to be the minimal time such that 
\begin{align}\label{zero for bar theta}
\lim_{t\to\tau^{-}}\|\theta(\cdot,t)\|_{L^2}=0.
\end{align}
Let $\dis m=\max_{t\in[0,\tau)}\|\theta(\cdot,t)\|_{L^2}$, and we write $v(t)=\log\Big[\frac{2m}{\|\theta(\cdot,t)\|_{L^2}}\Big]$, then $v(t)$ is well-defined and positive on $[0,\tau)$ with $v(0)<\infty$, and $v(t)\to\infty$ as $t\to\tau^{-}$ by \eqref{zero for bar theta}. Notice that $v$ satisfies
\begin{align}\label{eqn for v}
\frac{d}{dt}v&=-\frac{1}{\|\theta\|^2_{L^2}}\intox\theta\dt\theta\le-\frac{\kappa}{\|\theta\|^2_{L^2}}\|\Lambda^\frac{\gamma}{2}\theta\|^2_{L^2}+\frac{1}{\|\theta\|^2_{L^2}}\Big|\intox u\cdot\nabla\bar\theta \theta\Big|.
\end{align}
Using \eqref{bound on u in terms of nabla theta}, we can bound $\Big|\intox u\cdot\nabla\bar\theta \theta\Big|$ as follows.
\begin{align*}
\Big|\intox u\cdot\nabla\bar\theta \theta\Big|&\le\|u\|_{L^\infty}\|\nabla\bar\theta\|_{L^2}\|\theta\|_{L^2}\\
&\le C\|\Lambda^\frac{\gamma}{2}\theta\|_{L^2}\|\nabla\bar\theta\|_{L^2}\|\theta\|_{L^2}.
\end{align*}
Hence using Cauchy-Schwartz inequality and integrating \eqref{eqn for v} in time, we obtain
\begin{align*}
v(t)\le v(0)+C\int_0^\tau\|\bar\theta(\cdot,s)\|^2_{H^1}ds<\infty,\qquad\forall t\in[0,\tau),
\end{align*}
Hence $\dis\lim_{t\to\tau^{-}}v(t)<\infty$, which leads to a contradiction. Therefore we must have $\theta_0=0$ and $\tsup[1]{\varphi}_0=\tsup[2]{\varphi}_0$ as desired.
\end{proof}

\begin{remark}
In view of the results obtained from Lemma~\ref{Backwards uniqueness lemma}, the solution map $\pin(t)$ is injective on $\Gg$. The dynamics, when restricted to $\Gg$, actually defines a dynamical system. Hence $\pin(t)\Big|_{\Gg}$ makes sense for ${\it all}$ $t\in\R$, not just for $t\ge0$.
\end{remark}

Applying Lemma~\ref{continuity of solution map lemma}-\ref{Backwards uniqueness lemma} and the argument given in \cite[Proposition~6.4]{CZV16}, we can obtain the invariance and the maximality of the attractor $\Gg$ stated in Theorem~\ref{Existence of H1 global attractor}. We summarise the results in the following corollary:

\begin{corollary}
The global attractor $\Gg$ of $\pin(t)$ is fully invariant, namely
\begin{align*}
\pin(t)\Gg=\Gg,\qquad \forall t\ge0.
\end{align*}
In particular, $\Gg$ is maximal among the class of bounded invariant sets in $H^1$.
\end{corollary}

Finally, we address the fractal dimensions for the global attractors $\Gg$ obtained in Theorem~\ref{Existence of H1 global attractor}. Given a compact set $X$, we give the following definition for fractal dimension $\dimf(X)$, which is based on counting the number of closed balls of a fixed radius $\varepsilon$ needed to cover $X$; see \cite{R13} for further explanation. 

\begin{definition}\label{fractal dimension def}
Given a compact set $X$, let $N(X,\varepsilon)$ be the minimum number of balls of radius $\varepsilon$ that cover $X$. The {\it fractal dimension} $\dimf (X)$ of $X$ is given by
\begin{align*}
\dimf(X) := \limsup_{\varepsilon\to0}\frac{\log N(X, \varepsilon)}{-\log\varepsilon}.
\end{align*}
\end{definition}

In order to prove that $\dimf (\Gg)$ is finite, we need to show that the solution map $\pin$ given in \eqref{def of solution map nu>0 and kappa>0} is uniform differentiable. More precisely, we have the following definition:

\begin{definition}\label{uniform differentiable def}
We say that $\pin(t)$ is {\it uniform differentiable} on $\Gg$ if for every $\theta_0\in\Gg$, there exists a linear operator $D\pin(t, \theta_0)$ such that
\begin{align}\label{condition 1 for uniform differentiable}
\mbox{$\dis\lim_{\varepsilon\to0}\sup_{\theta_0,\varphi_0\in\Gg;0<\|\theta_0-\varphi_0\|_{H^1}\le\varepsilon}\frac{\|\pin(t)(\varphi_0)-\pin(t)(\theta_0)-D\pin(t)(\varphi_0-\theta_0)\|_{H^1}}{\|\theta_0-\varphi_0\|_{H^1}}=0$,}
\end{align}
and 
\begin{align}\label{condition 2 for uniform differentiable}
\mbox{$\dis\sup_{\theta_0\in\Gg}\sup_{\psi_0\in H^1}\frac{\|D\pin(t,\theta_0)(\psi_0)\|_{H^1}}{\|\psi_0\|_{H^1}}<\infty$,\qquad for all $t\ge0$.}
\end{align}
\end{definition}
The next lemma proves that $\pin(t)$ is indeed uniform differentiable, and the associated linear operator $D\pin(t, \theta_0)$ is given by 
\begin{align*}
D\pin(t, \theta_0)[\psi_0]:=\psi(t),
\end{align*}
where $\psi$ is the solution of the linearised problem with respect to \eqref{abstract active scalar eqn}
\begin{align}\label{linearised active scalar with gamma}
\left\{ \begin{array}{l}
\frac{d\psi}{dt}=A_{\theta}[\psi], \\
\psi(x,0)=\psi_0(x),
\end{array}\right.
\end{align}
with $\theta=\pin(t)\theta_0$ and $A_{\theta}$ is the elliptic operator given by
\begin{equation}\label{def of operator A theta}
A_{\theta}[\psi]=A_{\theta_0}(t)[\psi]:=-\kappa\Lambda^\gamma\psi-u[\theta]\cdot\nabla\psi-u[\psi]\cdot\nabla\theta.
\end{equation}
\begin{lemma}\label{uniform differentiable lemma}
For fixed $\nu>0$, if $\gamma$ satisfies the condition \eqref{conditions on gamma}, then the solution map $\pin(t)$ is uniform differentiable on $\Gg$ in the sense of Definition~\ref{uniform differentiable def}, and the associated linear operator $D\pin(t, \theta_0)$ is given by \eqref{linearised active scalar with gamma}. Furthermore, the linear operator $D\pin(t, \theta_0)$ is compact.
\end{lemma}

\begin{proof}
For $\theta_0$, $\varphi_0\in\Gg$, we let $\theta(t)=\pin(t)\theta_0$ and $\varphi(t)=\pin(t)\varphi_0$. We denote $\psi_0=\varphi_0-\theta_0$ and define $\psi(t)=D\pin(t, \theta_0)[\psi_0]$, where $\psi(t)$ satisfies \eqref{linearised active scalar with gamma}. We also define 
\begin{align*}
\eta(t)=\varphi(t)-\theta(t)-\psi(t)=\pin(t)\varphi_0-\pin(t)\theta_0-D\pin(t, \theta_0)[\psi_0].
\end{align*}
Then $\eta(t)$ satisfies
\begin{align}\label{eqn for eta}
\left\{ \begin{array}{l}
\dt\eta+\kappa\Lambda^\gamma\eta+u[\eta]\cdot\nabla\theta+u[\theta]\cdot\nabla\eta=-u[w]\cdot w, \\
\eta(x,0)=0,
\end{array}\right.
\end{align}
where $w(t)=\varphi(t)-\theta(t)=\pin(t)\psi_0$. We take $L^2$-inner product of \eqref{eqn for eta}$_1$ with $-\Delta\eta$ and obtain
\begin{align}\label{differential eqn for eta}
&\frac{1}{2}\frac{d}{dt}\|\eta\|^2_{H^1}+\kappa\|\eta\|^2_{H^{1+\frac{\gamma}{2}}}\notag\\
&=\intox u[\eta]\cdot\nabla\theta\Delta\eta-\intox\partial_{x_k}u[\theta]\cdot\nabla\eta\partial_{x_k}\eta+\intox u[w]\cdot\nabla w\Delta\eta.
\end{align}
Using \eqref{bound on nabla u in terms of nabla theta} on $\nabla u$, the second integrand of \eqref{differential eqn for eta} can be bounded by 
\begin{align*}
\Big|\intox\partial_{x_k}u[\theta]\cdot\nabla\eta\partial_{x_k}\eta\Big|\le C\|\nabla u\|_{L^\infty}\|\nabla\eta\|^2_{L^2}\le C\|\Lambda^{1+\frac{\gamma}{2}}\theta\|_{L^2}\|\nabla\eta\|^2_{L^2}.
\end{align*}
To estimate the first integrand appeared in \eqref{differential eqn for eta}, using the fact that $\divv(u[\eta])=0$ and integrating by parts, we have
\begin{align*}
\Big|\intox u[\eta]\cdot\nabla\theta\Delta\eta\Big|&=\Big|\intox u[\eta]\theta\cdot\nabla\Delta\eta\Big|\\
&=\Big|\intox\Lambda^{2-\frac{\gamma}{2}}(u[\eta]\theta)\cdot\Lambda^\frac{\gamma}{2}\nabla\eta\Big|\le \|\Lambda^{2-\frac{\gamma}{2}}(u[\eta]\theta)\|_{L^2}\|\Lambda^{1+\frac{\gamma}{2}}\eta\|_{L^2}.
\end{align*} 
Using \eqref{bound on Lambda theta for gamma>1}, we have
\begin{align*}
\|\Lambda^{2-\frac{\gamma}{2}}(u[\eta]\theta)\|_{L^2}\le C\|\Lambda^{1+\frac{\gamma}{2}}(u[\eta]\theta)\|_{L^2},
\end{align*}
and using the product estimate \eqref{product estimate}, 
\begin{align*}
\|\Lambda^{1+\frac{\gamma}{2}}(u[\eta]\theta)\|_{L^2}\le C\Big(\|\Lambda^{1+\frac{\gamma}{2}}u[\eta]\|_{L^4}\|\theta\|_{L^4}+\|u[\eta]\|_{L^\infty}\|\Lambda^{1+\frac{\gamma}{2}}\theta\|_{L^2}\Big).
\end{align*}
Using \eqref{two order smoothing for u when nu>0} and \eqref{bound on L4 using H1}, the term $\|\Lambda^{1+\frac{\gamma}{2}}u[\eta]\|_{L^4}\|\theta\|_{L^4}$ can be bounded by 
\begin{align*}
\|\Lambda^{1+\frac{\gamma}{2}}u[\eta]\|_{L^4}\|\theta\|_{L^4}&\le C\|\Lambda^{2+\frac{\gamma}{2}}u[\eta]\|_{L^2}\|\Lambda\theta\|_{L^2}\\
&\le C\|\Lambda^{\frac{\gamma}{2}}\eta\|_{L^2}\|\Lambda\theta\|_{L^2}\le C\|\Lambda\eta\|_{L^2}\|\Lambda\theta\|_{L^2},
\end{align*}
for $\gamma\le2$, while the term $\|u[\eta]\|_{L^\infty}\|\Lambda^{1+\frac{\gamma}{2}}\theta\|_{L^2}$ can be bounded by $C\|\Lambda \eta\|_{L^2}\|\Lambda^{1+\frac{\gamma}{2}}\theta\|_{L^2}$ with the help of \eqref{bound on u in terms of nabla theta}. Therefore, we have
\begin{align*}
\Big|\intox u[\eta]\cdot\nabla\theta\Delta\eta\Big|\le C\Big(\|\Lambda\eta\|_{L^2}\|\Lambda\theta\|_{L^2}+\|\Lambda \eta\|_{L^2}\|\Lambda^{1+\frac{\gamma}{2}}\theta\|_{L^2}\Big)\|\Lambda^{1+\frac{\gamma}{2}}\eta\|_{L^2}.
\end{align*} 
Similarly, we can obtain the following estimates on $\dis\intox u[w]\cdot\nabla w\Delta\eta$:
\begin{align*}
\Big|\intox u[w]\cdot\nabla w\Delta\eta\Big|\le\|\Lambda^{1+\frac{\gamma}{2}}(u[w]w)\|_{L^2}\|\Lambda^{1+\frac{\gamma}{2}}\eta\|_{L^2}.
\end{align*}
Using \eqref{product estimate} again, we have
\begin{align*}
\|\Lambda^{1+\frac{\gamma}{2}}(u[w]w)\|_{L^2}\le C\Big(\|\Lambda^{1+\frac{\gamma}{2}}u[w]\|_{L^4}\|w\|_{L^4}+\|u[w]\|_{L^\infty}\|\Lambda^{1+\frac{\gamma}{2}}w\|_{L^2}\Big),
\end{align*}
and we apply \eqref{bound on L4 using H1} and \eqref{bound on u in terms of nabla theta} to obtain
\begin{align*}
&\|\Lambda^{1+\frac{\gamma}{2}}u[w]\|_{L^4}\|w\|_{L^4}+\|u[w]\|_{L^\infty}\|\Lambda^{1+\frac{\gamma}{2}}w\|_{L^2}\\
&\le C\Big(\|\Lambda^{2+\frac{\gamma}{2}}u[w]\|_{L^2}\|\Lambda w\|_{L^2}+\|\Lambda w\|_{L^2}\|\Lambda^{1+\frac{\gamma}{2}}w\|_{L^2}\Big)\\
&\le C\|\Lambda w\|_{L^2}\|\Lambda^{1+\frac{\gamma}{2}}w\|_{L^2},
\end{align*}
where the last inequality follows by \eqref{two order smoothing for u when nu>0} and the assumption that $\gamma\le2$. Hence we have
\begin{align*}
\Big|\intox u[w]\cdot\nabla w\Delta\eta\Big|\le C\|\Lambda w\|_{L^2}\|\Lambda^{1+\frac{\gamma}{2}}w\|_{L^2}\|\Lambda^{1+\frac{\gamma}{2}}\eta\|_{L^2},
\end{align*}
and we infer from \eqref{differential eqn for eta} that
\begin{align}\label{differential eqn for eta step 2}
&\frac{1}{2}\frac{d}{dt}\|\eta\|^2_{H^1}+\kappa\|\eta\|^2_{H^{1+\frac{\gamma}{2}}}\notag\\
&\le C\|\Lambda^{1+\frac{\gamma}{2}}\theta\|_{L^2}\|\nabla\eta\|^2_{L^2}+C\Big(\|\Lambda\theta\|_{L^2}+\|\Lambda^{1+\frac{\gamma}{2}}\theta\|_{L^2}\Big)\|\Lambda \eta\|_{L^2}\|\Lambda^{1+\frac{\gamma}{2}}\eta\|_{L^2}\notag\\
&\qquad+C\|\Lambda w\|_{L^2}\|\Lambda^{1+\frac{\gamma}{2}}w\|_{L^2}\|\Lambda^{1+\frac{\gamma}{2}}\eta\|_{L^2}.
\end{align}
Next we focus on the function $w=w(t)$ which satisfies the following equations:
\begin{align}\label{eqn for w}
\left\{ \begin{array}{l}
\dt w+\kappa\Lambda^\gamma w+u[\varphi]\cdot\nabla w+u[w]\cdot\nabla\theta=0, \\
w(x,0)=\psi_0.
\end{array}\right.
\end{align}
We take $L^2$-inner product of \eqref{eqn for w}$_1$ with $-\Delta w$ to give
\begin{align}\label{differential eqn for w}
\frac{1}{2}\frac{d}{dt}\|w\|^2_{H^1}+\kappa\|w\|^2_{H^{1+\frac{\gamma}{2}}}=\intox u[\varphi]\cdot\nabla w\Delta w+\intox u[w]\cdot\nabla\theta\Delta w.
\end{align}
Similar to the previous case for $\eta$, we have
\begin{align*}
\Big|\intox u[\varphi]\cdot\nabla w\Delta w\Big|\le C\|\Lambda^{1+\frac{\gamma}{2}}\varphi\|_{L^2}\|\nabla w\|^2_{L^2},
\end{align*}
and 
\begin{align*}
\Big|\intox u[w]\cdot\nabla\theta\Delta w\Big|\le C\Big(\|\Lambda w\|_{L^2}\|\Lambda\theta\|_{L^2}+\|\Lambda w\|_{L^2}\|\Lambda^{1+\frac{\gamma}{2}}\theta\|_{L^2}\Big)\|\Lambda^{1+\frac{\gamma}{2}}w\|_{L^2}.
\end{align*} 
Hence \eqref{differential eqn for w} implies
\begin{align}\label{differential eqn for w step 2}
&\frac{1}{2}\frac{d}{dt}\|w\|^2_{H^1}+\kappa\|w\|^2_{H^{1+\frac{\gamma}{2}}}\notag\\
&\le C\|\Lambda^{1+\frac{\gamma}{2}}\varphi\|_{L^2}\|\nabla w\|^2_{L^2}+C\Big(\|\Lambda w\|_{L^2}\|\Lambda\theta\|_{L^2}+\|\Lambda w\|_{L^2}\|\Lambda^{1+\frac{\gamma}{2}}\theta\|_{L^2}\Big)\|\Lambda^{1+\frac{\gamma}{2}}w\|_{L^2}\notag\\
&\le\frac{\kappa}{2}\|w\|^2_{H^{1+\frac{\gamma}{2}}}+C\Big[\|\Lambda^{1+\frac{\gamma}{2}}\varphi\|_{L^2}+\frac{1}{\kappa}(\|\Lambda\theta\|^2_{L^2}+\|\Lambda^{1+\frac{\gamma}{2}}\varphi\|^2_{L^2})\Big]\|w\|^2_{H^1}.
\end{align}
Since $\theta_0$, $\varphi_0\in\Gg$, $\theta$ and $\varphi$ both satisfy the bounds \eqref{uniform bound on theta from the attractor}-\eqref{uniform bound on time integral of theta from the attractor}, namely,
\begin{align}\label{uniform bound on theta and varphi}
\|\theta\|_{H^{1+\frac{\gamma}{2}}}\le M_{\Gg},\qquad \|\varphi\|_{H^{1+\frac{\gamma}{2}}}\le M_{\Gg},
\end{align}
and
\begin{align}\label{uniform bound on time integral of theta and varphi}
\frac{1}{T}\int_0^{T}\|\theta(\cdot,\tau)\|_{H^{1+\gamma}}d\tau\le M_{\Gg},\qquad \frac{1}{T}\int_0^{T}\|\varphi(\cdot,\tau)\|_{H^{1+\gamma}}d\tau\le M_{\Gg},\qquad\forall T>0.
\end{align}
We apply the bounds \eqref{uniform bound on theta and varphi} on \eqref{differential eqn for w step 2} and use Gr\"{o}nwall's inequality to obtain
\begin{align}\label{H1 bound on w}
\|w(\cdot,t)\|^2_{H^1}\le\|\psi_0\|^2_{H^1}K(t,M_{\Gg}),\qquad\forall t\ge0,
\end{align}
as well as 
\begin{align}\label{time integral on H 1+gamma/2 norm of w}
\int_0^t\|w(\cdot,\tau)\|^2_{H^{1+\frac{\gamma}{2}}}d\tau\le\|\psi_0\|^2_{H^1}K(t,M_{\Gg}),\qquad\forall t\ge0.
\end{align}
where $K(t,M_{\Gg})$ is a positive function in $t$. We combine \eqref{differential eqn for eta step 2} with \eqref{uniform bound on theta and varphi} and \eqref{H1 bound on w} to get
\begin{align}\label{differential eqn for eta step 3}
&\frac{d}{dt}\|\eta\|^2_{H^1}+\kappa\|\eta\|^2_{H^{1+\frac{\gamma}{2}}}\notag\\
&\le CM_{\Gg}\|\eta\|^2_{H^1}+\frac{CM^2_{\Gg}}{\kappa}\|\eta\|^2_{H^1}+\frac{C}{\kappa}K(t,M_{\Gg})\|\psi_0\|^2_{H^1}\|w\|^2_{H^{1+\frac{\gamma}{2}}}.
\end{align}
Using Gr\"{o}nwall's inequality on \eqref{differential eqn for eta step 3} and recalling the fact that $\eta(0)=0$, we conclude that
\begin{align}
\|\eta(\cdot,t)\|^2_{H^1}&\le \exp\Big((CM_{\Gg}+\frac{CM^2_{\Gg}}{\kappa})t\Big)\frac{C}{\kappa}K(t,M_{\Gg})\|\psi_0\|^2_{H^1}\int_0^t\|w(\cdot,\tau)\|^2_{H^{1+\frac{\gamma}{2}}}d\tau\notag\\
&\le \tilde K(t,M_{\Gg})\|\psi_0\|^4_{H^1},
\end{align}
where $\tilde K(t,M_{\Gg}):=\exp\Big((CM_{\Gg}+\frac{CM^2_{\Gg}}{\kappa})t\Big)\frac{C}{\kappa}K^2(t,M_{\Gg})$. Hence we prove that
\begin{align*}
\lim_{\varepsilon\to0}\left(\sup_{\theta_0,\varphi_0\in\Gg;0<\|\psi_0\|_{H^1}\le\varepsilon}\frac{\|\eta(\cdot,t)\|_{H^1}}{\|\psi_0\|_{H^1}}\right)\le \lim_{\varepsilon\to0}\tilde K(t,M_{\Gg})\varepsilon^2=0,
\end{align*}
and \eqref{condition 1 for uniform differentiable} follows.

To prove \eqref{condition 2 for uniform differentiable}, it suffices to consider $\psi_0$ to be normalised so that $\|\psi_0\|_{H^1}=1$. Let $\theta_0\in\Gg$ be arbitrary, then using the similar estimates as given above, we have
\begin{align}\label{differential eqn for psi step 1}
&\frac{1}{2}\frac{d}{dt}\|\psi\|^2_{H^1}+\kappa\|\psi\|^2_{H^{1+\frac{\gamma}{2}}}\notag\\
&\le\Big|\intox u[\theta]\cdot\nabla\psi\Delta\psi\Big|+\Big|\intox u[\psi]\cdot\nabla\theta\Delta\psi\Big|\notag\\
&\le C\|\Lambda^{1+\frac{\gamma}{2}}\theta\|_{L^2}\|\nabla \psi\|^2_{L^2}+C\Big(\|\Lambda \psi\|_{L^2}\|\Lambda\theta\|_{L^2}+\|\Lambda \psi\|_{L^2}\|\Lambda^{1+\frac{\gamma}{2}}\theta\|_{L^2}\Big)\|\Lambda^{1+\frac{\gamma}{2}}\psi\|_{L^2},
\end{align} 
which gives 
\begin{align}\label{H1 bound on psi} 
\|\psi(\cdot,t)\|^2_{H^1}\le K(t,M_{\Gg}),\qquad\forall t\ge0,
\end{align}
and \eqref{condition 2 for uniform differentiable} holds as well.

Finally, we show that for any $t>0$ and $\theta_0\in\Gg$, the linear operator $D\pin(t, \theta_0)$ is compact. Following the argument given in \cite{CTV14}, it suffices to show that if $U_1$ is the unit ball in $H^1$, then $D\pin(t, \theta_0)U_1\subset H^{1+\frac{\gamma}{2}}$. In view of \eqref{differential eqn for psi step 1} and \eqref{H1 bound on psi}, we obtain
\begin{align}
\int_0^t\|\psi(\cdot,\tau)\|^2_{H^{1+\frac{\gamma}{2}}}d\tau\le K(t,M_{\Gg}),\qquad\forall t\ge0.
\end{align}
Hence using the mean value theorem, for $t>0$, there exists $\tau\in(0,t)$ such that
\begin{align}\label{bound on psi in H 1+gamma/2}
\|\psi(\cdot,\tau)\|^2_{H^{1+\frac{\gamma}{2}}}\le \frac{1}{t}K(t,M_{\Gg}).
\end{align}
We take the $L^2$-inner product of \eqref{linearised active scalar with gamma}$_1$ with $\Lambda^{2+\gamma}\psi$ and obtain
\begin{align*}
\frac{1}{2}\frac{d}{dt}\|\psi\|^2_{H^{1+\frac{\gamma}{2}}}+\kappa\|\psi\|^2_{H^{1+\gamma}}=\intox u[\theta]\cdot \nabla \psi\Lambda^{2+\gamma}\psi+\intox u[\psi]\cdot\nabla\theta\Lambda^{2+\gamma}\psi.
\end{align*}
Using the commutator estimate \eqref{commutator estimate} and the bound \eqref{bound on L4 using H1}, the term $\Big|\intox u[\theta]\cdot \nabla \psi\Lambda^{2+\gamma}\psi\Big|$ can be bounded by
\begin{align*}
&\Big|\intox u[\theta]\cdot \nabla \psi\Lambda^{2+\gamma}\psi\Big|\\
&\le\Big|\intox(\Lambda^{1+\frac{\gamma}{2}}(u[\theta]\nabla \psi)-u[\theta]\cdot\Lambda^{1+\frac{\gamma}{2}}\nabla \psi)\Lambda^{1+\frac{\gamma}{2}}\psi\Big|\\
&\le \|\Lambda^{1+\frac{\gamma}{2}}(u[\theta]\nabla \psi)-u[\theta]\cdot\Lambda^{1+\frac{\gamma}{2}}\nabla \psi\|_{L^2}\|\Lambda^{1+\frac{\gamma}{2}} \psi\|_{L^2}\\
&\le C\Big(\|\nabla(u[\theta])\|_{L^\infty}\|\Lambda^{1+\frac{\gamma}{2}}\psi\|_{L^2}+\|\Lambda^{1+\frac{\gamma}{2}}u[\theta]\|_{L^4}\|\nabla \psi\|_{L^4}\Big)\|\Lambda^{1+\frac{\gamma}{2}} \psi\|_{L^2}\\
&\le C\Big(\|\Lambda\theta\|_{L^2}\|\Lambda^{1+\frac{\gamma}{2}} \psi\|_{L^2}+\|\Lambda^\frac{\gamma}{2}\theta\|_{L^2}\|\Lambda^{1+\gamma} \psi\|_{L^2}\Big)\|\Lambda^{1+\frac{\gamma}{2}} \psi\|_{L^2},
\end{align*}
and similarly, we also have
\begin{align*}
&\Big|\intox u[\psi]\cdot \nabla \theta\Lambda^{2+\gamma}\psi\Big|\\
&=\Big|\intox\Lambda^2(u[\psi]\theta)\Lambda^{\gamma}\nabla \psi\Big|\\
&\le \|\Lambda^2(u[\psi]\theta)\|_{L^2}\|\Lambda^{1+\gamma} \psi\|_{L^2}\\
&\le C\Big(\|\Lambda^2(u[\psi])\|_{L^4}\|\theta\|_{L^4}+\|u[\psi]\|_{L^\infty}\|\Lambda^2 \theta\|_{L^2}\Big)\|\Lambda^{1+\gamma} \psi\|_{L^2}\\
&\le C\Big(\|\Lambda \psi\|_{L^2}\|\Lambda \theta\|_{L^2}+\|\Lambda \psi\|_{L^2}\|\Lambda^{1+\gamma} \theta\|_{L^2}\Big)\|\Lambda^{1+\gamma} \psi\|_{L^2},
\end{align*}
Therefore, using Young's inequality we obtain
\begin{align}\label{differential eqn for psi higher order}
&\frac{1}{2}\frac{d}{dt}\|\psi\|^2_{H^{1+\frac{\gamma}{2}}}+\frac{\kappa}{2}\|\psi\|^2_{H^{1+\gamma}}\notag\\&\le \frac{C}{\kappa}\Big(\|\Lambda\theta\|_{L^2}+\|\Lambda^\frac{\gamma}{2}\theta\|^2_{L^2}+\|\Lambda\theta\|^2_{L^2}+\|\Lambda^{1+\gamma}\theta\|^2_{L^2}\Big)\|\psi\|^2_{H^{1+\frac{\gamma}{2}}}
\end{align}
Integrating \eqref{differential eqn for psi higher order} from $\tau$ to $t$, applying Gr\"{o}nwall's inequality and using the bounds \eqref{uniform bound on time integral of theta and varphi} and \eqref{bound on psi in H 1+gamma/2}, we arrive at
\begin{align}\label{H 1+gamma/2 bound on psi}
\|\psi(\cdot,t)\|^2_{H^{1+\frac{\gamma}{2}}}&\le \|\psi(\cdot,\tau)\|^2_{H^{1+\frac{\gamma}{2}}}\exp\Big(\frac{C}{\kappa}\int_{\tau}^t\|\theta(\cdot,s)\|^2_{H^{1+\gamma}}ds\Big)\notag\\
&\le \frac{1}{t}K(t,M_{\Gg})\exp\Big(\frac{C}{\kappa}M_{\Gg}t\Big),
\end{align}
hence $\psi(t)\in H^{1+\frac{\gamma}{2}}$ and we finish the proof of lemma~\ref{uniform differentiable lemma}.
\end{proof}

Next we show that there is an $N$ such that volume elements which are carried by the flow of $\pin(t)\theta_0$, with $\theta_0\in\Gg$, decay exponentially for dimensions larger than $N$. We recall the following proposition for which the proof can be found in \cite{CF85}, \cite{CF88}.

\begin{proposition}\label{volume decay prop}
Consider $\theta_0\in\Gg$, and an initial orthogonal set of infinitesimal displacements $\{\psi_{1,0},\dots,\psi_{n,0}\}$ for some $n\ge1$. Suppose that $\psi_i$ obey the following equation:
\begin{align}
\dt\psi_i=A_{\theta_0}(t)[\psi_i],\qquad \psi_i(0)=\psi_{i,0},
\end{align}
for all $i\in\{1,\dots, n\}$ and $t\ge0$, where $A_{\theta_0}(t)$ is given by \eqref{def of operator A theta}. Then the volume elements
\begin{align*}
V_n(t)=\|\psi_1(t)\wedge\dots\wedge\psi_n(t)\|_{H^1}
\end{align*}
satsify
\begin{align*}
V_n(t)=V_n(0)\exp\left(\int_0^t \Tr(P_n(s)A_{\theta_0}(s))ds\right),
\end{align*}
where the orthogonal projection $P_n(s)$ is onto the linear span of $\{\psi_1(s),\dots,\psi_n(s)\}$ in the Hilbert space $H^1$, and $\Tr(P_n(s)A_{\theta})$ is defined by 
\[\Tr(P_n(s)A_{\theta}) = \sum_{j=1}^n\int_{\T^d}(-\Delta\varphi_j(s))A_{\theta}[\varphi_j(s)]dx\]
for $n\ge1$, with $\{\varphi_1(s),\dots,\varphi_n(s)\}$ an orthornormal set spanning the linear span of $\{\psi_1(s),\dots,\psi_n(s)\}$. If we define
\begin{align*}
\langle P_nA_{\theta_0}\rangle:=\limsup_{T\to\infty}\frac{1}{T}\int_0^T\Tr(P_n(t)A_{\theta_0}(t))dt,
\end{align*}
then we further obtain
\begin{align}\label{volume decay}
V_n(t)\le V_n(0)\left(t\sup_{\theta_0\in\Gg}\sup_{P_n(0)}\langle P_nA_{\theta_0}\rangle\right),\qquad \forall t\ge0,
\end{align}
where the supremum over $P_n(0)$ is a supremum over all choices of initial $n$ orthogonal set of infinitesimal displacements that we take around $\theta_0$.
\end{proposition}
Using Proposition~\ref{volume decay prop}, we show that the $n$-dimensional volume elements actually {\it decay exponentially in time} for $n$ is sufficiently large, which is based on the following lemma:

\begin{lemma}[Contractivity of large dimensional volume elements]\label{Contractivity of large dimensional volume elements lemma}
There exists $N$ such that for any $\theta_0\in\Gg$ and any set of initial orthogonal displacements $\{\psi_{i,0}\}_{i=1}^n$, we have
\begin{align}\label{negativity of <Pn A theta>}
\langle P_nA_{\theta_0}\rangle<0,
\end{align}
whenever $n\ge N$. Here $N$ can be chosen explicitly from \eqref{choice of N} below.
\end{lemma}

\begin{proof}
Let $\xi\in H^1$ be arbitrary. For $\theta_0\in\Gg$, using the definition of $A_{\theta}$ in \eqref{def of operator A theta} and the fact that $\divv u[\theta]=0$, we have
\begin{align*}
\intox\Lambda^2\xi A_{\theta}[\xi]dx\le-\kappa\|\xi\|^2_{H^{1+\frac{\gamma}{2}}}+\Big|\intox\partial_{x_k}u[\theta]\cdot\nabla\xi\partial_{x_k}\xi dx\Big|+\Big|\intox(u[\xi]\cdot\nabla\theta)\Lambda^2\xi dx\Big|. 
\end{align*}
Using the bound \eqref{bound on u in terms of nabla theta}, we readily have
\begin{align*}
\Big|\intox\partial_{x_k}u[\theta]\cdot\nabla\xi\partial_{x_k}\xi dx\Big|\le C\|\nabla u[\theta]\|_{L^\infty}\|\nabla \xi\|^2_{L^2}\le C\|\Lambda\theta\|_{L^2}\|\nabla \xi\|^2_{L^2}.
\end{align*}
And using the product estimate \eqref{product estimate} and the bounds \eqref{bound on L4 using H1}-\eqref{bound on u in terms of nabla theta}, we have
\begin{align*}
\Big|\intox(u[\xi]\cdot\nabla\theta)\Lambda^2\xi dx\Big|&=\Big|\intox(u[\xi]\theta)\cdot\nabla\Lambda^2\xi dx\Big|\\
&\le\|\Lambda^{2-\frac{\gamma}{2}}(u[\xi]\theta)\|_{L^2}\|\Lambda^{1+\frac{\gamma}{2}}\xi\|_{L^2}\\
&\le C\Big(\|\Lambda^{2-\frac{\gamma}{2}}u[\xi]\|_{L^4}\|\theta\|_{L^4}+\|u[\xi]\|_{L^\infty}\|\Lambda^{2-\frac{\gamma}{2}}\theta\|_{L^2}\Big)\|\Lambda^{1+\frac{\gamma}{2}}\xi\|_{L^2}\\
&\le C\|\theta\|_{H^{1+\frac{\gamma}{2}}}\|\Lambda\xi\|_{L^2}\|\Lambda^{1+\frac{\gamma}{2}}\xi\|_{L^2}.
\end{align*}
Hence we deduce that
\begin{align}\label{bound on integrand in A theta}
\intox\Lambda^2\xi A_{\theta}[\xi]dx\le-\frac{\kappa}{2}\|\xi\|^2_{H^{1+\frac{\gamma}{2}}}+C\Big(\|\Lambda\theta\|_{L^2}+\frac{1}{\kappa}\|\theta\|^2_{H^{1+\frac{\gamma}{2}}}\Big)\|\xi\|^2_{H^1}.
\end{align}
On the other hand, for $T>0$, we have
\begin{align}\label{identity for integrand in A theta}
\frac{1}{T}\int_0^T\Tr(P_n(t)A_{\theta_0}(t))dt=\frac{1}{T}\int_0^T\sum_{j=1}^n\intox(\Lambda^2\varphi_j(t))A_{\theta}[\varphi_j(t)]dxdt,
\end{align}
hence we apply \eqref{bound on integrand in A theta} on \eqref{identity for integrand in A theta} and together with the bound \eqref{uniform bound on theta from the attractor} on $\theta$,
\begin{align*}
&\frac{1}{T}\int_0^T\Tr(P_n(t)A_{\theta_0}(t))dt\\
&\le -\frac{\kappa}{2T}\int_0^T\sum_{j=1}^n\|\varphi_j(t)\|^2_{H^{1+\frac{\gamma}{2}}}+\frac{C}{T}\int_0^T\Big(\|\Lambda\theta\|_{L^2}+\frac{1}{\kappa}\|\theta\|^2_{H^{1+\frac{\gamma}{2}}}\Big)\sum_{j=1}^n\|\varphi_j\|^2_{H^1}\\
&\le -\frac{\kappa}{2T}\int_0^T\Tr(P_n(t)\Lambda^\gamma)+C(M_{\Gg}+M_{\Gg}^2)n\le -\frac{\kappa}{C}n^{1+\frac{\gamma}{d}}+C(M_{\Gg}+M_{\Gg}^2)n,
\end{align*}
where the last inequality follows from the fact that the eigenvalues $\{\lambda_{j}^{(\gamma)}\}$ of $\Lambda^\gamma$ obey the following estimate (see \cite[Theorem~1.1]{YY13}):
\begin{align*}
\sum_{j=1}^n\lambda_{j}^{(\gamma)}\ge\frac{1}{C}n^{1+\frac{\gamma}{d}},
\end{align*}
for some universal constant $C>0$ which depends only on $d$ and $\gamma$. We choose $N>0$ such that
\begin{align}\label{choice of N}
 -\frac{\kappa}{C}N^{1+\frac{\gamma}{d}}+C(M_{\Gg}+M_{\Gg}^2)N<0,
\end{align}
then \eqref{negativity of <Pn A theta>} holds whenever $n\ge N$.
\end{proof}

From the results obtained in Lemma~\ref{uniform differentiable lemma} and Lemma~\ref{Contractivity of large dimensional volume elements lemma}, together with Proposition~\ref{volume decay prop}, we obtain:
\begin{itemize}
\item the solution map $\pin(t)$ is uniform differentiable on $\Gg$;
\item the linearisation $D\pin(t, \theta_0)$ of $\pin(t)$ is compact;
\item the large-dimensional volume elements which are carried by the flow of $\pin(t)\theta_0$, with $\theta_0\in\Gg$, have exponential decay in time.
\end{itemize}
Therefore, following the lines of the argument in \cite[pp. 115--130, and Chapter 14]{CF88}, we can finally conclude that $\dimf(\Gg)$ is finite. The results can be summarised in the following corollary:

\begin{corollary}[Finite dimensionality of the attractor]\label{Finite dimensionality of the attractor corollary}
Let $N$ be as defined in Lemma~\ref{Contractivity of large dimensional volume elements lemma}. Then the fractal dimension of $\Gg$ is finite, and we have $\dimf(\Gg)\le N$.
\end{corollary}

\begin{proof}[Proof of Theorem~\ref{existence of global attractor theorem}]
For $\nu$, $\kappa>0$ and $\gamma\in(0,2]$, the existence and regularity of unique global attractor $\Gg$ follows by Theorem~\ref{Existence of H1 global attractor} and Corollary~\ref{Unique attractor general corollary}. For the case $\gamma\in[1,2]$, the finite dimensionality of the attractor follows by Corollary~\ref{Finite dimensionality of the attractor corollary}.
\end{proof}

\section{Applications to magneto-geostrophic equations}\label{Applications to magneto-geostrophic equations section}

\subsection{The MG equations in the class of drift-diffusion equations}\label{MG equations in the class of drift-diffusion equations}

We now apply our results claimed by Section~\ref{main results} to the magnetogeostrophic (MG) active scalar equation. Specifically, we are interested in the following active scalar equation in the domain $\mathbb{T}^3\times(0,\infty)=[0,2\pi]^3\times(0,\infty)$ (with periodic boundary conditions):
\begin{align}
\label{MG active scalar} \left\{ \begin{array}{l}
\partial_t\theta+u\cdot\nabla\theta=\kappa\Delta\theta+S, \\
u=M^{\nu}[\theta],\theta(x,0)=\theta_0(x)
\end{array}\right.
\end{align}
via a Fourier multiplier operator $M^{\nu}$ which relates $u$ and $\theta$. More precisely,
\begin{align}\label{def of u by Fourier transform}
u_j=M^{\nu}_j [\theta]=(\widehat{M^{\nu}_j}\hat\theta)^\vee
\end{align}
for $j\in\{1,2,3\}$. The explicit expression for the components of $\widehat M^{\nu}$ as functions of the Fourier variable $k=(k_1,k_2,k_3)\in\Z^3$ with $k_3\neq0$ are given by \eqref{MG Fourier symbol_1 intro}-\eqref{MG Fourier symbol_4 intro}, as discussed in the introduction, in particular $\widehat M^{\nu}$ is not defined on the set $\{k_3 = 0\}$. Since for self-consistency of the model, we assume that $\theta$ and $u$ have zero vertical mean, and we take $\widehat M^{\nu}_j(k) = 0$ on $\{k_3=0\}$ for all $j=1,2,3$ and $\nu\ge0$. We write $M^\nu_j=\partial_{x_i}T_{ij}^{\nu}$ for convenience. We also refer to \eqref{MG active scalar} as the MG$^\nu$ equation when $\nu>0$, and to the case when $\nu=0$ as the MG$^0$ equation (see \cite{FS15}, \cite{FS18}, \cite{FS19} for related discussions).

To apply the results from Section~\ref{main results}, it suffices to show that the sequence of operators $\{T^{\nu}_{ij}\}_{\nu\ge0}$ satisfy the assumptions A1 to A5 given in Section~\ref{introduction}. 
\begin{proposition}\label{assumption checked}
We define $T_{ij}^{\nu}$ by $M^\nu_j=\partial_{x_i}T_{ij}^{\nu}$. Then $T_{ij}^{\nu}$ satisfy the assumptions A1 to A5 given in Section~\ref{introduction}.
\end{proposition}
\begin{proof}
The details for the proof can be found in \cite[Lemma~5.1--5.2]{FS18} and from the discussion in \cite[Section 4]{FV11a}. We omit the details here. 
\end{proof} 
In view of Proposition~\ref{assumption checked}, the abstract Theorem~\ref{Hs convergence}, Theorem~\ref{analytic convergence kappa} and Theorem~\ref{existence of global attractor theorem} can then be applied to the MG equations \eqref{MG active scalar}. More precisely, we have

\begin{theorem}[$H^s$-convergence as $\kappa\rightarrow0$ for MG equations]\label{Hs convergence MG}
Let $\nu>0$ be given as in \eqref{MG active scalar}, and let $\theta_0,S\in C^\infty$ be the initial datum and forcing term respectively with zero mean. If $\theta^\kappa$ and $\theta^0$ are smooth solutions to \eqref{MG active scalar} for $\kappa>0$ and $\kappa=0$ respectively, then 
\begin{align*} 
\lim_{\kappa\rightarrow0}\|(\theta^\kappa-\theta^0)(\cdot,t)\|_{H^s}=0,
\end{align*}
for all $s\ge0$ and $t\ge0$.
\end{theorem}

\begin{theorem}[Analytic convergence as $\kappa\rightarrow0$ for MG equations]\label{Analytic convergence kappa MG}
Let $\nu=0$ be given as in \eqref{MG active scalar}, and let $\theta_0,S$ the initial datum and forcing term respectively. Suppose that $\theta_0$ and $S$ are both analytic functions with zero mean. Then if $\theta^{\kappa}$, $\theta^{0}$ are analytic solutions to \eqref{MG active scalar} for $\kappa>0$ and $\kappa=0$ respectively with initial datum $\theta_0$ and with radius of convergence at least $\bar{\tau}$, then there exists $T\le\bar{T}$ and $\tau=\tau(t)<\bar{\tau}$ such that, for $t\in[0,T]$, we have:
\begin{align*}
\lim_{\kappa\to0}\|(\Lambda^re^{\tau\Lambda}\theta^{\kappa}-\Lambda^re^{\tau\Lambda}\theta^{0})(\cdot,t)\|_{L^2}=0.
\end{align*}
\end{theorem}

\begin{theorem}[Existence of global attractors for MG equations]\label{existence of global attractor MG theorem}
Let $S\in L^\infty\cap H^1$. For $\nu$, $\kappa>0$, let $\pin(t)$ be solution operator for the initial value problem \eqref{MG active scalar} via \eqref{def of solution map nu>0 and kappa>0}. Then the solution map $\pin(t):H^1\to H^1$ associated to \eqref{abstract active scalar eqn} possesses a unique global attractor $\Gg$ for all $\nu>0$. In particular, for each $\nu>0$, the global attractor $\Gg$ of $\pin(t)$ enjoys the following properties:
\begin{itemize}
\item $\Gg$ is fully invariant, namely
\begin{align*}
\pin(t)\Gg=\Gg,\qquad \forall t\ge0.
\end{align*}
\item $\Gg$ is maximal in the class of $H^1$-bounded invariant sets.
\item $\Gg$ has finite fractal dimension.
\end{itemize}
\end{theorem}

In the coming subsection, we will further address the limiting properties of $\Gg$ which are related to the critical MG$^{0}$ equation.

\subsection{Behaviour of global attractors for varying $\nu\ge0$}\label{Behaviour of global attractors section}

In the work \cite{FS18}, the authors proved the existence of a compact global attractor $\A$ in $L^2(\mathbb{T}^3)$ for the MG$^0$ equations, namely the equations \eqref{MG active scalar} when $\kappa>0$, $\nu=0$ and $S\in L^\infty\cap H^1$. More precisely, $\A$ is the global attractor generated by the solution map $\pio$ via
\begin{align}\label{def of solution map nu=0 and kappa>0}
\pio(t): L^2\to L^2,\qquad \pio(t)\theta_0=\theta(\cdot,t),\qquad t\ge0,
\end{align}
where $\theta$ is the solution to the MG$^0$ equation with $\theta(\cdot,0)=\theta_0$. In this subsection, we obtain results when $\nu$ is varying, which can be summarised in the following theorem:
\begin{theorem}\label{varying nu theorem}
Let $\kappa>0$ be fixed in \eqref{MG active scalar}. Then we have:
\begin{enumerate}
\item If $\Gg$ are the global attractors for the MG$^\nu$ equations \eqref{MG active scalar} as obtained by Theorem~\ref{existence of global attractor MG theorem}, then $\Gg$ and $\A$ satisfy
\begin{align}\label{upper semi-continuity at nu=0}
\mbox{$\dis\sup_{\phi\in\Gg}\inf_{\psi\in\A}\|\phi-\psi\|_{L^2}\rightarrow0$ as $\nu\rightarrow0$.}
\end{align}
\item Let $\nu^*>\nu_*>0$ be arbitrary. For each $\nu_0\in[\nu_*,\nu^*]$, the collection $\dis\{\Gg\}_{\nu\in[\nu_*,\nu^*]}$ is {\it upper semicontinuous} at $\nu_0$ in the following sense: 
\begin{align}\label{upper semi-continuity at fixed nu>0}
\mbox{$\dis\sup_{\phi\in\Gg}\inf_{\psi\in\Ggo}\|\phi-\psi\|_{H^1}\rightarrow0$ as $\nu\rightarrow\nu_0$.}
\end{align}
\end{enumerate}
\end{theorem}

\begin{remark}
Here are some relevant remarks regarding Theorem~\ref{varying nu theorem}:
\begin{itemize}
\item Since $\A\subset L^2$ and $\Gg\subset H^1\subset L^2$ for all $\nu>0$, it makes sense to address the $L^2$-difference $\|\phi-\psi\|_{L^2}$ for $\phi\in\Gg$ and $\psi\in\A$.
\item Although $\Gg$ has finite fractal dimension for all $\nu>0$, it is unknown whether $\A$ has finite fractal dimension.
\end{itemize}
\end{remark}

\subsubsection{Convergence of attractors as $\nu\to0$} We first prove the convergence result as claimed by \eqref{upper semi-continuity at nu=0}. We recall from \cite[Theorem~6.3]{FS18} that for $\kappa>0$, $\nu\in[0,1]$ and $S\in L^\infty\cap H^2$, there exists global attractor $\An$ in $L^2$ generated by the solution map $\pino$ via
\begin{align}\label{def of solution map nu=0 and kappa>0 in L2}
\pino(t): L^2\to L^2,\qquad \pino(t)\theta_0=\theta(\cdot,t),\qquad t\ge0,
\end{align}
where $\theta$ is the solution to \eqref{MG active scalar} with $\theta(\cdot,0)=\theta_0$, and in particular $\An\Big|_{\nu=0}=\A$. Furthermore, $\An$ is {\it upper semicontinuous} at $\nu=0$ in the following sense (a proof can be found in \cite{FS18}):

\begin{proposition}\label{Upper semi-continuity proposition}
For $\nu\in[0,1]$, the global attractors $\An\subset L^2$ as mentioned above are upper semicontinuous with respect to $\nu$ at $\nu=0$, which means that
\begin{align}\label{uc_0}
\mbox{$\dis\sup_{\phi\in\An}\inf_{\psi\in\mathcal{A}}\|\phi-\psi\|_{L^2}\rightarrow0$ as $\nu\rightarrow0$.}
\end{align}
\end{proposition} 

In view of Proposition~\ref{Upper semi-continuity proposition}, the result \eqref{upper semi-continuity at nu=0} now follows easily from \eqref{uc_0} by observing that $\Gg\subset \An$ for all $\nu\in(0,1]$, which gives 
\begin{align*}
\sup_{\phi\in\Gg}\inf_{\psi\in\mathcal{A}}\|\phi-\psi\|_{L^2}\le \sup_{\phi\in\An}\inf_{\psi\in\mathcal{A}}\|\phi-\psi\|_{L^2},\qquad \forall \nu\in(0,1].
\end{align*}

\begin{remark}
Concerning $\pino$ and $\An$:
\begin{itemize}
\item We readily have $\pino\Big|_{H^1}=\pin$, where $\pin$ is the solution map defined in \eqref{def of solution map nu>0 and kappa>0} for the initial value problem \eqref{MG active scalar} with $\theta(\cdot,0)=\theta_0$.
\item For each $\nu>0$, it is not clear whether the two attractors $\Gg$ and $\An$ coincide. In particular, it is unknown whether $\An$ has finite fractal dimension.
\end{itemize}
\end{remark}

\subsubsection{Upper semicontinuity of global attractors at $\nu>0$} For fixed $\nu^*>\nu_*>0$, we define 
\begin{align}\label{def of I*}
\I=[\nu_*,\nu^*].
\end{align}
In order to obtain the convergence result claimed by \eqref{upper semi-continuity at fixed nu>0}, we need to prove that
\begin{itemize}
\item[L1] there is a compact subset $\U$ of $H^1$ such that $\Gg\subset\U$ for every $\nu\in\I$; and
\item[L2] for $t > 0$, $\pin\theta_0$ is continuous in $\I$, uniformly for $\theta_0$ in compact subsets of $H^1$.
\end{itemize}
Once the conditions L1 and L2 are fufilled, we can apply the result from \cite{HOR15} to conclude that \eqref{upper semi-continuity at fixed nu>0} holds as well. 

Conditions L1 and L2 will be proved in the subsequent lemmas. We first give the following bound on $u^{(\nu)}$ in terms of $\theta^{(\nu)}$ for $\nu\in\I$.

\begin{lemma}
We fix $\nu^*>\nu_*>0$ and define $\I$ by \eqref{def of I*}. Then for any $\nu\in\I$, $s\in[0,2]$ and $f\in L^p$ with $p>1$, we have
\begin{align}\label{two order smoothing for u when nu inside I*}
\|\Lambda^s u^{(\nu)}[f]\|_{L^p}\le C_*\|f\|_{L^p},
\end{align}
where $u^{(\nu)}$ is given by \eqref{def of u by Fourier transform} and \eqref{MG Fourier symbol_1 intro}-\eqref{MG Fourier symbol_4 intro} for $\nu\in\I$, and $C_*$ is a positive constant which depends only on $p$, $\nu_*$ and $\nu^*$.
\end{lemma}
\begin{proof}
By examining the explicit expression for the components of $\widehat M^{\nu}$ given by \eqref{MG Fourier symbol_1 intro}-\eqref{MG Fourier symbol_4 intro}, it is not hard to see that for $s\in[0,2]$, the Fourier symbols of $\Lambda^s u^{(\nu)}[\cdot]$ can be bounded in terms of $\nu_*$ and $\nu^*$ but independent of $s$. The rest follows by the standard Fourier multiplier theorem (see \cite{S70} for example) and we omit the details here.
\end{proof}

\begin{remark}\label{uniform bounded set in H1 remark}
Using the estimate \eqref{two order smoothing for u when nu inside I*} on $\unu[\cdot]$ and following lines by lines the arguments and proofs given in Subsection~\ref{existence of global attractors subsection}, we can obtain parallel results given in Subsection~\ref{existence of global attractors subsection} for all $\nu\in\I$. In particular, as claimed by Lemma~\ref{Existence of an H 1+gamma/2 absorbing set lemma}, there exists a constant $R_{2}\ge1$ which depends only on $\nu_*$, $\nu^*$, $\kappa$, $\|S\|_{L^\infty\cap H^1}$ such that the set
\begin{align*}
B_{2}=\left\{\phi\in H^{2}:\|\phi\|_{H^{2}}\le R_{2}\right\}
\end{align*}
enjoys the following properties:
\begin{itemize}
\item $B_{2}$ is a compact set in $H^1$ which depends only on $\nu_*$, $\nu^*$, $\kappa$, $\|S\|_{L^\infty\cap H^1}$;
\item $\Gg\subset B_{2}$ for all $\nu\in\I$.
\end{itemize}
\end{remark}

The following lemma gives the necessary $H^1$-estimates which will be useful to our analysis.

\begin{lemma}
We fix $\nu^*>\nu_*>0$ and define $\I$ by \eqref{def of I*}. Define $\U=\{\phi\in H^1:\|\phi\|^2_{H^1}\le R_{\U}\}$ where $R_{\U}>0$. For any $\theta_0\in\U$ and $\nu\in\I$, if $\thetanu(t)=\pin(t)\theta_0$, then $\thetanu(t)$ satisfies
\begin{align}\label{H1 estimate in compact set of H1}
\sup_{0\le \tau\le t}\|\thetanu(\cdot,\tau)\|^2_{H^1}+\int_0^t\|\thetanu(\cdot,\tau)\|^2_{H^2}d\tau\le M_*(t),\qquad\forall t>0,
\end{align}
where $M_*(t)$ is a positive function in $t$ which depends only on $t$, $\kappa$, $\nu_*$, $\nu^*$, $\|S\|_{H^1}$ and $R_{\U}$.
\end{lemma}

\begin{proof}
First of all, we recall the the following $L^2$-estimate on $\thetanu$ from \eqref{energy inequality kappa>0}, namely
\begin{align}\label{energy inequality varying nu}
\|\thetanu(\cdot,t)\|^2_{L^2}+\kappa\int_0^t\|\Lambda\thetanu(\cdot,\tau)\|^2_{L^2}d\tau\le\|\theta_0\|^2_{L^2}+\frac{t}{c_0\kappa}\|S\|^2_{L^2}, \qquad\forall t\ge0.
\end{align}
Following the argument given in the proof of Theorem~\ref{global-in-time wellposedness in Sobolev}, for all $\nu\in\I$, we readily have
\begin{align}\label{differential eqn for H1 estimate in compact set of H1}
\frac{1}{2}\frac{d}{dt}\|\nabla\thetanu\|^2_{L^2}+\kappa\|\Delta\thetanu\|^2_{L^2}\le\|\Lambda S\|_{L^2}\|\nabla\thetanu\|_{L^2}+\Big|\intox \unu\cdot\nabla\thetanu\Delta\thetanu\Big|.
\end{align}
Using the fact that $\divv (\unu)=0$ and the commutator estimate given in \eqref{commutator estimate}, we obtain
\begin{align*}
\Big|\intox \unu\cdot\nabla\thetanu\Delta\thetanu\Big|&\le C\|\Lambda\unu\|_{L^4}\|\nabla\thetanu\|_{l^4}\|\Lambda\thetanu\|_{L^2}\\
&\le C\|\Lambda^2\unu\|_{L^2}\|\Lambda^2\thetanu\|_{L^2}\|\Lambda\thetanu\|_{L^2},
\end{align*}
where the last inequality follows from \eqref{bound on L4 using H1}. With the help of \eqref{two order smoothing for u when nu inside I*}, we can further bound $\|\Lambda^2\unu\|_{L^2}$ by $C_*\|\thetanu\|_{L^2}$. Hence we conclude from \eqref{differential eqn for H1 estimate in compact set of H1} that 
\begin{align}\label{differential eqn for H1 estimate in compact set of H1 final}
\frac{1}{2}\frac{d}{dt}\|\nabla\thetanu\|^2_{L^2}+\frac{\kappa}{2}\|\Delta\thetanu\|^2_{L^2}\le\|\Lambda S\|_{L^2}\|\nabla\thetanu\|_{L^2}+\frac{C_*}{\kappa}\|\thetanu\|^2_{L^2}\|\nabla\thetanu\|^2_{L^2},
\end{align}
and the result \eqref{H1 estimate in compact set of H1} follows from \eqref{energy inequality varying nu}, \eqref{differential eqn for H1 estimate in compact set of H1 final} and Gr\"{o}nwall's inequality.
\end{proof}

Next we estimate the difference between the Fourier symbols of $u^{(\nu_1)}$ and $u^{(\nu_2)}$ for $\nu_1$, $\nu_2\in\I$:

\begin{lemma}
We fix $\nu^*>\nu_*>0$ and define $\I$ by \eqref{def of I*}. There exists $C_*>0$ which depends on $\nu^*$ and $\nu_*$ only such that for any $\nu_1$, $\nu_2\in\I$ and $k=(k_1,k_2,k_3)\in\Z^3\backslash\{(0,0,0)\}$, we have
\begin{align}\label{estimates on the difference on Fourier symbols}
|k|^2|\widehat M^{\nu_1}_j(k)-\widehat M^{\nu_2}_j(k)|\le C_*|\nu_1-\nu_2|,\qquad\forall j\in\{1,2,3\}.
\end{align}
\end{lemma}

\begin{proof}
By direct computation, for each $k=(k_1,k_2,k_3)\in\Z^3\backslash\{(0,0,0)\}$ and $\nu_1$, $\nu_2\in\I$,
\begin{align*}
&|k|^2|\widehat M^{\nu_1}_1(k)-\widehat M^{\nu_2}_1(k)|\\
&=|\nu_1-\nu_2||k|^2\\
&\times\frac{\Big|k_1k_3^3|k|^6+k_1k_3|k|^4(k_2^2+\nu_1|k|^4)(k_2^2+\nu_2|k|^4)-k_2k_3|k|^6(2k_2^2+\nu_1|k|^4+\nu_2|k|^4)\Big|}{\Big[|k|^2k_3^2+(k_2^2+\nu_1|k|^4)^2\Big]\Big[|k|^2k_3^2+(k_2^2+\nu_2|k|^4)^2\Big]}\\
&\le|\nu_1-\nu_2||k|^2\frac{|k|^{10}+|k|^{10}(1+\nu^*|k|^2)^2+|k|^{10}(2+2\nu^*|k|^4)}{\nu_*^4|k|^{16}}\\
&\le C_*|\nu_1-\nu_2|,
\end{align*}
where $\dis C_*\ge\frac{1+(1+\nu^*)^2+2(1+\nu^*)}{\nu_*^4}$. The cases for $j=2$ and $j=3$ are just similar and we omit the details.
\end{proof}

We are now ready to state and prove the following lemma which gives the continuity of $\pin$ in $\nu\in\I$.

\begin{lemma}\label{continuity in nu lemma}
We fix $\nu^*>\nu_*>0$ and define $\I$ by \eqref{def of I*}. Then for each $t > 0$, $\pin(t)\theta_0$ is continuous in $\I$, uniformly for $\theta_0$ in compact subsets of $H^1$.
\end{lemma}

\begin{proof}
Given a compact set $\U$ in $H^1$, we choose $R_\U>0$ such that $\U\subset\{\phi\in H^1:\|\phi\|^2_{H^1}\le R_{\U}\}$. For each $\theta_0\in\U$ and $\nu_1$, $\nu_2\in\I$, we define 
$$\theta^{(\nu_i)}(t)=\pi^{\nu_i}(t)\theta_0,\qquad i\in\{1,2\}.$$
We write $\phi=\theta^{(\nu_1)}-\theta^{(\nu_2)}$, then $\phi$ satisfies $\phi(\cdot,0)=0$ and 
\begin{align}\label{differential eqn for the difference in nu}
\dt \phi+(u^{(\nu_1)}-u^{(\nu_2)})\cdot\nabla\theta^{\nu_2}+u^{(\nu_1)}\cdot\nabla \phi=\kappa\Delta \phi,
\end{align}
where $u^{(\nu_i)}=u^{(\nu_i)}[\theta^{(\nu_i)}]$ for $i=1,2$. Multiply \eqref{differential eqn for the difference in nu} by $-\Delta\phi$ and integrate,
\begin{align}\label{integral eqn for the difference in nu}
\frac{1}{2}\frac{d}{dt}\|\nabla\phi\|^2_{L^2}+\kappa\|\Delta\phi\|^2_{L^2}\le\Big|\intox (u^{(\nu_1)}-u^{(\nu_2)})\cdot\nabla\theta^{\nu_2}\Delta \phi\Big|+\Big|\intox u^{(\nu_1)}\cdot\nabla \phi\Delta \phi\Big|
\end{align}
Upon integrating by part and exploiting the fact that $\divv(u^{(\nu_1)})=0$, we readily obtain
\begin{align*}
\Big|\intox u^{(\nu_1)}\cdot\nabla \phi\Delta \phi\Big|\le \|\nabla u^{(\nu_1)}\|_{L^\infty}\|\nabla\phi\|^2_{L^2}.
\end{align*}
Using the estimates \eqref{L infty bound 2} and \eqref{two order smoothing for u when nu inside I*}, we can bound $\|\nabla u^{(\nu_1)}\|_{L^\infty}$ by $C_*\|\Delta\theta^{(\nu_1)}\|_{L^2}$ for some positive constant $C_*$ which only depends on $\nu^*$ and $\nu_*$, and $C_*$ may change from line to line. On the other hand, using H\"{o}lder's inequality and \eqref{bound on L4 using H1}, we have
\begin{align*}
\Big|\intox (u^{(\nu_1)}-u^{(\nu_2)})\cdot\nabla\theta^{\nu_2}\Delta \phi\Big|\le C\|\nabla(u^{(\nu_1)}-u^{(\nu_2)})\|_{L^2}\|\Delta\theta^{(\nu_2)}\|_{L^2}\|\Delta\phi\|_{L^2}.
\end{align*}
To bound the term $\|\nabla(u^{(\nu_1)}-u^{(\nu_2)})\|_{L^2}$, we apply the estimate \eqref{estimates on the difference on Fourier symbols} to obtain
\begin{align*}
\|\nabla(u^{(\nu_1)}-u^{(\nu_2)})\|^2_{L^2}&=\sum_{k\in\mathbb{Z}^3\backslash\{(0,0,0)\}}|k|^2|(\widehat{u^{(\nu_1)}}-\widehat{u^{(\nu_2)}})(k)|^2\\
&\le \sum_{k\in\mathbb{Z}^3\backslash\{(0,0,0)\}}|k|^2|\widehat{M^{\nu_1}}|^2|(\widehat{\theta^{(\nu_1)}}-\widehat{\theta^{(\nu_2)}})(k)|^2\\
&\qquad+\sum_{k\in\mathbb{Z}^3\backslash\{(0,0,0)\}}|k|^2|\widehat{M^{\nu_1}}-\widehat{M^{\nu_2}}|^2|\widehat{\theta^{(\nu_2)}}(k)|^2\\
&\le C_*\|\phi\|^2_{L^2}+C_*|\nu_1-\nu_2|^2\|\theta^{(\nu_2)}\|^2_{L^2}.
\end{align*}
Hence we deduce from \eqref{integral eqn for the difference in nu} that
\begin{align}\label{integral eqn for the difference in nu final}
\frac{1}{2}\frac{d}{dt}\|\nabla\phi\|^2_{L^2}+\frac{\kappa}{2}\|\Delta\phi\|^2_{L^2}&\le \frac{C_*}{\kappa}\|\Delta\theta^{(\nu_2)}\|^2_{L^2}(\|\phi\|^2_{L^2}+|\nu_1-\nu_2|^2\|\theta^{(\nu_2)}\|^2_{L^2})\notag\\
&\qquad+C_*\|\Delta\theta^{(\nu_1)}\|_{L^2}\|\nabla\phi\|^2_{L^2}.
\end{align}
By integrating \eqref{integral eqn for the difference in nu final} from 0 to $t$, using Gr\"{o}nwall's inequality and applying the bound \eqref{H1 estimate in compact set of H1} on $\theta^{(\nu_1)}$ and $\theta^{(\nu_2)}$, there exists a positive function $C_*(t)$ which depends only on $t$, $\kappa$, $\nu_*$, $\nu^*$, $\|S\|_{H^1}$ and $R_{\U}$ such that
\begin{align}\label{continuity of pin in nu}
\|(\theta^{(\nu_1)}-\theta^{(\nu_2)})(\cdot,t)\|^2_{H^1}=\|\phi(t)\|^2_{H^1}\le C_*(t)|\nu_1-\nu_2|^2,\qquad\forall t>0,
\end{align}
and \eqref{continuity of pin in nu} implies $\pin(t)\theta_0$ is continuous in $\I$ uniformly for $\theta_0$ in $\U$.
\end{proof}

In view of Remark~\ref{uniform bounded set in H1 remark} and Lemma~\ref{continuity in nu lemma}, the conditions L1 and L2 as stated at the beginning of this subsection follow immediately and we conclude that \eqref{upper semi-continuity at fixed nu>0} holds. We finish the proof of Theorem~\ref{varying nu theorem}.

% ------------------------------------------------------------------------

\subsection*{Acknowledgment} We would like to thank the reviewers for their valuable comments and suggestions which helped to improve the manuscript. S. Friedlander is supported by NSF DMS-1613135 and A. Suen is supported by Hong Kong General Research Fund (GRF) grant project number 18300720.

% ------------------------------------------------------------------------
\end{document}